\documentclass[11pt, onecolumn]{article}

\usepackage{constants}

\usepackage{stmaryrd}
\usepackage{tikz}
\usetikzlibrary{shapes.misc}
\tikzset{cross/.style={cross out, draw=black, minimum size=2*(#1-\pgflinewidth), inner sep=0pt, outer sep=0pt},
cross/.default={1pt}}

\usepackage{pstricks}
\usepackage[utf8]{inputenc}
\usepackage[T1]{fontenc}
\usepackage{lmodern}
\usepackage{xr}
\usepackage{bold-extra}
\usepackage{dsfont}

\usepackage{amsmath}
\usepackage{amssymb}
\usepackage{mathrsfs}
\usepackage{mathtools}

\usepackage{amsfonts}
\usepackage{amsthm}

\usepackage{enumerate}
\usepackage{multirow}
\usepackage{graphicx}
\usepackage{hyperref}
\usepackage{xcolor}
\hypersetup{
    colorlinks,
    linkcolor={black},
    citecolor={black},
    urlcolor={black}
}

\usepackage[top=2.5cm,bottom=2.5cm,left=2cm,right=2cm,marginparwidth=2.5cm]{geometry}
\reversemarginpar

\makeatletter
\renewcommand{\paragraph}{%
  \@startsection{paragraph}{4}%
  {\z@}{1.5ex \@plus 1ex \@minus .2ex}{-1em}%
  {\normalfont\normalsize\bfseries}%
}
\makeatother

{\theoremstyle{plain}
\newtheorem{Proposition}{\textbf{Proposition}}[section]

\newtheorem{Lemma}[Proposition]{\textbf{Lemma}}

 }

{\theoremstyle{plain}
\newtheorem{Theorem}[Proposition]{Theorem}}

{\theoremstyle{definition}
\newtheorem{Definition}[Proposition]{Definition}}

{\theoremstyle{remark}

\newtheorem{Remark}[Proposition]{Remark}

}

\numberwithin{equation}{section}

\newcommand{\eps}{\varepsilon}

\newcommand{\bbE}{\mathbb{E}}

\newcommand{\bbP}{\mathbb{P}}

\newcommand{\bbX}{\mathbb{X}}

\newcommand{\N}{\mathbb{N}}
\newcommand{\Z}{\mathbb{Z}}
\newcommand{\R}{\mathbb{R}}
\newcommand{\Q}{\mathbb{Q}}
\renewcommand{\S}{\mathbb S} 

\newcommand{\cE}{\mathcal{E}}

\newcommand{\cR}{\mathcal{R}}

\newcommand{\cX}{\mathcal{X}}

\let\limsup\relax
\let\liminf\relax
\DeclareMathOperator* \limsup {\overline{lim}}
\DeclareMathOperator* \liminf {\underline{lim}}

\newcommand{\inclim}[1]{\lim_{#1}\!\!\uparrow\!}

\let\originalleft\left
\let\originalright\right
\renewcommand{\left}{\mathopen{}\mathclose\bgroup\originalleft}
\renewcommand{\right}{\aftergroup\egroup\originalright}

\newcommand{\p}[1]{\left( #1 \right)}
\newcommand{\acc}[1]{\left\{ #1 \right\}}
\newcommand{\cro}[1]{\left[ #1 \right]}

\newcommand{\set}[2]{\acc{#1 \;\middle\vert\; #2 } }

\newcommand{\ind}[1]{\mathds{1}_{#1}}

\newcommand{\dpe}{\coloneqq}

\newcommand{\eol}{\nonumber\\}

\def\restriction#1#2{\mathchoice
              {\setbox1\hbox{${\displaystyle #1}_{\scriptstyle #2}$}
              \restrictionaux{#1}{#2}}
              {\setbox1\hbox{${\textstyle #1}_{\scriptstyle #2}$}
              \restrictionaux{#1}{#2}}
              {\setbox1\hbox{${\scriptstyle #1}_{\scriptscriptstyle #2}$}
              \restrictionaux{#1}{#2}}
              {\setbox1\hbox{${\scriptscriptstyle #1}_{\scriptscriptstyle #2}$}
              \restrictionaux{#1}{#2}}}
\def\restrictionaux#1#2{{#1\,\smash{\vrule height .8\ht1 depth .85\dp1}}_{\,#2}} 

\newcommand{\module}[2][]{\left\lvert #2 \right\rvert_{#1}}
\newcommand{\norme}[2][]{\left\| #2 \right\|_{#1}}
\newcommand{\normeUn}[1]{\left\| #1 \right\|_{1}}

\newcommand{\ceil}[1]{\!\left\lceil #1 \right\rceil\!}
\newcommand{\floor}[1]{\!\left\lfloor #1 \right\rfloor\!}

\newcommand{\ball}[2][]{\mathrm{B}_{#1}\p{#2}}
\newcommand{\clball}[2][]{\overline{\mathrm{B}}_{#1}\p{#2}}

\newcommand{\intervalle}[4]{#1#2\,,#3#4}

\newcommand{\intervalleff}[2]{\intervalle{\left[}{#1}{#2}{\right]}}
\newcommand{\intervalleof}[2]{\intervalle{\left(}{#1}{#2}{\right]}}
\newcommand{\intervallefo}[2]{\intervalle{\left[}{#1}{#2}{\right)}}
\newcommand{\intervalleoo}[2]{\intervalle{\left(}{#1}{#2}{\right)}}

\newcommand{\intint}[2]{\left\llbracket#1\,,#2\right\rrbracket}

\renewcommand{\d}{\mathrm{d}}

\DeclareMathOperator \petito {o}

\DeclareMathOperator \diam {diam}

\DeclareMathOperator \Pol {Pol}

\DeclareMathOperator \Pgrad {P-grad}

\newcommand{\base}[1]{\mathrm e_{#1}}

\makeatletter
\newcommand{\SymbolsFav}[1]{%
  \ensuremath{%
    \ifcase#1
    \or 
      *%
    \or 
      \dagger
    \or 
      \ddagger
    \or 
      **
    \or 
      \dagger\dagger
    \or 
      \ddagger \ddagger
    \or 
      \diamond 
    \fi
  }%
}   
\makeatother

\newcounter{fav}

\newcommand{\cFav}{\Fav^{\SymbolsFav{\thefav} }}

\newcommand{\Borel}[1]{\mathcal{B}\left( #1 \right) }

\newcommand{\hd}[1][1]{\mathcal H^{#1} }

\newcommand{\preHD}[1]{\phi_{#1}^1}

\newcommand{\E}[2][]{\mathbb{E}_{#1} \left[ #2\right]}

\newcommand{\Pb}[2][]{\mathbb{P}_{#1}\left( #2\right)}

\newcommand{\Path}[1]{\overset{#1}{\rightsquigarrow}}

\newcommand{\edges}[2][]{\mathrm E_{#1}\p{#2}}

\DeclareMathOperator{\Cylinder}{Corridor}

\newcommand{\BoundedF}{\mathcal F_{\mathrm b} }

\newcommand{\RegPD}[1][\TC]{\mathcal D_{#1} }

\newcommand{\UnifDistance}{\mathfrak d_\infty}

\newcommand{\BoundTC}{ \norme[1]\mu}
\newcommand{\BoundNorm}[1]{\norme[1] #1}
\newcommand{\HW}[2]{#1 \cro{ #2} }

\newcommand{\MD}[2][D]{\module[#1]{ \dot{#2} } }

\newcommand{\pc}{p_c(\Z^d)}

\newcommand{\PT}{\mathbf T}

\newcommand{\RestPT}[1]{\mathbf T_{#1} }

\newcommand{\BoxPT}[1][n]{{\mathbf T}_{ \intervalleff0{#1}^d } }
\newcommand{\SPT}[1][n]{{ \widehat{ \mathbf T}}_{ #1 } }

\newcommand{\TightBoxPT}[1][n]{{\mathbf T}_{ \intervalleff0{#1}^d }^{(b)} }
\newcommand{\TightSPT}[1][n]{{ \widetilde{ \mathbf T}}_{ #1 }^{(b)} }
\newcommand{\PathPT}[1]{\tau\p{#1}}

\newcommand{\TimeConstant}{\mu}
\newcommand{\TC}{\mu}

\newcommand{\FdTKesten}{J_{\mathrm K}}

\newcommand{\FdT}[1][]{J_{#1}}
\newcommand{\FdTsup}[1][]{\overline{J}_{#1}}
\newcommand{\FdTinf}[1][]{\underline{J}_{#1}}

\newcommand{\FdTmon}[1][]{J^-_{#1}}
\newcommand{\FdTmonsup}[1][]{\overline{J}^-_{#1}}
\newcommand{\FdTmoninf}[1][]{\underline{J}^-_{#1}}

\newcommand{\hatFdTpp}{\hat J_{\mathrm{pp}}}
\newcommand{\FdTpp}{J_{\mathrm{pp}} }

\newcommand{\LD}{\mathrm{LD}}
\DeclareMathOperator \Fav {Fav}
\DeclareMathOperator{\LongGeo}{LongGeo}
\DeclareMathOperator{\LDGreedy}{LDGreedy}

\newcommand{\X}{X}

\title{Large deviation principle at speed $n$ for the random metric in first-passage percolation}
\author{Julien \textsc{Verges}\footnote{julien.verges@univ-tours.fr} }
\date{\today}

\begin{document}

\maketitle

\abstract{Consider standard first-passage percolation on $\mathbb Z^d$. We study the lower-tail large deviations of the rescaled random metric $\widehat{\mathbf T}_n$ restricted to a box. If all exponential moments are finite, we prove that $\widehat{\mathbf T}_n$ follows the large deviation principle at speed $n$ with a rate function $J$, in a suitable space of metrics. Moreover, we give three expressions for $J(D)$. The first two involve the metric derivative with respect to $D$ of Lipschitz paths and the lower-tail rate function for the point-point passage time. The third is an integral against the $1$-dimensional Hausdorff measure of a local cost. Under a much weaker moment assumption, we give an estimate for the probability of events of the type $\acc{\widehat{\mathbf T}_n \le D}$. } 

\section{Introduction}
\label{sec : Intro}

\subsection{Framework}
\label{subsec : Intro/framework}

\subsubsection{First-passage percolation}
We first present the model of \emph{first-passage percolation} (FPP), introduced in 1966 by Hammersley and Welsh \cite{Ham65}. The reader interested in a summary of the achievements on this topic is invited to consult Auffinger, Damron and Hanson's survey \cite{50yFPP}. Let $d\ge2$ be an integer and $\bbE^d$ the set of all non-oriented nearest-neighbour edges in $\Z^d$. A finite sequence $\pi\dpe (x_0,\dots, x_r)$ of elements of $\Z^d$ is called a \emph{discrete path} if for all $i\in \intint0{r-1}$, $(x_i,x_{i+1}) \in \bbE^d$. We denote by $\norme\pi$ its number of edges.

Let $\nu$ be a probability distribution on $\intervallefo0\infty$ and consider a family $\p{\tau_e}_{e\in \bbE^d}$ of i.i.d.\ random variables with distribution $\nu$. In this paper $a$ will denote the infimum of its support. The variable $\tau_e$ is called the \emph{passage time along the edge} $e$. The \emph{passage time along a discrete path} $\pi=(x_0,\dots , x_r)$ is defined as
\begin{align}
    \label{eqn : Intro/framework/def_PathPT}
    \PathPT{\pi} &\dpe \sum_{i=0}^{r-1} \tau_{ (x_i, x_{i+1}) }.
    \intertext{For all $A\subseteq \R^d$ and $x,y\in \Z^d$, the \emph{passage time between $x$ and $y$ restricted in $A$} is defined as}
    \label{eqn : Intro/framework/def_PT}
    \RestPT{A}(x,y) &\dpe \inf_{ \substack{ x \Path{\pi} y \\ \pi \subseteq A}  } \PathPT{\pi},
\end{align}
where the infimum spans over discrete paths included in $A$, whose endpoints are $x$ and $y$. The map $\RestPT{A}(\cdot, \cdot)$ is a pseudometric on $A$. We call $\RestPT{A}$-\emph{discrete geodesic} between $x$ and $y$ any minimizer in~\eqref{eqn : Intro/framework/def_PT}. We will write $\PT\dpe \RestPT{\Z^d} = \RestPT{\R^d} $. A well-known result (\cite{50yFPP}, Equation~2.4) states that, under a moment condition on $\nu$, there exists a homogeneous function $\TimeConstant$ on $\R^d$, known as the \emph{time constant}, such that for all $x\in \Z^d$,
\begin{equation}    \label{eqn : Intro/framework/def_TimeConstant}
    \frac{\PT \p{ 0, nx } }{n} \xrightarrow[n\to\infty]{ \text{a.s.\ }} \TimeConstant(x).
\end{equation}
Furthermore, without any moment assumption on $\nu$, one can still define $\TC$, such that for all $x\in\Z^d$,
\begin{equation}
    \label{eqn : Intro/framework/def_TimeConstant_CerfTheret}
    \frac{\PT \p{ 0, nx } }{n} \xrightarrow[n\to\infty]{ \bbP } \TimeConstant(x).
\end{equation}
See e.g.\ (\cite{Cer16}, Theorem~4) for the formulation~\eqref{eqn : Intro/framework/def_TimeConstant_CerfTheret}. The time constant is a norm if
  \begin{equation}
        \label{ass : Intro/main_thm/SubcriticalAtom} \tag{SubC}
        \nu\p{\acc0} < \pc,
    \end{equation}
where $\pc$ is the critical parameter for bond percolation in $\Z^d$; otherwise $\mu(x) = 0$ for all $x\in \R^d$. We only study the former case, as the latter is trivial for our purposes. 

As a consequence of~\eqref{eqn : Intro/framework/def_TimeConstant_CerfTheret} the probability of an event of the form $\acc{\PT(0,n\base 1)\le \zeta n}$, with $\zeta < \TimeConstant(\base 1)$ or $\acc{\PT(0,n\base 1)\ge \zeta n}$, with $\zeta > \TimeConstant(\base 1)$ (the so-called lower-tail and upper-tail large deviation events) converges to $0$ as $n\to\infty$. In 1986, Kesten \cite{KestenStFlour} obtained estimates for the speed of convergence: there exists (\cite{KestenStFlour}, Theorem~5.2) a convex decreasing function $\FdTKesten:\intervalleoo{a}{\TimeConstant(\base{1})} \rightarrow \intervalleoo{0}{\infty}$ such that for all $a < \zeta < \TimeConstant(\base 1)$,
\begin{equation}
    \label{eqn : Intro/framework/monotone_rate_function_n}
    \lim_{n\to\infty} -\frac1n \log\Pb{ \PT\p{0, n\base1}\le \zeta n } = \FdTKesten(\zeta).
\end{equation}
Besides (\cite{KestenStFlour}, Theorem~5.9), under the assumption
\begin{equation}
    \label{ass : Intro/main_thm/exp_moment} \tag{Moment}
    \forall \lambda>0,\quad \int_{\intervallefo0\infty} \exp\p{\lambda t} \nu (\d t) < \infty,
\end{equation}
for all $\zeta > \TC(\base 1)$,
\begin{equation}
    \label{eqn : Intro/EstimateUpperTail}
    \lim_{n\to\infty} -\frac1n \log\Pb{ \PT(0,n\base 1) \ge \zeta n} = \infty.
\end{equation}
If $\nu$ has a bounded support, we have the stronger estimate
\begin{equation*}
    \liminf_{n\to\infty} -\frac1{n^d} \log\Pb{ \PT(0,n\base 1) \ge \zeta n} > 0.
\end{equation*}
Kesten's proof of~\eqref{eqn : Intro/EstimateUpperTail} may be adapted to any direction, i.e.\ under~\eqref{ass : Intro/main_thm/exp_moment}, for all $x \in \Z^d$ and $\zeta > \TC(x)$,
\begin{equation}
    \lim_{n\to\infty} -\frac1n \log\Pb{ \PT\p{0, nx } \ge \zeta n} = \infty.
\end{equation}
We prove a somewhat stronger version in Appendix~\ref{appsec : OdG}. 

\subsubsection{Aim of the paper}
Our first result, Theorem~\ref{thm : PointPoint}, is an extension of~\eqref{eqn : Intro/framework/monotone_rate_function_n} to any direction, for $\BoxPT$ and $\PT$. Consider
\begin{equation}
    \cX \dpe \set{(x,\zeta)\in\R^d \times \intervalleoo0\infty }{ \zeta >  a\norme[1] x  },
\end{equation}
where $\norme[1]\cdot$ is the usual $1$-norm, defined by~\eqref{eqn : Intro/DefNormeUn}, and $a$ is the infimum of $\nu$'s support, which may be zero. Define \begin{equation} \X \dpe \intervalleff01^d.\end{equation}
\begin{Theorem}
    \label{thm : PointPoint}
    There exists a function $\FdTpp : \cX \rightarrow \intervallefo0\infty$, such that for all $x,y\in \X$ and $\zeta > a \norme[1]{x-y}$,
    \begin{equation}
        \label{eqn : PointPoint/FdT}
        \lim_{n\to\infty}-\frac1n \log\Pb{ \vphantom{\BoxPT} \PT\p{\floor{nx}, \floor{ny}} \le n \zeta } =  \lim_{n\to\infty} -\frac1n \log\Pb{ \BoxPT\p{\floor{nx} ,\floor{ny} } \le n\zeta } = \FdTpp(x-y, \zeta),
    \end{equation}
    where $\floor\cdot$ denotes the componentwise floor function. Moreover:
    \begin{enumerate}[(i)]
        \item
        \label{item : PointPoint/General/ConvexAbsCont}
         $\FdTpp$ is convex and absolutely homogeneous.
        \item
        \label{item : PointPoint/General/Symmetry}
         For all $(x,\zeta)\in \cX$,
        \begin{equation}
            \FdTpp\p{x, \zeta} = \FdTpp\p{\module x,\zeta},
        \end{equation}
        where $\module x$ denotes the vector of $\R^d$ whose components are the modules of $x$'s components. 
        \item 
        \label{item : PointPoint/General/Croissance}
        For all $(x_1,\zeta_1), (x_2,\zeta_2)\in \cX$, if $0\le x_1 \le x_2$ for the componentwise order and $\zeta_1 \ge \zeta_2$, then
        \begin{equation}
            \FdTpp(x_1, \zeta_1) \le \FdTpp(x_2, \zeta_2).
        \end{equation}
        \item 
        \label{item : PointPoint/General/OrdreGrandeur}
        For all $(x,\zeta) \in \cX$, $\FdTpp(x,\zeta)=0$ if and only if $\zeta \ge \TC(x)$.
    \end{enumerate}
\end{Theorem}
For all $x\in\R^d$, we also define
\begin{equation}
    \FdTpp\p{x, a\norme[1] x} \dpe \inclim{\substack{\zeta \to a\norme[1] x \\ \zeta > a\norme[1] x} } \FdTpp(x,\zeta),
\end{equation}
where the notation $\uparrow$ is used to emphasize that $\zeta \mapsto  \FdTpp(x,\zeta)$ is nonincreasing. Note that the extension of $\FdTpp$ on $\overline\cX$ is also convex and lower semicontinuous. We will call $\FdTpp$ the \emph{elementary rate function}, as it will appear as an integrand in the expression of more sophisticated rate functions.
\begin{Remark}
    \label{rk : PointPoint/StrictDecroissance}
    By convexity and~\eqref{item : PointPoint/General/OrdreGrandeur}, for all $x \in \R^d \setminus \acc{0}$, the function $\zeta \mapsto \FdTpp\p{x, \zeta}$ is strictly decreasing on $\intervalleff{ a\norme[1] x}{ \TC(x)}$.
\end{Remark}

We are interested in extending~\eqref{eqn : PointPoint/FdT} for the random metric $\BoxPT$ rather than simply $\PT\p{\floor{nx}, \floor{ny} }$. We define, for all $n\ge1$,
\begin{align}
 \SPT : \X \times \X &\longrightarrow \intervallefo0\infty \eol
  (x,y) &\longmapsto \frac1n \BoxPT( \floor{nx} ,  \floor{ny} ). \label{eqn  : Intro/defSPT}
\end{align}
The function $\SPT$ is a pseudometric on $\X$. Our second result, Theorem~\ref{thm : Intro/FdTmon}, gives an estimate for the probability of $\SPT$ taking values below a target pseudometric $D$, of the form
\begin{equation*}
    \Pb{\SPT \lesssim D} \approx \exp\p{ -\FdTmon(D) n }.
\end{equation*}
Its main assumption is the existence of $\xi>0$ such that
\begin{equation}
    \label{ass : Intro/main_thm/ShapeThmStrong}\tag{StrongShape}
        \E{ \p{\min_{1\le i \le d}  \tau_i }^{d+\xi} }< \infty,
\end{equation}
where the $\tau_i$ are i.i.d.\ with distribution $\nu$.

Under Assumption~\eqref{ass : Intro/main_thm/exp_moment}, Theorem~\ref{thm : MAIN} goes further and states that $(\SPT)_{n\ge 1}$ follows a \emph{large deviation principle} at speed $n$, i.e.\ essentially an approximation of the form
\begin{equation*}
    \Pb{\SPT \simeq D} \approx \exp\p{ -\FdT(D) n}.
\end{equation*}

\subsection{Main theorems}
Sections~\ref{subsubsec : Intro/MainThm/LimitSpace} and~\ref{subsubsec : Intro/MainThm/LD} contain the minimal definitions for our main results, namely Theorems~\ref{thm : Intro/FdTmon} and~\ref{thm : MAIN} to make sense: respectively, the natural limit space for $\SPT$ and usual large deviation objects. Our main theorems are stated in Section~\ref{subsec : Intro/main_thm}.

\subsubsection{The limit space}
\label{subsubsec : Intro/MainThm/LimitSpace}
Almost surely, for all $n\ge 1$, $\SPT$ belongs to the space $\BoundedF$ of bounded functions on $\X ^2$. We endow $\BoundedF$ with the uniform distance, defined for all $D_1, D_2 \in \BoundedF$ as
\begin{equation}
    \label{eqn : Intro/main_thm/distance_uniforme}
        \UnifDistance(D_1, D_2) \dpe \max_{(x,y)\in \X^2}\module{D_1(x,y) - D_2(x,y)},
 \end{equation}
For all $D_1, D_2\in \BoundedF$ we denote by $D_1\le D_2$ the assertion 
    \begin{equation}
    \label{eqn : Intro/main_thm/partial_order}
        \forall x,y\in \X,\quad  D_1(x,y)\le D_2(x,y),
    \end{equation}
    which defines a partial order on $\BoundedF$.

\begin{Definition}
    Given a subset $A$ of $\R^d$, a pseudometric $D$ on $A$ and a continuous path $\gamma : \intervalleff0T \rightarrow A$ for the usual topology, we define the $D$-length of $\gamma$ as
    \begin{equation}
        \label{eqn : Intro/DefDLength}
        D(\gamma) \dpe \sup_{0=t_0 <\dots < t_r = T} \sum_{i=0}^{r-1} D\p{ \gamma(t_i), \gamma(t_{i+1})}.
    \end{equation}
    In the special case where $D$ is the metric induced by $\norme[1]\cdot$, we will denote by $\norme[1]\gamma$ the $D$-length of $\gamma$. We say that $\gamma$ is a $D$-\emph{geodesic} if it is Lipschitz for $\norme[1]\cdot$, and
    \begin{equation}
        D(\gamma) = D\p{\gamma(0), \gamma(T)}.     
    \end{equation}
    For all $t\in\intervalleoo0T$ such that the following limit exists, we define the \emph{metric derivative} of $\gamma$ at $t$, with respect to $D$ as
    \begin{equation}
        \label{eqn : Intro/def_MD}
        \MD{\gamma}(t) \dpe \lim_{h\to 0} \frac{1}{\module h}D\p{ \gamma(t) , \gamma(t+h)}.
    \end{equation}
\end{Definition}
Note that our definition of geodesic differ from the usual meaning, as we do not require that $\gamma$ is an isometry with respect to $D$, but some regularity with respect to $\norme[1]\cdot$. All the pseudometrics $D$ we will consider are dominated by $\norme[1]\cdot$, therefore their geodesics will also be Lipschitz for the metric $D$.

\begin{Definition}
\label{def : Intro/main_thm/adm_metrics}
    Let $L>0$ and $g$ be a norm on $\R^d$. We define $\RegPD[g]^L$ as the set of all pseudometrics $D$ on $\X$ such that
    \begin{enumerate}[(i)]
        \item For all $x,y\in \X$,
            \begin{equation}
            \label{eqn : Intro/main_thm/equivalence_distances}
                D(x,y) \le g\p{x-y}.
            \end{equation}
        \item For all $x,y \in \X$, there exists a $D$-geodesic $\sigma$ from $x$ to $y$, such that $\norme[1]\sigma \le L$.
    \end{enumerate}
    We define
    \begin{equation}
        \RegPD[g]\dpe \bigcup_{L\ge 0} \RegPD[g]^L.
    \end{equation}
\end{Definition}


\subsubsection{Large deviations}
\label{subsubsec : Intro/MainThm/LD}

We give here some general large deviation tools. See Dembo-Zeitouni (2009) \cite{LDTA} for the general theory.
\begin{Definition}
    Let $\bbX$ be a Hausdorff topological space. We call \emph{rate function} a lower semicontinuous map $I:\bbX\rightarrow \intervalleff0\infty$, i.e. a map whose sublevels $\set{x \in \bbX}{I(x)\le t}$, for $t\ge 0$, are closed. We further say that $I$ is a \emph{good rate function} if its sublevels are compact.

    We say that a random process $(X_n)_{n\ge 1}$ with values in $\bbX$ follows the \emph{large deviation principle} (LDP), at speed $a_n$, with the rate function $I$ if for every Borel set $A\subseteq \bbX$,
\begin{equation}
    \label{eqn : Intro/framework/LDP_general case}
    \inf_{x\in \overline{A}} I(x)%
    \le \liminf_{n\to\infty}-\frac{1}{a_n} \log \Pb{ X_n \in A } %
    \le \limsup_{n\to\infty}-\frac{1}{a_n} \log \Pb{ X_n \in A } %
    \le \inf_{x\in \mathring{A}} I(x).
\end{equation}  
\end{Definition}
In this article, except when stated otherwise, we will only consider LDPs at speed $n$. Lemma~\ref{lem : Intro/sketch/UB_LB} will be of constant use. It is a consequence of arguments given in the proof of Theorem~4.1.11 in \cite{LDTA}. The version here was used in (\cite{Ver24+_LDP_PPP_nd}, Lemma~1.2) with the speed $n^d$ instead of $n$. Apart from this difference, the proof may be copied verbatim. 
\begin{Lemma}
    \label{lem : Intro/sketch/UB_LB}
    Let $(\bbX, \d_\bbX)$ be a metric space and $(X_n)_{n\ge 1}$ a random process with values in $\bbX$. Define, for all $x\in \bbX$,
    \begin{align*}
        \overline{I}(x)&\dpe\inclim{\eps \to 0} \limsup_{n\to\infty} -\frac{1}{n}\log \Pb{\d_\bbX(x, X_n)\le \eps }
        \intertext{and}\underline{I}(x)&\dpe\inclim{\eps \to 0} \liminf_{n\to\infty} -\frac{1}{n}\log \Pb{\d_\bbX(x, X_n)\le \eps }.
    \end{align*}
    Then
    \begin{enumerate}[(i)]
        \item \label{item : Intro/sketch/UB_LB/rate_function}%
        $\overline I$ and $\underline I$ are rate functions on $\bbX$.
        \item For every open set $U\subseteq \bbX$,
            \begin{equation}
                \label{eqn : Intro/sketch/UB_LB/UB}
                \limsup_{n\to\infty}-\frac{1}{n}\log\Pb{X_n \in U} \le \inf_{x\in U} \overline{I}(x).
            \end{equation}
        \item \label{item : Intro/sketch/UB_LB/LB} For every compact set $K\subseteq \bbX$,
            \begin{equation}
                \label{eqn : Intro/sketch/UB_LB/LB}
                \liminf_{n\to\infty}-\frac{1}{n}\log\Pb{X_n \in K} \ge \min_{x\in K} \underline{I}(x).
            \end{equation}
    \end{enumerate}
\end{Lemma}
If $I\dpe\overline I = \underline I$, we say that $(X_n)_{n\ge1}$ satisfies the \emph{weak LDP} at speed $n$, with the rate function $I$. 

\subsubsection{Result statements}
\label{subsec : Intro/main_thm}

For all $D\in \RegPD$, $n\in \N^*$ and $\eps>0$, we introduce the event
\begin{equation}
    \label{eqn : Intro/def_LDn-}
    \LD_{n}^-(D, \eps)\dpe \acc{\forall x,y\in X, \quad \SPT\p{x,y} \le D(x,y) + \eps }.
\end{equation}
We then define
\begin{align}
    \label{eqn : Intro/main_thm/def_FdTmonsup}
    \FdTmonsup(D) &\dpe \inclim{\eps \to 0} \limsup_{n\to \infty} -\frac1n \log \Pb{ \LD_{n}^-(D, \eps) }
    \intertext{and }
    \label{eqn : Intro/main_thm/def_FdTmoninf}
    \FdTmoninf(D) &\dpe \inclim{\eps \to 0} \liminf_{n\to \infty} -\frac1n \log \Pb{ \LD_{n}^-(D, \eps) }.
\end{align}
Those functions are both rate functions. Theorem~\ref{thm : Intro/FdTmon} states that $\FdTmon(D) \dpe \FdTmonsup(D) = \FdTmoninf(D)$ and provides several expressions of $\FdTmon(D)$, which involve:
\begin{itemize}
    \item The elementary rate function $\FdTpp$, defined in Theorem~\ref{thm : PointPoint}.
    \item A \emph{highway network} of $D$, which is essentially a family of disjoint $D$-geodesics dense enough so that any two points $x,y\in \X$ are linked by a path inside the network with $D$-length $D(x,y)$ (see Definition~\ref{def : Topology/RegPD/HW_Network}).
    \item The \emph{gradient by paths} $(\Pgrad D)_z(u)$ of $D$ at $z\in \X$, in the direction $u\in \R^d$, which describes the $D$-length of infinitesimal paths originating from $D$ with speed $u$ (see Definition~\ref{def : Topology/GMT/Gradient}).
\end{itemize}
In Equation~\eqref{eqn : FdTmon/Intrinsic}, $\hd$ denotes the usual $1$-dimensional Hausdorff measure on $\R^d$, whose definition is recalled in Section~\ref{subsec : Intro/Notations}, and $\S_2$ denotes the Euclidean unit sphere.
\begin{Theorem}
    \label{thm : Intro/FdTmon}
    Assume~\eqref{ass : Intro/main_thm/SubcriticalAtom} and~\eqref{ass : Intro/main_thm/ShapeThmStrong}. For all $D\in \RegPD$,
    \begin{equation}
        \FdTmon (D) \dpe \FdTmoninf(D) = \FdTmonsup(D).
    \end{equation}
    If $\p{\sigma_k : \intervalleff0{L}\rightarrow \X }_{k\ge1}$ is a highway network for $D$, then
    \begin{equation}
        \FdTmon (D)
        \label{eqn : FdTmon/Existence/ExpressionGeodesics}
            = \sum_{k\ge 1} \int_0^{L} \FdTpp\p{\sigma_k'(t), \MD{\sigma_k}(t)}\d t. \end{equation}
    Moreover,
    \begin{equation}
        \label{eqn : FdTmon/Existence/ExpressionSup}
        \begin{split} \FdTmon(D)
            &= \sup \set{ \sum_{k= 1}^K \int_0^{T_k} \FdTpp\p{\gamma_k'(t), \MD{\gamma_k}(t)}\d t }%
            {\begin{array}{c} K\in\N, \p{\gamma_k : \intervalleff{0}{T_k} \rightarrow \X}_{1\le k\le K} \text{ $1$-Lipschitz}, \\ \text{injective and pairwise disjoint.} \end{array}  }\\
            &= \sup \set{ \sum_{k= 1}^\infty \int_0^{T_k} \FdTpp\p{\gamma_k'(t), \MD{\gamma_k}(t)}\d t }%
            {\begin{array}{c} \p{\gamma_k : \intervalleff{0}{T_k} \rightarrow \X}_{k\ge 1} \text{ $1$-Lipschitz}, \\ \text{injective and pairwise disjoint.} \end{array}  }, \end{split} \end{equation}
        and
    \begin{equation}
        \label{eqn : FdTmon/Intrinsic}
        \FdTmon(D)%
            =\int_\X \max_{u\in \S_2} \FdTpp\p{u, (\Pgrad D)_z(u) } \hd\p{\d z}.
    \end{equation}
    Besides, for all distinct $D_1, D_2 \in \RegPD$, if $D_1 \le D_2$ and $\FdTmon(D_2) < \infty$, 
    \begin{equation}
        \FdTmon(D_1) > \FdTmon(D_2).
    \end{equation}
\end{Theorem}
\begin{Remark}
    Equation~\eqref{eqn : FdTmon/Intrinsic} holds with the unit sphere with respect to any norm instead of $\S_2$, provided one also takes the Hausdorff measure associated with the chosen norm. However, we use the Euclidean norm because it is the usual framework. 
\end{Remark}

\begin{Remark}
    Contrary to~\eqref{eqn : FdTmon/Existence/ExpressionGeodesics},~\eqref{eqn : FdTmon/Intrinsic} does not depend on an arbitrary choice of highway network for $D$. Actually, we will prove that if $\FdTmon(D)<\infty$, then for $\hd$-almost all $z$, for all $u\in \S_2$ except maybe in one direction, $(\Pgrad D)_z(u) = \TC(u)$ (see Lemma~\ref{lem : FdTmon/Intrinsic/Pgrad}). In particular, by Item in Theorem~\ref{thm : PointPoint},\eqref{item : PointPoint/General/OrdreGrandeur}, for $\hd$-almost all $z$, the maximum in~\eqref{eqn : FdTmon/Intrinsic} is either $0$ or the only positive value of $ \FdTpp\p{u, (\Pgrad D)_z(u) }$.
\end{Remark}

Under stronger assumptions, the LDP holds with a rate function that coincides with $\FdTmon$ on $\RegPD$, and is infinite outside.
\begin{Theorem}
    \label{thm : MAIN}
    Assume~\eqref{ass : Intro/main_thm/SubcriticalAtom} and~\eqref{ass : Intro/main_thm/exp_moment}. The process $\p{\SPT}_{n\ge 1}$ follows the large deviation principle at speed $n$ with the good rate function
    \begin{align}
        \notag \FdT : \BoundedF &\longrightarrow \intervalleff0\infty \\
        D &\longmapsto%
        \begin{cases}
            \FdTmon(D) &\text{ if } D \in \RegPD[\TC],\\
            \infty &\text{ otherwise.}
        \end{cases}
    \end{align}
    In other words, for all $A\in \Borel{\BoundedF}$,
    \begin{equation}
    \label{eqn : MAIN/LDP}
        \inf_{D\in \overline{A}} \FdT(D) \le %
        \liminf_{n\to\infty}-\frac{1}{n}\log\Pb{\SPT \in A} \le %
        \limsup_{n\to\infty}-\frac{1}{n}\log\Pb{\SPT \in A} \le %
        \inf_{D\in \mathring{A}} \FdT(D).
   \end{equation}
\end{Theorem}

\subsection{Open questions and related works}

\paragraph{Upper-tail large deviations for the point-point and face-face times.}
Contrary to the lower-tail, the order of the upper-tail large deviation probability $\Pb{\PT(0, n\base 1) \ge n\zeta}$, with $\zeta > \base 1$, depends on the distribution. Kesten (\cite{KestenStFlour}, Theorem~5.9) proved that if $\nu$ has a bounded support then it is of order $\exp\p{-n^d}$. Basu, Ganguly and Sly \cite{Bas21} later proved the existence of the rate function, in dimension $2$,\footnote{Their proof may be adapted in any dimension, though.} under a regularity assumption on $\nu$. This regularity assumption was recently relaxed by the author (see Corollary~1.6 in \cite{Ver24+_LDP_PPP_nd}).

Concerning distributions with unbounded support, Cranston, Gauthier and Mountford gave (\cite{Cra09}, Theorem~1.3) a criterion for $\Pb{\PT(0, n\base 1) \ge n\zeta}$ to be of order $\exp\p{-n^d}$, provided the tail of $\nu$ has a specific shape. In a recent article, Cosco and Nakajima \cite{Cos23} proved that for all $0 < r < d$, if the tail of $\nu$ decays like $\exp\p{-t^r}$, then $\Pb{\PT(0, n\base 1) \ge n\zeta}$ has order $\exp\p{-n^r}$, and they gave an expression of the point-point rate function.

It is worth noting that the so-called face-face passage time (i.e.\ the minimal passage time among paths traversing the box $\intint0n^d$ from one face to the opposite one) seems to display less diverse orders in its upper-tail large deviations. Indeed Chow and Zhang \cite{Cho03} showed that if $\nu$ has one exponential moment, then the probability that the face-face passage time takes abnormally large values is of order $\exp\p{-n^d}$. They also proved the existence of rate function.

\paragraph{Large deviations at the metric level.}
Assuming that $\nu$ is supported by $\intervalleff ab$, with $0<a<b<\infty$, the author \cite{Ver24+_LDP_PPP_nd} proved a LDP at speed $n^d$ for $\SPT$, with a good rate function $I$. An integral expression similar to~\eqref{eqn : FdTmon/Intrinsic} is provided, except the measure is the Lebesgue measure on $\X$. For all possible adherence values $D$ of $\SPT$, $I(D) = 0$ if and only if $D \le \TC$, and is finite except on marginal cases, meaning that result provides an appropriate estimate for the probability of $\{\SPT \simeq D \}$, whenever $D$ is not bounded by $\TC$. On the other hand, whenever $D\le \TC$ and $D\neq \TC$, $\FdT(D)>0$ and is the limit of metrics $D_n$ such that $\FdT(D_n)<\infty$, meaning that Theorem~\ref{thm : MAIN} provides an appropriate estimate for the probability of $\{\SPT \simeq D \}$, whenever $D\le \TC$. Consequently, the LDPs at speed $n$ and $n^d$ give a full picture of the large deviations for $(\SPT)_{n\ge 1}$, in the sense that for all positive sequence $(a_n)_{n\ge 1}$ satisfying
\begin{equation}
    a_n \ll n, \text{ or } n \ll a_n \ll n^d, \text{ or } n^d \ll a_n,
\end{equation}
for all pseudometrics $D$ on $\X$, either
\begin{align}
    \lim_{\eps \to 0} \liminf_{n\to\infty} -\frac1{a_n} \log \Pb{\LD_n(D,\eps)} &= \infty,
    \intertext{or}
    \lim_{\eps \to 0} \limsup_{n\to\infty} -\frac1{a_n} \log \Pb{\LD_n(D,\eps)} &= 0.
\end{align}
In other words, there cannot be a third speed $a_n$, for which $(\SPT)_{n\ge 1}$ satisfies the LDP with a rate function taking at least one positive, finite value.

However, if $\nu(\intervallefo t\infty)$ decays like $\exp\p{-t^r}$, with $1<r<d$ then at least three LDPs, at three different speeds, are required to describe the large deviations of $(\SPT)_{n\ge 1}$. Indeed, in this regime the probability of deviation events studied in the present article have order $\exp\p{-n}$, while as stated above, the upper-tail deviations of the point-point passage time have order $\exp\p{-n^r}$ and the upper-tail deviations of the face-face passage time have order $\exp\p{-n^d}$. It is thus plausible that there exist at least three speeds for which $(\SPT)_{n\ge 1}$ follows a LDP with a nontrivial rate function. The case where $\nu(\intervallefo t\infty)$ decays like $\exp\p{-t}$ may also be quite rich, because lower-tail behaviours and some upper-tail behaviours are both of order $\exp\p{-n}$. In particular the conclusion of Theorem~\ref{thm : MAIN} fails in this context, because if $(\SPT)_{n\ge 1}$ followed a LDP at speed $n$, the associated rate function would take positive, finite values for some metric greater than $\TC$.

\paragraph{Large deviations for the chemical distance in bond percolation.}

In supercritical bond percolation, we call \emph{chemical distance} between two vertices the length of the shortest open path between these vertices. In the framework of FPP, this corresponds to the case where $\nu$ is supported by $\acc{1, \infty}$. Garet and Marchand \cite{Gar07} showed that the probability of large deviations events for the chemical distance has order $\exp\p{-n}$. For the upper-tail part, Dembin and Nakajima \cite{Dem23+} recently proved the existence of the associated rate function.

\subsection{Outline of the proofs}
\label{subsec : Intro/OotP}

In Section~\ref{sec : Topology}, we give some topological preliminaries about $\RegPD[g]$. We show that the sets $\RegPD[g]^L$ are compact subsets of $\BoundedF$, any pseudometric $D\in \RegPD$ has a highway network and the $D$-length of Lipschitz paths is given by an integral involving the gradient by paths. The highway method is a refinement of an argument used by the author in~\cite{Ver24+_LDP_PPP_nd}.

In Section~\ref{sec : PointPoint}, we prove Theorem~\ref{thm : PointPoint}. By the FKG inequality, the probability that two translations of the event in~\eqref{eqn : PointPoint/FdT} occur simultaneously is greater than the product of the individual probabilities. This allows for a subadditive argument, providing the existence and convexity of the elementary rate function.

In Section~\ref{sec : FdTmon}, we prove Theorem~\ref{thm : Intro/FdTmon}. We first upper bound $\FdTmonsup(D)$ by the right-hand side of~\eqref{eqn : FdTmon/Existence/ExpressionGeodesics}. To do so, we note that prescribing a passage time smaller than $D$ between a large number of milestones scattered along the highways essentially implies prescribing a passage time smaller than $D$ everywhere. Thanks to the FKG inequality, the probability of this scenario is lower bounded by the product of the individual point-point lower-tail deviation events involved. Applying Theorem~\ref{thm : PointPoint} and letting the number of milestones tend to infinity yields the desired bound. We then lower bound $\FdTmoninf(D)$ by the right-hand side of~\eqref{eqn : FdTmon/Existence/ExpressionSup}, using a somehow opposite strategy. Rather than providing an appropriate scenario for the large deviation event $\LD_n^-(D,\eps)$ to occur, we show that for all families $(\gamma_k)$ as in~\eqref{eqn : FdTmon/Existence/ExpressionSup}, on the large deviation event, there exists a family of abnormally fast \emph{pairwise disjoint} discrete paths that "follow" the paths $\gamma_k$. Finally, we prove~\eqref{eqn : FdTmon/Intrinsic} by the so-called area formula (see e.g.\ Corollary~5.1.13 in \cite{Kra08}). The hard part is to show that $(\Pgrad D)$ has $\hd$-almost everywhere a simple expression with respect to the speeds of the highways $\sigma_k$.

In Section~\ref{sec : LDP}, we prove Theorem~\ref{thm : MAIN}. For all $D\in \BoundedF$, $n\in \N^*$ and $\eps>0$, we introduce the event
\begin{equation}
    \label{eqn : Intro/def_LDn}
    \LD_{n}(D, \eps)\dpe \acc{\UnifDistance\p{\SPT , D}\le \eps },
\end{equation}
and the rate functions
\begin{align}
    \label{eqn : Intro/def_FdTsup}
    \FdTsup(D) &\dpe \inclim{\eps \to 0} \limsup_{n\to \infty} -\frac1n \log \Pb{ \LD_{n}(D, \eps) },\\
    \label{eqn : Intro/def_FdTinf}
    \FdTinf(D) &\dpe \inclim{\eps \to 0} \liminf_{n\to \infty} -\frac1n \log \Pb{ \LD_{n}(D, \eps) },
\end{align}
i.e.\ the functions $\overline{I}$ and $\underline{I}$ involved in Lemma~\ref{lem : Intro/sketch/UB_LB}, in the special case $\bbX = \BoundedF$ and $X_n = \SPT$. We first show that if the metric $\SPT^{(b)}$ associated with the truncated passage times $\tau_e \wedge b$ is exponentially equivalent to another pseudometric $\TightSPT$, which has better tightness properties. We then show that $\TightSPT$, thus $\SPT^{(b)}$, follows the LDP with the rate function $\FdT$, which essentially amounts to showing that $\FdTinf = \FdTsup = \FdT$. In the case where $\nu$ has an unbounded support, we use the fact that the truncated passage times are exponentially good approximations of $\SPT$.

\subsection{Notations and conventions}
\label{subsec : Intro/Notations}

\paragraph{Vectors of $\R^d$ and norms on $\R^d$.}
Let $(e_i)_{1\le i \le n}$ denote the canonical basis of $\R^d$. We endow $\R^d$ with the norms defined by
\begin{equation}
    \label{eqn : Intro/DefNormeUn}
    \norme[1]{x} \dpe \sum_{i=1}^d \module{x_i}, \text{ and } \norme[2]x \dpe \p{\sum_{i=1}^d x_i^2 }^{1/2}, 
\end{equation}
for all $x=(x_1,\dots,x_d) \in \R^d$. We define $\d$ as the metric associated with $\norme[1]\cdot$, $\S_1$ the unit sphere for $\norme[1]\cdot$ and $\ball[1]{z,r}$ (resp. $\clball[1]{z,r}$) the open (resp. closed) ball of center $z$ and radius $r$ for $\norme[1]\cdot$. Likewise, we denote by $\S_2$, $\ball[2]{z,r}$ and $\clball[2]{z,r}$ their analogues for $\norme[2]\cdot$. Given a norm $g$ on $\R^d$, we define
\begin{equation}
    \label{eqn : Intro/DefNormeNorme}
    \BoundNorm g \dpe \sup_{u \in \S_1} g(u) = \sup_{x\in \R^d \setminus\acc0} \frac{g(x)}{\norme[1] x}.
\end{equation}
For all $\lambda > 0$, we say that a function with values in $\R^d$ is $\lambda$-Lipschitz if it is $\lambda$-Lipschitz for the norm $\norme[1]\cdot$.

For all $x\in\R^d$, we denote by $\module x$ (resp.\ $\floor x$) the element of $\R^d$ whose components are the modules (resp.\ the floors) of the components of $x$. We denote by $\le$ the componentwise order on $\R^d$.

\paragraph{Edges and paths.}
We will identify any discrete path $\pi = (x_0, \dots, x_r)$ to the continuous path defined as the piecewise affine function $\pi : \intervalleff0r \rightarrow \R^d$ such that for all $i\in\intint0r$, $\pi(i)= x_i$. The number of edges $\norme\pi$ of $\pi$, seen as a discrete path, is equal to $\norme[1]\pi$ and $\norme[2]\pi$ as defined by~\eqref{eqn : Intro/DefDLength}, so there is no collision of notations. Given a continuous path $\gamma:\intervalleff0T \rightarrow \R^d$, and $x,y\in \R^d$, we will write $x\Path\gamma y$ if $\gamma(0)=x$ and $\gamma(T)=y$. For all subsets $A$ of $\Z^d$, we denote by $\edges{A}$ the set of edges whose endpoints are both in $A$. We say that a point $x\in \R^d$ belongs to an edge $e\in \bbE^d$ if it belongs to the segment with the same endpoints as $e$.

Whenever a continuous path $\gamma : \intervalleff0T \rightarrow \R^d$ appears inside an expression involving set operators, we identify it with its image. For example, "$x\in \gamma$" will mean "$x\in \gamma\p{\intervalleff0T}$".

\paragraph{Hausdorff measure.}
Consider a Borel subset $A$ of $\R^d$. For all $R>0$, we define
\begin{equation}
    \label{eqn : Intro/Def_preHD}
    \preHD{R} (A) \dpe \inf \set{ \sum_{i\ge 1} \diam(A_i) }{ \acc{A_i}_{i\ge 1} \text{ subsets of } \R^d \text{ s.t.\ } A \subseteq \bigcup_{i\ge 1}A_i\text{ and } \forall i\ge1, \; \diam(A_i) \le R },
\end{equation}
where $\diam(A_i)$ denotes the diameter of $A_i$ for the Euclidean norm. The $1$-\emph{dimensional Hausdorff measure of }A is defined as (see e.g.\ Section~2.1 in~\cite{Kra08})
\begin{equation}
    \hd(A) \dpe \inclim{R \to 0} \preHD{R}(A).
\end{equation}
For more details on the Hausdorff measure, see Krantz-Parks (2008)~\cite{Kra08}.

\paragraph{Miscellaneous}
We denote by $\# A$ the cardinal of a set $A$. For all integers $n \le m$, we define $\intint nm \dpe \intervalleff nm \cap \Z$. The infimum of $\nu$'s support is denoted by $a$.

\section{Topological preliminaries}
\label{sec : Topology}

\subsection{Pseudometrics with bounded-length geodesics}
\label{subsec : Topology/RegPD}
In this section we fix a norm $g$ and study the properties of the space $\RegPD[g]$. Up to changing the parametrization, any Lipschitz path on $\R^d$ may be assumed to be $1$-Lipschitz. Unless stated otherwise, for all $L>0$ and $D\in \RegPD[g]^L$, our $D$-geodesics are $1$-Lipschitz paths $\sigma : \intervalleff0L \rightarrow \X$.
\subsubsection{Compactness}
\begin{Proposition}
	\label{prop : Topology/RegPD/Compact}
	For all $L>0$, the space $\RegPD[g]^L$ is compact for the topology of the uniform convergence.
\end{Proposition}
\begin{proof}
Let $L>0$. For all $D\in \RegPD[g]^L$, $x,x',y,y'\in \X$, by triangle inequality,
\begin{equation*}
	\module{D(x,y) - D(x',y')} \le D(x,x') + D(y,y') \le g\p{x-x'} + g\p{y-y'}.
\end{equation*}
In particular, $\RegPD[g]^L$ is equicontinuous.

By Arzelà-Ascoli theorem it is thus sufficient to show that it is closed in $\BoundedF$. Let $(D_n)_{n\ge1}$ be a sequence of elements of $\RegPD[g]^L$ converging to $D \in \BoundedF$. Then $D$ is clearly a pseudometric dominated by $g$. Let $x,y\in \X$. Let us show that $D$ has a geodesic from $x$ to $y$ with $\norme[1]\cdot$-length at most $L$.For all $n\ge1$, there exists a $D_n$-geodesic $\sigma_n : \intervalleff{0}{L}\rightarrow \X$ from $x$ to $y$. Since all the $\sigma_n$ are $1$-Lipschitz, the family $(\sigma_n)_{n\ge 1}$ is equicontinuous, hence there exists an extraction $\varphi : \N\rightarrow \N$ and $1$-Lipschitz path $\sigma : \intervalleff0L \rightarrow \X$ such that $\sigma_{\varphi(n)}$ uniformly converges to $\sigma$ as $n\to\infty$. Consider a subdivision $0=t_0 \le t_1 \le \dots \le t_K = L$. For all $n\ge 1$, by definition of $\sigma_{\varphi(n)}$,
\begin{align}
	\sum_{k=0}^{K-1} D_{\varphi(n)} \p{ \sigma_{\varphi(n)}(t_k), \sigma_{\varphi(n)}(t_{k+1}) } %
		&= D_{\varphi(n)} \p{x,y}.\notag
	\intertext{Letting $n\to \infty$ yields, by uniform convergence,}
	\sum_{k=0}^{K-1} D \p{ \sigma(t_k), \sigma(t_{k+1}) } %
		&= D \p{x,y}.\notag
	\intertext{Taking the supremum with respect to the subdivision, we get}
	D\p{\sigma} %
		&= D \p{x,y}.
\end{align}
Besides, $\sigma(0)=x$ and $\sigma(L)=y$, thus $\sigma$ is a $D$-geodesic from $x$ to $y$. Consequently, $D\in \RegPD[g]^L$.
\end{proof}
\subsubsection{The highway method}
The main result of this paragraph is Proposition~\ref{prop : Topology/RegPD/HW_cvg}, with states that a pseudometric $D\in \RegPD[\TC]$ may be approximated by considering the uniform space $g$ then "activating" a large number of geodesics for $D$. Moreover, those geodesics may be chosen injective and pairwise disjoint. Some of the work, gathered in Lemma~\ref{lem : Topology/RegPD/HW_simple_ptys}, may easily be adapted from Lemma~3.3 in \cite{Ver24+_LDP_PPP_nd}. 
\begin{Definition}
	\label{def : Topology/RegPD/HW}
	Let $L>0$ and $D\in \RegPD[\TC]^L$. Consider a sequence $\p{ \sigma_k : \intervalleff0{L} \rightarrow \X }_{k\ge 1}$ of $D$-geodesics. We recursively define the sequence $\p{\HW{D}{ (\sigma_k)_{1\le k \le K} } }_{K\ge 0}$ of functions in $\BoundedF$ as such : $\HW{D}{\emptyset}= \TC$ and for all $K\ge 0$ and $x,y\in\X$,
	\begin{equation}
		\label{eqn : Topology/RegPD/HW}
	\begin{split}
		\HW{D}{ (\sigma_k)_{1\le k \le K +1} }(x,y) \dpe &\HW{D}{ (\sigma_k)_{1\le k \le K} }(x,y) \wedge  \min_{0 \le s,t \le L } \Bigl( \HW{D}{ (\sigma_k)_{1\le k \le K} }(x,\sigma_{K+1}(s) )\\
		 &\quad + D(\sigma_{K+1}(s) , \sigma_{K+1}(t) ) + \HW{D}{ (\sigma_k)_{1\le k \le K} }( \sigma_{K+1}(t) , y) \Bigr). 
	\end{split}
	\end{equation}
\end{Definition}
\begin{Lemma}
	\label{lem : Topology/RegPD/HW_simple_ptys}
	We adopt the same notations as in Definition~\ref{def : Topology/RegPD/HW}. 
	\begin{enumerate}[(i)]
		\item For all $K\ge 0$, $\HW{D}{ (\sigma_k)_{1\le k \le K} } \in \RegPD[\TC]$.
		\item For all $K\ge 0$,
		\begin{equation}
			\label{eqn : Topology/RegPD/HW_simple_ptys/noninc}
			\HW{D}{ (\sigma_k)_{1\le k \le K} } \ge \HW{D}{ (\sigma_k)_{1\le k \le K +1} } \ge D.
		\end{equation}
		\item If the sequence $\bigl( ( \sigma_k(0) , \sigma_k(L) ) \bigr)_{k\ge 1}$ is dense in $\X^2$ for the usual topology, then
		\begin{equation}
			\label{eqn : Topology/RegPD/HW_simple_ptys/Convergence}
			D= \lim_{K\to \infty} \HW{D}{ (\sigma_k)_{1\le k \le K} }.
		\end{equation}
	\end{enumerate}
\end{Lemma}
\begin{Definition}
	\label{def : Topology/RegPD/HW_Network}
	Let $L>0$ and $D\in \RegPD[\TC]^L$. We say that a sequence $\p{\sigma_k : \intervalleff0L \rightarrow \X}_{k\ge 1}$ of $D$-geodesics is a \emph{highway network} for $D$ if they are injective, pairwise disjoint and~\eqref{eqn : Topology/RegPD/HW_simple_ptys/Convergence} holds.
\end{Definition}

\begin{Proposition}
	\label{prop : Topology/RegPD/HW_cvg}
	Let $L>0$ and $D\in \RegPD[\TC]^L$. Then $D$ has a highway network.	Moreover, any injective $D$-geodesic $\sigma$ may be chosen as the first highway in the network.
\end{Proposition}
The general idea is to consider a family of geodesics $(\hat \sigma_k)_{k\ge 1}$ whose endpoints form a dense subset of $\X$, then get rid of the interesctions in order to create pairwise disjoint paths. The second part of the proposition will be a clear consequence of the proof. We rely on Lemma~\ref{lem : Topology/RegPD/HW_cvg/KeyLemma} to cut paths.
\begin{Lemma}
	\label{lem : Topology/RegPD/HW_cvg/KeyLemma}
	Let $\gamma : \intervalleff0T \rightarrow \X$ be a continuous path and $X_0\subseteq \X$ be a compact set. There exists a countable family $\p{\intervalleff{s_p}{t_p}}_{p\ge 1}$ of pairwise disjoint segments of $\intervalleff0T$ such that
	\begin{enumerate}[(i)]
		\item For all $p\ge 1$, $\gamma\p{\intervalleff{s_p}{t_p}} \cap X_0 = \emptyset$.
		\item For almost every $t\in \intervalleff0T \setminus \p{ \bigcup_{p\ge1} \intervalleff{s_p}{t_p} }^\mathrm{c}$, $\gamma(t)\in X_0$.
 	\end{enumerate}
\end{Lemma}
\begin{proof}[Proof of Lemma~\ref{lem : Topology/RegPD/HW_cvg/KeyLemma}]
	Since $\gamma$ is continuous, $\gamma^{-1}\p{\X\setminus X_0}$ is an open subset of $\intervalleff0T$, thus a countable reunion of disjoint open intervals. Each of those can be covered almost everywhere by a countable reunion of disjoint segments.
\end{proof}
\begin{proof}[Proof of Proposition~\ref{prop : Topology/RegPD/HW_cvg}]
	Let $\p{(x_k, y_k)}_{k\ge 1}$ be a dense sequence in $\X^2$. For all $k\ge 1$, consider a $D$-geodesic $\hat \sigma_k : \intervalleff0{L}\rightarrow \X$ from $x_k$ to $y_k$. Up to removing loops and changing the parametrization, we may assume that they are injective. By Lemma~\ref{lem : Topology/RegPD/HW_simple_ptys},
	\begin{equation}
		\label{eqn : Topology/RegPD/HW_cvg/input_cvg}
		D = \lim_{K\to\infty} \HW{D}{ (\hat \sigma_k)_{1\le k \le K} }.
	\end{equation}
	Let $k\ge1$. Applying Lemma~\ref{lem : Topology/RegPD/HW_cvg/KeyLemma} to $\gamma = \hat \sigma_k$ and $X_0 = \bigcup_{i=1}^{k-1} \hat \sigma_i$, we know that there exists a family of pairwise disjoint segments $\p{\intervalleff{s_{k,p}}{t_{k,p} } }_{p\ge 1}$ of $\intervalleff{0}{L}$ such that
	\begin{enumerate}[(i)]
		\item For all $p\ge1$, $\hat \sigma_k\p{ \intervalleff{s_{k,p}}{t_{k,p} } } \cap \p{\bigcup_{i=1}^{k-1} \hat \sigma_i } = \emptyset$.
		\item For almost every $t\in \intervalleff{0}{L} \setminus \p{\bigcup_{p\ge 1} \intervalleff{s_{k,p} }{t_{k,p} } }^\mathrm{c}$, $\hat \sigma_k(t) \in \bigcup_{i=1}^{k-1} \hat \sigma_i$.
	\end{enumerate}
	We denote by $\sigma_{k,p} : \intervalleff0L \rightarrow \X$ a $1$-Lipschitz and injective reparametrization of $\restriction{\sigma_k}{\intervalleff{s_{k,p}}{t_{k,p} }}$; note that it is a $D$-geodesic. Let $\p{\sigma_r}_{r\ge 1}$ be an enumeration of $\p{\sigma_{k,p}}_{k,p \ge 1}$. It is sufficient to show
	\begin{equation}
		\label{eqn : Topology/RegPD/HW_cvg/WinCon}
		D = \lim_{R\to\infty}\HW{D}{ ( \sigma_r)_{1\le r \le R} }.
	\end{equation}
	Let $\eps>0$. By~\eqref{eqn : Topology/RegPD/HW_cvg/input_cvg} there exists $K>0$ such that
	\begin{equation}
		\label{eqn : Topology/RegPD/HW_cvg/bigK}
		\UnifDistance\p{D , \HW{D}{ (\hat \sigma_k)_{1\le k \le K} }  } \le \eps.
	\end{equation}
	Besides, by definition of the $\intervalleff{s_{k,p}}{t_{k,p}}$ and Lipschitz continuity, there exists $P>0$ such that
	\begin{equation}
		\label{eqn : Topology/RegPD/HW_cvg/controle_reste}
		\sum_{k=1}^K \sum_{p=P+1}^\infty \norme[1]{\sigma_{k,p}} \le \eps.
	\end{equation}
	Let $R>0$ be an integer large enough so that $\acc{\sigma_{k,p}}_{\substack{1\le k \le K \\ 1\le p \le P} } \subseteq \acc{\sigma_r}_{1\le r \le R}$, and $x,y \in X$. Any $\HW{D}{ (\hat \sigma_k)_{1\le k \le K} }$-geodesic between $x$ and $y$ is concatenation of a finite number of straight lines and subpaths of $\hat \sigma_k$, for $1\le k \le K$. Thanks to~\eqref{eqn : Topology/RegPD/HW_cvg/controle_reste}, the former may be covered by subpaths of the $\sigma_{k,p}$, for $1\le k \le K$ and $1\le p \le P$, except along a $\norme[1]\cdot$-length smaller than $\eps$. Consequently, by triangle inequality,
	\begin{equation}
		\label{eqn : Topology/RegPD/HW_cvg/bigI}
		\HW{D}{ ( \sigma_r)_{1\le r \le R} }(x,y) \le \HW{D}{ (\hat \sigma_k)_{1\le k \le K} }(x,y) + \BoundTC \eps.
	\end{equation}
	Applying~\eqref{lem : Topology/RegPD/HW_simple_ptys},~\eqref{eqn : Topology/RegPD/HW_cvg/bigK} and~\eqref{eqn : Topology/RegPD/HW_cvg/bigI} and, we deduce that for large enough $R$, for all $x,y\in \X$,
	\begin{equation}
		D(x,y) \le \HW{D}{ ( \sigma_r)_{1\le r \le R} }(x,y) \le D(x,y) + (\BoundTC +1)\eps,
	\end{equation} 
	hence~\eqref{eqn : Topology/RegPD/HW_cvg/WinCon}.
\end{proof}


\subsection{Integration along a Lipschitz path}
\label{subsec : Topology/GMT}
Integrals along a Lipschitz path will appear naturally when we need to compute its length or cost (see e.g.~\eqref{eqn : FdTmon/Existence/ExpressionSup}). We gather here the geometric measure theory tools we need to handle these objects, namely the \emph{metric derivative} of a path, already defined by~\eqref{eqn : Intro/def_MD} and a special case of the so-called \emph{area formula} (see Lemma~\ref{lem : Topology/GMT/AreaFormula}). We also define the \emph{gradient by paths} of a metric, which makes the link between the metric derivative of a path with respect to this metric, and its derivative in the usual sense (see Lemma~\ref{lem  : Topology/GMT/Gradient}). We rely on the monographs by Ambrosio and Tilli \cite{Amb04}, and Krantz and Parks \cite{Kra08}.
\begin{Lemma}
	\label{lem : Topology/GMT/Rademacher}
	Let $\gamma : \intervalleff0T \rightarrow \X$ be a Lipschitz path.
	\begin{enumerate}[(i)]
		\item For almost every $t\in\intervalleff0T$, $\gamma$ is differentiable at $t$, and for all $0\le t_1 \le t_2 \le T$,
		\begin{equation}
			\label{eqn : Topology/GMT/Rademacher/Classic}
			\gamma(t_2) - \gamma(t_1) = \int_{t_1}^{t_2} \gamma'(t)\d t.
		\end{equation}
		\item For all $D\in \RegPD[\TC]$, for almost every $t\in \intervalleff0T$, the limit
		\begin{equation}
			\label{eqn : Topology/GMT/Rademacher/MetricD}
			\MD{\gamma}(t) \dpe \lim_{h\to 0} \frac{ D\p{\gamma(t) , \gamma(t+h)} }{\module h}
		\end{equation}
		exists, and for all $0\le t_1 \le t_2 \le T$,
		\begin{equation}
			\label{eqn : Topology/GMT/Rademacher/VariationTotale}
			D\p{\restriction{\gamma}{\intervalleff{t_1}{t_2}}} = \int_{t_1}^{t_2} \MD{\gamma}(t)\d t.
		\end{equation}
	\end{enumerate}
\end{Lemma}
\begin{proof}
	The first part is a consequence of the fact that Lipschitz functions are absolutely continuous, thus satisfy the fundamental theorem of calculus (see e.g.\ Theorem~7.18 in \cite{Rudin}). The second part is an adaptation to pseudometrics of Theorem~4.1.6 in \cite{Amb04}.
\end{proof}
\begin{Definition}
	\label{def : Topology/GMT/Gradient}
	Let $D \in \RegPD[\TC]$ and $z\in \X$. The \emph{gradient by paths} of $D$ at $z$ is defined as the function
	\begin{align}
		(\Pgrad D)_z : \R^d &\longrightarrow \intervallefo0\infty \eol
		u&\longmapsto \inf\set%
			{\liminf_{t\to0} \frac{D\p{\restriction{\gamma}{\intervalleff0t} }}{t} }%
			{\begin{array}{c} \gamma : \intervalleff0T \rightarrow \X\text{ Lipschitz}, \gamma(0)=z, \\ \gamma \text{ is differentiable at }t=0\text{ and }\gamma'(0)=u \end{array} }
	\end{align}
\end{Definition}
\begin{Lemma}
	\label{lem  : Topology/GMT/Gradient}
	Let $D \in \RegPD[\TC]$, and $\gamma: \intervalleff0T \rightarrow \X$ be a Lipschitz path.
	\begin{enumerate}[(i)]
	\item For almost every $t\in\intervalleff0T$,
	\begin{equation}
		\label{eqn : Topology/GMT/Gradient/Equality}
		\MD\gamma(t) = (\Pgrad D)_{\gamma(t)}\p{\gamma'(t)}.
		\end{equation}
	\item For all $0\le t_1 \le t_2 \le T$,
	\begin{equation}
		\label{eqn : Topology/GMT/Gradient/Integral}
		D\p{ \restriction\gamma{\intervalleff{t_1}{t_2} } } =  \int_{t_1}^{t_2} (\Pgrad D)_{\gamma(t)}\p{\gamma'(t)} \d t.
	\end{equation}
	\end{enumerate}
\end{Lemma}
\begin{proof}
	By~\eqref{eqn : Topology/GMT/Rademacher/VariationTotale} and Lebesgue's differentiation theorem (see e.g.\ Theorem~7.10 in \cite{Rudin}), for almost all $t\in \intervalleoo0T$,
	\begin{equation}
		\label{eqn : Topology/GMT/Gradient/AlmostAll}
		\MD{\gamma}(t) = \lim_{h\to 0^+} \frac{D\p{\gamma(t), \gamma(t+h)}}{h} =%
		\lim_{h\to 0^+} \frac{D\p{ \restriction{\gamma}{\intervalleff{t}{t+h}} } }{h},
	\end{equation}
	and $\gamma$ is differentiable at $t$. Fix such $t$. By definition of the gradient by paths,
	\begin{equation}
		\label{eqn : Topology/GMT/Gradient/LowerB}
		\MD{\gamma}(t) \ge (\Pgrad D)_{\gamma(t)}\p{\gamma'(t)}.
	\end{equation}
	Besides, let $ \gamma_1 : \intervalleff0{T_1} \rightarrow \X$ be a Lipschitz path such that $\gamma_1(0) = \gamma(t)$ and $\gamma_1'(0) = \gamma'(t)$. By the triangle inequality,
	\begin{align}
		\module{D\p{\gamma(t) , \gamma(t+h)} - D\p{\gamma_1(0) , \gamma_1(h)} } %
			&= \module{D\p{z , \gamma(t+h)} - D\p{z , \gamma_1(h)} } \eol
			&\le D\p{\gamma(t+h), \gamma_1(h)}  \eol
			&\le \TC\p{\gamma(t+h) - \gamma_1(h)} = \petito(h). \nonumber
	\end{align}
	Consequently, by~\eqref{eqn : Topology/GMT/Gradient/AlmostAll},
	\begin{align}
		\MD{\gamma}(t) &= \lim_{h\to 0^+} \frac{D\p{\gamma_1(0), \gamma_1(h) } }{h} \eol
			&\le \liminf_{h \to 0^+} \frac{D\p{\restriction{\gamma_1}{\intervalleff0h} } }{h}.\nonumber
		\intertext{Taking the infimum over all paths $\gamma_1$, we get}
		\label{eqn : Topology/GMT/Gradient/UpperB}
		\MD{\gamma}(t) &\le (\Pgrad D)_{\gamma(t)}\p{\gamma'(t)}.
	\end{align}
	Inequalities~\eqref{eqn : Topology/GMT/Gradient/LowerB} and~\eqref{eqn : Topology/GMT/Gradient/UpperB} give the first part of the lemma. The second part is a consequence of the first one and~\eqref{eqn : Topology/GMT/Rademacher/VariationTotale}. 
\end{proof}
\begin{Definition}
	\label{def : Topology/GMT/RegularPoint}
	Given a Lipschitz, injective path $\gamma$ and $z = \gamma(t) \in \gamma$, we say that $z$ is a \emph{regular point} of $\gamma$ if the derivative $\gamma'(t)$ exists and is nonzero. 
\end{Definition}
\begin{Lemma}
	\label{lem : Topology/GMT/AreaFormula}
	Let $\gamma : \intervalleff0T \rightarrow \X$ be a Lipschitz, injective path. Then
	\begin{enumerate}[(i)]
		\item $\hd$-almost every point of $\gamma$ is regular.
		\item For all measurable function $\Phi : \X \times \R^d  \rightarrow \R^+$ such that for all $z \in \X$, $\Phi(z, \cdot)$ is absolutely homogeneous,
			\begin{equation}
				\label{eqn : Topology/GMT/AreaFormula}
				\int_0^T \Phi\p{\gamma(t), \gamma'(t) }\d t%
			=\int_\gamma \Phi\p{ z , \frac{\gamma'\p{\gamma^{-1}(z) } }{ \norme[2]{\gamma'\p{\gamma^{-1}(z)} } } } \hd(\d z).
	\end{equation}
	\end{enumerate}
\end{Lemma}
\begin{proof}
	The first item is a consequence of the so-called area formula, in the version stated by Theorem~5.1.1 in \cite{Kra08}, with $M=1$, $N=d$, $f=\gamma$ and $A$ the preimage by $\gamma$ of the set of non-regular points of $\gamma$.

	To prove the second one, note that
	\begin{align*}
		\int_0^T \Phi\p{\gamma(t), \gamma'(t) }\d t%
			&= \int_0^T \Phi\p{\gamma(t), \gamma'(t) } \ind{ \gamma(t) \text{ is regular}}\d t \\
			&= \int_0^T \Phi\p{\gamma(t), \frac{\gamma'(t)}{\norme[2]{\gamma'(t)} } } \norme[2]{\gamma'(t)} \ind{ \gamma(t) \text{ is regular}}\d t.
		\intertext{Another version of the area formula, Corollary~5.1.13 in \cite{Kra08}, applied for $M=1$, $N=d$, $f = \gamma$ and $g(t) = \Phi\p{\gamma(t), \frac{\gamma'(t)}{ \norme[2]{\gamma'(t)} } } \ind{ \gamma(t) \text{ is regular}}$, gives}
		\int_0^T \Phi\p{\gamma(t), \gamma'(t) }\d t%
			&= \int_\gamma \Phi\p{z,\frac{\gamma'\p{\gamma^{-1}(z) } }{\norme[2]{\gamma'\p{\gamma^{-1}(z) } } }  }\ind{ z \text{ is regular}}  \hd(\d z)\\
			&= \int_\gamma \Phi\p{z,\frac{\gamma'\p{\gamma^{-1}(z) } }{\norme[2]{\gamma'\p{\gamma^{-1}(z) } } }  } \hd(\d z).
	\end{align*}
\end{proof}


\section{Elementary rate function}
\label{sec : PointPoint}

In this section we prove Theorem~\ref{thm : PointPoint}. Our general strategy follows a classic approach we first define $\FdTpp(x,\zeta)$ in the case where $x\in \Z^d$ with a classic subadditive argument (see Lemma~\ref{lem : PointPoint/Entier}), then extend to all $x\in \Q^d$ by homogeneity and to all $x\in \R^d$ by monotonicity (see Lemma~\ref{lem : PointPoint/Extension}). Equation~\eqref{eqn : PointPoint/FdT} follows by stationarity. The characterization of the case $\FdTpp(x,\zeta)=0$ is covered by Lemma~\ref{lem : PointPoint/OrdreGrandeur}.
\begin{Lemma}
	\label{lem : PointPoint/Entier}
	For all $x\in \Z^d$ and $\zeta > a \norme[1] x$, the limit
	\begin{equation}
		\label{eqn : PointPoint/Entier/Limite}
		\FdTpp(x,\zeta) \dpe \lim_{n\to\infty} - \frac1n \log \Pb{\PT(0,nx)\le n\zeta} = \inf_{n\ge 1}- \frac1n \log \Pb{\PT(0,nx)\le n\zeta}
	\end{equation}
	exists, and it is finite. Moreover,
	\begin{enumerate}[(i)]
		\item For all $x\in \Z^d$, $\zeta>a\norme[1] x$ and $k\in \N^*$,
		\begin{equation}
			\label{eqn : PointPoint/Entier/Homogeneity}
			\FdTpp(kx,k\zeta) =  k\FdTpp(x,\zeta).
		\end{equation}
		\item For all $x_1, x_2 \in \Z^d$, $\zeta_1 > a \norme[1]{x_1}$ and $\zeta_2 > a\norme[1]{x_2}$,
		\begin{equation}
			\label{eqn : PointPoint/Entier/Subadditivity}
			\FdTpp(x_1 + x_2, \zeta_1 + \zeta_2) \le \FdTpp(x_1, \zeta_1) + \FdTpp(x_2, \zeta_2).
		\end{equation}
		\item For all $x\in \Z^d$ and $\zeta>a\norme[1] x$,
		\begin{equation}
			\label{eqn : PointPoint/Entier/Symmetry}
			\FdTpp\p{x, \zeta} = \FdTpp\p{ \module x,\zeta}.
		\end{equation}
		\item For all $x_1, x_2 \in \Z^d$, $\zeta_1 > a \norme[1]{x_1}$ and $\zeta_2 > a\norme[1]{x_2}$, if $0\le x_1 \le x_2$ and $\zeta_1 \ge \zeta_2$ then
		\begin{equation}
			\label{eqn : PointPoint/Entier/Croissance}
			\FdTpp\p{x_1,\zeta_1} \le \FdTpp\p{x_2, \zeta_2}.
		\end{equation}
	\end{enumerate}
\end{Lemma}
For all $x\in \Q^d$ and $\zeta>a \norme[1] x$, we define
\begin{equation}
	\label{eqn : PointPoint/Rationnel}
	\FdTpp(x,\zeta) \dpe \frac1k \FdTpp(kx, k\zeta),
\end{equation}
where $k\in \N^*$ is any integer such that $kx\in \Z^d$; the choice does not matter thanks to~\eqref{eqn : PointPoint/Entier/Homogeneity}. Moreover, for all $x\in \Q^d$, $\zeta > a \norme[1] x$ and $\lambda \in \Q \cap \intervalleoo0\infty$,
\begin{equation}
	\label{eqn : PointPoint/Rationnel/Homegeneity}
	\FdTpp(\lambda x, \lambda \zeta)= \lambda \FdTpp(x, \zeta).
\end{equation}
\begin{Lemma}
	\label{lem : PointPoint/Extension}
	The function $\FdTpp$ admits a continuous extension on $\cX$ that satisfies items~\eqref{item : PointPoint/General/ConvexAbsCont},~\eqref{item : PointPoint/General/Symmetry} and~\eqref{item : PointPoint/General/Croissance} in Theorem~\ref{thm : PointPoint}.
\end{Lemma}
\begin{Lemma}
	\label{lem : PointPoint/OrdreGrandeur}
	For all $(x,\zeta)\in \cX$,  $\FdTpp(x,\zeta)>0$  if and only if $\zeta < \TC(x)$.
\end{Lemma}
For all $u,v\in \Z^d$, $\zeta >0$ and $\ell >0$, we consider the event
	\begin{equation}
		\cE(u,v,t, \ell) \dpe \acc{\text{There exists a discrete path $u\Path\pi v$ with $\norme\pi \le \ell$ and $\tau(\pi) \le t$ }}. 
	\end{equation}
We first prove Theorem~\ref{thm : PointPoint}, assuming Lemmas~\ref{lem : PointPoint/Entier},~\ref{lem : PointPoint/Extension} and~\ref{lem : PointPoint/OrdreGrandeur} are true. We will use several times the following straightforward fact: for all $n\in \N^*$, for all $u,v\in \intint0n^d$ and $t > a$,
	\begin{equation}
		\label{eqn : PointPoint/BorneGrossiere}
		\Pb{ \BoxPT(u,v) \le t \norme[1]{u-v} } \ge \nu\p{ \intervalleff{a}{t}}^{\norme[1]{u-v}}.
	\end{equation}
\begin{proof}[Proof of Theorem~\ref{thm : PointPoint}]
	We first show that for all distinct $x,y\in \intervalleoo01^d$, for small enough $\eps>0$, there exists $z\in \clball[1]{y-x, \eps\norme[1]{y-x}  }$ such that
	\begin{equation}
		\label{eqn : PointPoint/WinCon1}
		\limsup_{n\to \infty} -\frac1n \log\Pb{\BoxPT\p{\floor{nx} ,\floor{ny} } \le n\zeta}%
			\le \FdTpp\p{z, (1-2\eps)\zeta } +2\eps - 2\eps d\log \nu\p{\intervalleff a{ \frac{\zeta}{\norme[1]{y-x} } } }
	\end{equation}
	and
	\begin{equation}
		\label{eqn : PointPoint/WinCon2}
		\liminf_{n\to\infty} -\frac1n \log\Pb{\vphantom{\BoxPT} \PT\p{ \floor{nx}, \floor{ny} } \le n \zeta } %
			\ge \FdTpp\p{z, (1+3\eps)\zeta } + 2\eps d\log \nu\p{\intervalleff a{ \frac{\zeta}{\norme[1]{y-x} } }}.
	\end{equation}
	Let $x,y \in \intervalleoo01^d$ be distinct,
	\begin{equation}
		\delta \dpe \d\p{ \acc{x,y}, \partial \X } >0 \text{ and }0 < \eps < \frac{\zeta - a \norme[1]{y-x}}{2\zeta +a\norme[1]{y-x}} \wedge \frac\delta3.
	\end{equation}
	Let $z=z(x,y,\eps)\in \Q^d$ be such that
	\begin{equation}
		\label{eqn : PointPoint/Approx_z}
		\norme[1]{y-x-z} \le \eps\norme[1]{y-x}.
	\end{equation}
	Note that $(1-2\eps)\zeta > a(1+\eps)\norme[1]{y-x} \ge a\norme[1] z$. By definition of $\FdTpp\p{z, (1-2\eps)\zeta}$ (see~\eqref{eqn : PointPoint/Entier/Limite} and~\eqref{eqn : PointPoint/Rationnel}), there exists $m\in \N^*$ such that $mz \in \Z^d$, and
	\begin{equation*}
		-\frac1m \log\Pb{ \PT\p{0, mz}\le m(1-2\eps) \zeta } \le \FdTpp\p{z, (1-2\eps)\zeta} + \eps.
	\end{equation*}
	By monotone convergence, there exists $\ell>0$ such that
	\begin{equation}
		\label{eqn : PointPoint/Pb_cE}
		-\frac1m \log\Pb{\cE\p{ 0,mz,m(1-2\eps)\zeta, \ell } } \le \FdTpp\p{z, (1-2\eps)\zeta} + 2\eps.
	\end{equation}
	Let $n \in \N^*$ be such that
	\begin{equation}\label{eqn : PointPoint/def_n} \frac{ m\norme[1]{y-x} +d} n < \eps \norme[1]{y-x} \text{ and } \frac\ell n < \eps.\end{equation} 
	Let 
	\begin{equation}\label{eqn : PointPoint/def_K} K \dpe \floor{n/m}. \end{equation} We have
	\begin{align}
		\norme[1]{ \vphantom{\big\vert} \floor{nx} + Kmz - \floor{ny} }%
			&\le d + \norme[1]{nx + Kmz - ny} \eol
			&\le d + \norme[1]{Kmx + Kmz - Kmy} + (n-Km)\norme[1]{y-x}\eol
			&\le d + Km\norme[1]{x+z-y} + m\norme[1]{y-x}.\notag
		\intertext{Consequently, by~\eqref{eqn : PointPoint/Approx_z} and~\eqref{eqn : PointPoint/def_n},}
		\label{eqn : PointPoint/Approx_ny_Kmz}
		\norme[1]{\vphantom{\big\vert} \floor{nx} + Kmz - \floor{ny} }%
			&\le 2\eps n \norme[1]{y-x}.
	\end{align}
	Consider the event
	\begin{equation}
		\label{eqn : PointPoint/def_Fav}
		\cFav \dpe \p{ \bigcap_{k=0}^{K-1} \cE\p{ \floor{nx} + kmz, \floor{nx} + (k+1)mz, m(1-2\eps)\zeta, \ell } }  \cap \acc{\BoxPT\p{\floor{nx} + Kmz, \floor{ny} }\le 2\zeta \eps n  } .
	\end{equation}
	Since $\cFav$ is an intersection of decreasing events, by the FKG inequality,~\eqref{eqn : PointPoint/BorneGrossiere} and~\eqref{eqn : PointPoint/Approx_ny_Kmz},
	\begin{align*}
		\Pb{\cFav}%
			&\ge \p{ \prod_{k=0}^{K-1} \Pb{ \cE\p{ \floor{nx} + kmz, \floor{nx} + (k+1)mz, m(1-2\eps)\zeta, \ell } } } \cdot \nu\p{\intervalleff a{\frac{ \zeta}{\norme[1]{y-x}}} }^{2\eps n \norme[1]{y-x}}. \notag
		\intertext{By stationarity of the model,}
		\Pb{\cFav}%
			&\ge \Pb{ \cE\p{ 0, mz, m(1-2\eps)\zeta, \ell } }^K  \cdot \nu\p{\intervalleff a{\frac{ \zeta}{\norme[1]{y-x}}} }^{2\eps n \norme[1]{y-x}}.
	\end{align*}
	Applying~\eqref{eqn : PointPoint/Pb_cE} gives
	\begin{align} 
		-\frac1n \log\Pb{\cFav} &\le \frac{Km}{n} \cdot\cro{ \FdTpp\p{z, (1-2\eps)\zeta} + 2 \eps } - 2\eps \norme[1]{y-x} \log \nu\p{\intervalleff a{\frac{ \zeta}{\norme[1]{y-x}}}},\notag
		\intertext{thus}
		\label{eqn : PointPoint/Pb_Fav}
		-\frac1n \log\Pb{\cFav} &\le \FdTpp\p{z, (1-2\eps)\zeta} + 2 \eps - 2\eps d \log \nu\p{\intervalleff a{\frac{ \zeta}{\norme[1]{y-x}}}}.
	\end{align}
	Assume that $\cFav$ occurs. By the inequalities $\frac \ell n \le \eps \le \delta/3$ and~\eqref{eqn : PointPoint/Approx_ny_Kmz}, for all $k\in \intint0{K-1}$,
	\begin{equation*}
		 \clball[1]{ \floor{nx} + kmz , \ell} \subseteq \intervalleff0n^d.
	\end{equation*}
	Consequently, for all $k\in \intint0{K-1}$,
	\begin{equation*}
		\BoxPT\p{  \floor{nx} + kmz,  \floor{nx} + (k+1)mz } \le m(1-2\eps)\zeta.
	\end{equation*}
	By triangle inequality,
	\begin{align*}
		\BoxPT\p{\floor{nx}, \floor{ny}}%
			&\le \p{ \sum_{k=0}^{K-1} \BoxPT\p{ \floor{nx} + kmz ,  \floor{nx} + (k+1)mz} } + \BoxPT\p{\floor{nx} + Kmz, \floor{ny} } \eol
			&\le Km(1-2\eps)\zeta + 2\eps n \zeta \le n\zeta,
	\end{align*}
	thus
	\begin{equation}
		\label{eqn : PointPoint/Inclusion_Fav}
		\cFav \subseteq \acc{ \BoxPT\p{ \floor{nx} , \floor{ny} } \le n\zeta}.
	\end{equation}
	Combining~\eqref{eqn : PointPoint/Pb_Fav} and~\eqref{eqn : PointPoint/Inclusion_Fav} leads to
	\begin{equation*}
		-\frac1n \log \Pb{\BoxPT\p{ \floor{nx} , \floor{ny} } \le n\zeta} \le \FdTpp\p{z, (1-2\eps)\zeta} + 2 \eps - 2\eps d\log \nu\p{\intervalleff a{\frac\zeta{\norme[1]{x-y} } } }.
	\end{equation*}
	Taking the superior limit as $n\to\infty$ gives~\eqref{eqn : PointPoint/WinCon1}.

	We now turn to the proof of the lower bound~\eqref{eqn : PointPoint/WinCon2}. Let $m \in \N^*$ be such that $mz \in \Z^d$. Fix $n$ and $K$ as in~\eqref{eqn : PointPoint/def_n} and~\eqref{eqn : PointPoint/def_K}. Consider the event
	\stepcounter{fav}
	\begin{equation}
		\cFav \dpe \acc{ \PT\p{\floor{nx}, \floor{ny}} \le n \zeta } \cap \acc{\BoxPT\p{\floor{nx} + Kmz, \floor{ny} }\le 2\zeta \eps n  }.
	\end{equation}
	By the FKG inequality and~\eqref{eqn : PointPoint/BorneGrossiere},
	\begin{equation}
		-\frac1n \log\Pb{\cFav} \le -\frac1n \log \Pb{ \PT\p{\floor{nx}, \floor{ny}} \le n \zeta } - 2\eps d \log \nu\p{\intervalleff a{\frac\zeta{\norme[1]{x-y} }} }.
	\end{equation}
	Besides, the triangle inequality gives
	\begin{equation}
		\cFav\subseteq \acc{ \PT\p{\floor{nx}, \floor{nx} + Km z} \le n(\zeta+2\eps) } \subseteq \acc{\PT\p{\floor{nx}, \floor{nx} + Km z} \le Km(\zeta+3\eps) },
	\end{equation}
	for large enough $n$. Consequently, by stationarity, for large enough $n$,
	\begin{align*}
		&-\frac1n \log \Pb{ \PT\p{\floor{nx}, \floor{ny}} \le n \zeta } - 2\eps d \log \nu\p{\intervalleff a{\frac\zeta{\norme[1]{x-y} }} } \\
			&\quad\ge -\frac 1n\log\Pb{\PT\p{\floor{nx}, \floor{nx} + Km z} \le Km(\zeta+3\eps) } \\
			&\quad = - \frac{Km}{n}\frac{1}{Km}\log\Pb{\PT\p{0,  Km z} \le Km(\zeta+3\eps) }.
	\end{align*}
	Taking the inferior limit as $n\to\infty$ gives~\eqref{eqn : PointPoint/WinCon2}.

	We now prove~\eqref{eqn : PointPoint/FdT}. If $x=y$, this is clear. If $x,y \in \intervalleoo01^d$ and are distinct, letting $\eps\to 0$ in~\eqref{eqn : PointPoint/WinCon1} and~\eqref{eqn : PointPoint/WinCon2} gives, by continuity of $\FdTpp$ on $\cX$, the desired result (recall that $z$ depends on $\eps$). The remaining case is when $x,y\in \intervalleff01^d$, are distinct and $\zeta > a \norme[1]{x-y}$. Let $0 < \eps < \frac12 \p{1 - \frac{a\norme[1]{x-y} }{\zeta} }$. There exist distinct $\hat x , \hat y \in \intervalleoo01^d$ such that
	\begin{equation*}
		(1- 2\eps)\zeta > a\norme[1]{\hat x - \hat y},
	\end{equation*}
	and for large enough $n$,
	\begin{equation*}
		\norme[1]{\vphantom{\big\vert} \floor{nx}-\floor{n\hat x} } \le n\eps  \norme[1]{x-y}, \quad \norme[1]{\vphantom{\big\vert} \floor{ny}-\floor{n\hat y} } \le n\eps  \norme[1]{x-y}.
	\end{equation*}
	The triangle inequality gives for large enough $n$ the inclusions
	\begin{align}
	\begin{split}
		&\acc{\BoxPT\p{\floor{n \hat x} , \floor{n \hat y} } \le n (1-2\eps)\zeta } \cap \acc{ \BoxPT\p{\floor{n \hat x} , \floor{n x}}\le n \zeta \eps  } \\ &\quad \cap  \acc{ \BoxPT\p{\floor{n \hat y} , \floor{n y}}\le n \zeta \eps  } \subseteq \acc{ \BoxPT\p{\floor{nx}, \floor{ny} }\le n\zeta }
	\end{split}
		\intertext{and}
	\begin{split}
		&\acc{\PT\p{\floor{n  x} , \floor{n  y} } \le n \zeta } \cap \acc{ \BoxPT\p{\floor{n \hat x} , \floor{n x}}\le n \zeta \eps  } \\ &\quad\cap  \acc{ \BoxPT\p{\floor{n \hat y} , \floor{n y}}\le n \zeta \eps  } \subseteq \acc{ \PT\p{\floor{n \hat x}, \floor{n \hat y} }\le n(1+2\eps)\zeta }.
	\end{split}
	\end{align}
	By the FKG inequality,~\eqref{eqn : PointPoint/BorneGrossiere} and~\eqref{eqn : PointPoint/FdT} for $\hat x$ and $\hat y$, we have
	\begin{align}
		\FdTpp\p{\hat x - \hat y, (1-2\eps)\zeta} - 2\eps \log \nu\p{\intervalleff{a}{ \frac{\zeta}{ \norme[1]{x-y} } } } &\ge \limsup_{n\to\infty} -\frac1n \log\Pb{ \BoxPT\p{\floor{nx}, \floor{ny} }\le n\zeta }
		\intertext{and}
		\FdTpp\p{ \hat x - \hat y, (1+2\eps)\zeta } + 2\eps \log \nu\p{\intervalleff{a}{ \frac{\zeta}{ \norme[1]{x-y} } } } &\le \liminf_{n\to\infty} -\frac1n \log \Pb{ \PT\p{\floor{n x}, \floor{n y} }\le n(1+2\eps)\zeta }.
	\end{align}
	By continuity of $\FdTpp$ on $\cX$, letting $\eps\to 0$ gives~\eqref{eqn : PointPoint/FdT} in full generality.
\end{proof}
\begin{proof}[Proof of Lemma~\ref{lem : PointPoint/Entier}]
	By the FKG inequality and stationarity of the model, for all $x_1, x_2 \in \Z^d$ and $t_1, t_2 \ge 0$,
	\begin{equation*}
		\Pb{\PT(0,x_1) \le t_1, \; \PT(x_1,x_1+x_2) \le t_2} \ge \Pb{\PT(0,x_1) \le t_1} \cdot \Pb{\PT(0,x_2) \le t_2}.
	\end{equation*}
	In particular, by triangle inequality,
	\begin{equation}
		\label{eqn : PointPoint/Entier/FKG}
		\Pb{\PT(0,x_1+ x_2) \le t_1 + t_2} \ge \Pb{\PT(0,x_1) \le t_1} \cdot \Pb{\PT(0,x_2) \le t_2}.
	\end{equation}

	Let $x\in \Z^d$ and $\zeta > a\norme[1]{x}$. For all $n,m \in \N^*$,~\eqref{eqn : PointPoint/Entier/FKG} with $x_1 = nx$, $x_2 = mx$, $t_1 = n\zeta$ and $t_2 = m\zeta$ implies
	\begin{equation}
	 	\Pb{\PT\p{0,(n+m)x} \le (n+m)\zeta}%
	 		\ge \Pb{\PT(0,nx) \le n\zeta} \cdot \Pb{\PT(0,mx) \le m\zeta}.\nonumber
	\end{equation}
	Fekete's lemma gives the existence of the limit in~\eqref{eqn : PointPoint/Entier/Limite}. The finiteness of the limit is a consequence of~\eqref{eqn : PointPoint/BorneGrossiere}.

	Let $x\in \Z^d$, $\zeta > a\norme[1] x$ and $k\in \N^*$. By~\eqref{eqn : PointPoint/Entier/Limite}, considering the extraction $n\mapsto kn$ yields~\eqref{eqn : PointPoint/Entier/Homogeneity}.

	Let $x_1, x_2 \in \Z^d$, $\zeta_1 > a \norme[1]{x_1}$ and $\zeta_2> a \norme[1]{x_2}$. Plugging $nx_i$ and $n\zeta_i$ into the inequality~\eqref{eqn : PointPoint/Entier/FKG} yields
	\begin{equation*}
		-\frac1n \log\Pb{ \PT\p{0,n(x_1+x_2)} \le n (\zeta_1+\zeta_2) }%
			\le -\frac1n \log\Pb{ \PT\p{0,nx_1} \le n\zeta_1 } -\frac1n \log\Pb{ \PT\p{0,nx_2} \le n\zeta_2 }.
	\end{equation*}
	Letting $n\to \infty$ gives~\eqref{eqn : PointPoint/Entier/Subadditivity}.

	Equation~\eqref{eqn : PointPoint/Entier/Symmetry} is a consequence of the invariance of the model with respect to the orthogonal symmetries of $\Z^d$.

	Let $x = \sum_{i=1}^d x(i)\base i\in\Z^d$ and $\zeta > a\norme[1] x$. Assume that $x\ge 0$. Consider
	\begin{equation*}
		y \dpe -(x(1)+1)\base 1 + \sum_{i=2}^d x(i)\base i
	\end{equation*}
	and note that
	\begin{equation*}
		2\p{x(1) + 1 } x = y + \p{2x(1) + 1 } (x + \base 1).
	\end{equation*}
		By~\eqref{eqn : PointPoint/Entier/Homogeneity} and~\eqref{eqn : PointPoint/Entier/Subadditivity},
	\begin{equation*}
		2\p{x(1) + 1 }\FdTpp\p{  x,\zeta} \le  \FdTpp\p{y, \zeta}  + \p{2x(1) + 1 } \FdTpp\p{x+ \base 1, \zeta}.
	\end{equation*}
	Besides,~\eqref{eqn : PointPoint/Entier/Symmetry} yields $\FdTpp\p{y, \zeta} = \FdTpp\p{x+\base 1, \zeta}$, thus
	\begin{equation*}
		2\p{x(1) + 1 }\FdTpp\p{  x,\zeta} \le  2\p{x(1) + 1 } \FdTpp\p{x+ \base 1, \zeta},
	\end{equation*}
	i.e.
	\begin{equation*}
		\FdTpp\p{x,\zeta} \le  \FdTpp\p{x+ \base 1, \zeta}.
	\end{equation*}
	The analogous inequality with any base vector $\base i$ instead of $\base 1$ holds. A straightforward induction argument gives~\eqref{eqn : PointPoint/Entier/Croissance}.
\end{proof}
\begin{proof}[Proof of Lemma~\ref{lem : PointPoint/Extension}]
	For all $(x,\zeta)\in \cX$, we define
	\begin{equation}
		\hatFdTpp(x,\zeta) \dpe \inf \set{\FdTpp(x',\zeta')}{ (x',\zeta')\in \cX, \quad x'\in\Q^d,\quad x'\ge \module x,\quad \zeta' \le \zeta}.
	\end{equation}
	By~\eqref{eqn : PointPoint/Entier/Symmetry} and~\eqref{eqn : PointPoint/Entier/Croissance}, if $x\in \Q^d$ then $\hatFdTpp(x,\zeta) =\FdTpp(x,\zeta)$. Item~\eqref{item : PointPoint/General/ConvexAbsCont} is a consequence of~\eqref{eqn : PointPoint/Entier/Homogeneity} and~\eqref{eqn : PointPoint/Entier/Subadditivity}. Items~\eqref{item : PointPoint/General/Symmetry} and~\eqref{item : PointPoint/General/Croissance} hold by definition. The function $\hatFdTpp$ is convex on the open set $\cX$ therefore it is continuous.
\end{proof}
\begin{proof}[Proof of Lemma~\ref{lem : PointPoint/OrdreGrandeur}]
	Let $(x,\zeta)\in \cX$.

	\emph{Direct implication:} Assume that $\zeta \ge \TC(x)$. We prove that $\FdTpp(x,\zeta)=0$. By lower semicontinuity of $\FdTpp$, it is sufficient to treat the case where $\zeta > \TC(x)$ and $x\in \Q^d$, which may be further reduced to $x\in\Z^d$. We conclude by convergence in probability of the rescaled passage time (see~\eqref{eqn : Intro/framework/def_TimeConstant_CerfTheret}).

	\emph{Converse implication:} Assume that $\zeta < \TimeConstant(x)$. Fix $\eps \dpe \frac12\p{\TC(x) - \zeta}$. We first treat the case where the passage times $\tau_e$ are a.s.\ bounded by $b<\infty$. Theorem~6.12 in Boucheron-Lugosi-Massart (2013) \cite{BLM} applied to the random variable $Z \dpe \frac1b \PT\p{0 , \floor{nx}}$, which depends only on a finite number of edge passage times and is a self-bounding function of these passage times, implies that for all $t>0$,
	\begin{equation}
		\Pb{Z \le \E Z - t} \le \exp\p{ - h\p{\frac{t}{\E Z} }\E Z },
	\end{equation}
	where $h(u) = (1+u)\log(1+u) -u$. Consequently,
	\begin{align*}
		\Pb{ \PT(0, \floor{nx} ) \le \E{\PT(0, \floor{nx} )} - t }%
			&= \Pb{Z \le \E Z - \frac tb}\\
			&\le \exp\p{ - h\p{\frac{t}{\E{\PT(0, \floor{nx} )} } }\cdot \frac{ \E{\PT(0, \floor{nx} )}}{b}    }.
	\end{align*}
	Besides, $\E{\PT(0, \floor{nx} )} \ge n(\TimeConstant(x)-\eps)$ for large $n$, thus the choice $t\dpe \eps n $ gives, for large $n$,
	\begin{align*}
		\Pb{ \PT(0, \floor{nx} ) \le n \zeta}%
			&\le \Pb{ \PT(0, \floor{nx} ) \le \E{\PT(0, \floor{nx} )} - \eps n}\\
			&\le \exp\p{ -h\p{\frac{ \eps n}{  \E{\PT(0, \floor{nx} )} } }\cdot \frac{ \E{\PT(0, \floor{nx} )} }{b}    },
	\end{align*}
	hence
	\begin{equation*}
		\lim_{n\to\infty}-\frac1n \log \Pb{ \PT(0, \floor{nx} ) \le n \zeta} \ge h\p{\frac{\eps}{\TC(x) }} \cdot \frac{\TC(x)}{b} >0,
	\end{equation*}
	i.e.\ $\FdTpp\p{x,\zeta}>0$.

	We now turn to the general case. For all $b>a$ and $e\in \bbE^d$, we define $\tau_e^{(b)} \dpe \tau_e \wedge b$ and denote by $\TC^{(b)}$ the associated time constant. Theorem~1.6 in Garet-Marchand-Procaccia-Théret (2017)~\cite{GMPT},\footnote{The result there is stated for $x\in \Z^d$, but the general case follows by standard arguments.} states that
	\begin{equation*}
		\lim_{b \to \infty} \TC^{(b)}(x) = \TC(x).	 
	\end{equation*} In particular there exists $b>a$ such that $\zeta < \TC^{(b)}(x)$. The straightforward inclusion
	\begin{equation*}
		\acc{\PT(0, \floor{nx} ) \le \zeta} \subseteq \acc{\PT^{(b)}(0, \floor{nx} ) \le \zeta}
	\end{equation*}
	and the previous case concludes the proof. 
\end{proof}

\section{Monotonous rate function}
\label{sec : FdTmon}

In this section we assume that~\eqref{ass : Intro/main_thm/SubcriticalAtom} and~\eqref{ass : Intro/main_thm/ShapeThmStrong} are satisfied and prove Theorem~\ref{thm : Intro/FdTmon}, which amounts to proving Propositions~\ref{prop : FdTmon/UpperBound},~\ref{prop : FdTmon/LowerBound},~\ref{prop : FdTmon/Intrinsic} and~\ref{prop : FdTmon/StrMon}. Recall the definitions~\eqref{eqn : Intro/main_thm/def_FdTmonsup} and~\eqref{eqn : Intro/main_thm/def_FdTmoninf} of $\FdTmonsup$ and $\FdTmoninf$.

\begin{Proposition}
	\label{prop : FdTmon/UpperBound}
	Let $D\in \RegPD$ and $\p{\sigma_k : \intervalleff0{L}\rightarrow \X }_{k\ge1}$ be a highway network for $D$ (see Definition~\ref{def : Topology/RegPD/HW_Network}). Then
	\begin{equation}
		\label{eqn : FdTmon/UpperBound}
		\FdTmonsup(D) \le \sum_{k\ge 1} \int_0^{L} \FdTpp\p{\sigma_k'(t), \MD{\sigma_k}(t)}\d t.
	\end{equation}
\end{Proposition}

\begin{Proposition}
	\label{prop : FdTmon/LowerBound}
	Let $D\in \RegPD$ and $\p{\gamma_k : \intervalleff{0}{T_k} \rightarrow \X}_{1\le k \le K}$ be a finite family of $1$-Lipschitz, injective and pairwise disjoint paths. Then
	\begin{equation}
		\label{eqn : FdTmon/Existence/ExpressionSupFini}
		\FdTmoninf(D)
			\ge \sum_{k= 1}^K \int_0^{T_k} \FdTpp\p{\gamma_k'(t), \MD{\gamma_k}(t)}\d t.
	\end{equation}
\end{Proposition}
Every pseudometric $D\in \RegPD$ has a highway network by Proposition~\ref{prop : Topology/RegPD/HW_cvg}, thus combining Propositions~\ref{prop : FdTmon/UpperBound} and~\ref{prop : FdTmon/LowerBound} gives $\FdTmon(D) \dpe \FdTmoninf(D) = \FdTmonsup(D)$. Moreover Equations~\eqref{eqn : FdTmon/Existence/ExpressionGeodesics},~\eqref{eqn : FdTmon/Existence/ExpressionSup} hold.
\begin{Proposition}
	\label{prop : FdTmon/Intrinsic}
	For all $D\in \RegPD$,
	\begin{equation}
		\label{eqn : FdTmon/IntrinsicExp}
		\FdTmon(D) = \int_\X \max_{u\in \S_2} \FdTpp\p{u, (\Pgrad D)_z(u)}  \hd\p{\d z}.
	\end{equation}
\end{Proposition}
\begin{Proposition}
	\label{prop : FdTmon/StrMon}
	Let $D_1, D_2 \in \RegPD$ be distinct pseudometrics satsifying $D_1 \le D_2$. If $\FdTmon(D_2)<\infty$, then
	\begin{equation}
		\label{eqn : FdTmon/StrMon}
		\FdTmon(D_1) <  \FdTmon(D_2).
	\end{equation}
\end{Proposition}

Assumption~\eqref{ass : Intro/main_thm/ShapeThmStrong}'s usefulness stems from Lemma~\ref{lem : Intro/main_thm/LemmeCool}, a variant of a core lemma in the proof of the shape theorem (see Lemma~2.20 in \cite{50yFPP}). We postpone its proof to Appendix~\ref{appsec : Cool}.
\begin{Lemma}
    \label{lem : Intro/main_thm/LemmeCool}
    Under Assumption~\eqref{ass : Intro/main_thm/ShapeThmStrong}, there exists a constant $\kappa<\infty$ such that 
    \begin{equation}
        \label{eqn : Intro/main_thm/LemmeCool}
        \liminf_{n\to\infty} \min_{x\in \intint0n^d} \Pb{\begin{array}{c} \text{For all $y \in \intint0n^d$, there exists a discrete path $x\Path\pi y$ included} \\ \text{in $\intint0n^d$, such that } \tau(\pi) \le \kappa \norme[1]{x-y} \text{ and } \norme{\pi} \le 2\norme[1]{x-y} + 4 \end{array}} >0.
    \end{equation}
\end{Lemma}
We will say that $x\in\intint0n^d$ is a \emph{hub} if the event in~\eqref{eqn : Intro/main_thm/LemmeCool} occurs.


\subsection{Upper bounding the monotonous rate function}
\label{subsec : FdTmon/UpperBound}

In this section we prove Proposition~\ref{prop : FdTmon/UpperBound}. We fix a pseudometric $D\in \RegPD$ and a highway network $\p{\sigma_k : \intervalleff0{L}\rightarrow \X }_{k\ge1}$ for $D$. For all $K\in \N$, we define
\begin{equation}
	D_K \dpe \HW{D}{\sigma_1, \dots, \sigma_K}.
\end{equation}
We first show that Lemmas~\ref{lem : FdTmon/UpperBound/Initialisation} and~\ref{lem : FdTmon/UpperBound/Heredity} imply Proposition~\ref{prop : FdTmon/UpperBound}, then prove said lemmas.
\begin{Lemma}
	\label{lem : FdTmon/UpperBound/Initialisation}
	The metric induced by $\TC$ on $X$ satisfies
	\begin{equation}
		\FdTmonsup(\TC) = 0.
	\end{equation}
\end{Lemma}
\begin{Lemma}
	\label{lem : FdTmon/UpperBound/Heredity}
	For all $K\in \N$,
	\begin{equation}
		\label{eqn : FdTmon/UpperBound/Heredity}
		\FdTmonsup(D_{K+1}) \le \FdTmonsup(D_{K}) + \int_0^L \FdTpp\p{\sigma_{K+1}'(t), \module[D]{\dot{\sigma}_{K+1} }(t) } \d t.
	\end{equation}
\end{Lemma}
\begin{proof}[Proof of Proposition~\ref{prop : FdTmon/UpperBound}]
	By Lemma~\ref{lem : FdTmon/UpperBound/Initialisation}, $\FdTmonsup(D_0) =  \FdTmonsup( \TC ) = 0$. By induction, using~\eqref{eqn : FdTmon/UpperBound/Heredity}, for all $K \ge 0$,
	\begin{equation}
		\FdTmonsup(D_{K}) \le \sum_{k=1}^K \int_0^L \FdTpp\p{\sigma_{k}'(t), \MD{\sigma_{k} }(t) } \d t.
	\end{equation}
	Since $(D_K)_{K\ge 0}$ converges to $D$ in $\BoundedF$ and $\FdTmonsup$ is lower semicontinuous,~\eqref{eqn : FdTmon/UpperBound} holds.
\end{proof}

\begin{proof}[Proof of Lemma~\ref{lem : FdTmon/UpperBound/Initialisation}]
	Let $\eps>0$. Consider a finite subset $\acc{x_1,\dots, x_K}$ of $\X$ such that $\X \subseteq \bigcup_{k=1}^K \clball[1]{x_k, \eps}$. Recall the definition of hubs below Lemma~\ref{lem : Intro/main_thm/LemmeCool}. We define the event
	\stepcounter{fav}
	\begin{equation}
		\label{eqn : FdTmon/UpperBound/Initialisation/DefFav}
		\cFav\dpe \p{\bigcap_{1\le k \le K} \acc{ \floor{nx_k}\text{ is a hub} } }%
			\cap \p{\bigcap_{1\le k_1 < k_2 \le K}\acc{ \BoxPT\p{ \floor{n x_{k_1}}, \floor{n x_{k_2}} } \le n \TC\p{x_{k_1} - x_{k_2}}  }  }.
	\end{equation}
	This event is an intersection of decreasing events, thus by the FKG inequality,
	\begin{equation*}
		\Pb\cFav \ge \p{\prod_{1\le k \le K} \Pb{ \floor{nx_k}\text{ is a hub} }  }%
			\cdot \p{\prod_{1\le k_1 < k_2 \le K}\Pb{ \BoxPT\p{ \floor{n x_{k_1}}, \floor{n x_{k_2}} } \le n \TC\p{x_{k_1} - x_{k_2}}  }  }.
	\end{equation*}
	By Theorem~\ref{thm : PointPoint} and Lemma~\ref{lem : Intro/main_thm/LemmeCool},
	\begin{equation}
		\label{eqn : FdTmon/UpperBound/Initialisation/PbFav}
		\lim_{n\to\infty} -\frac1n \log \Pb\cFav = 0.
	\end{equation}
	Besides, we claim that for large enough $n$,
	\begin{equation}
		\label{eqn : FdTmon/UpperBound/Initialisation/InclusionFav}
		\cFav \subseteq \LD_n^-\p{\TC, 2(\BoundTC + 2\kappa )\eps },
	\end{equation}
	where $\BoundTC$ is defined by~\eqref{eqn : Intro/DefNormeNorme}. Indeed assume that $\cFav$ occurs, and $n \ge \frac{d}{\eps}$. Let $x,y\in \X$. There exist $1\le k_1, k_2 \le K$ such that
	\begin{equation*}
		\norme[1]{ x - x_{k_1} } \le \eps \text{ and } \norme[1]{ y - x_{k_2} } \le \eps,
	\end{equation*}
	therefore
	\begin{equation*}
		\norme[1]{\vphantom{\big\vert} \floor{nx} - \floor{ nx_{k_1} } } \le 2n\eps \text{ and } \norme[1]{\vphantom{\big\vert} \floor{ny} - \floor{nx_{k_2}} } \le 2n\eps.
	\end{equation*}
	By triangle inequality,
	\begin{align*}
		\BoxPT\p{\floor{nx} , \floor{ny} } %
			&\le \BoxPT\p{\floor{nx}, \floor{nx_{k_1}} }  + \BoxPT\p{ \floor{nx_{k_1}} , \floor{nx_{k_2}} } + \BoxPT\p{ \floor{nx_{k_2} } ,\floor{ny} } \\
			&\le \kappa \norme[1]{\vphantom{\big\vert} \floor{nx} - \floor{ nx_{k_1} } }   + n\TC\p{x_{k_1} - x_{k_2}} + \kappa  \norme[1]{\vphantom{\big\vert} \floor{ny} - \floor{nx_{k_2}} } \\
			&\le n\TC\p{x_{k_1} - x_{k_2}} + 4\kappa n \eps.
		\intertext{Besides, using the triangle inequality again gives}
		\TC\p{x_{k_1} - x_{k_2}}%
			&\le \TC\p{x_{k_1} - x} + \TC\p{ x - y } +\TC\p{y - x_{k_2}}\\
			&\le \TC\p{x-y} + 2\BoundTC \eps .
		\intertext{Consequently,}
		\SPT(x,y)%
			&\le \TC\p{x-y} + 2(\BoundTC + 2\kappa )\eps,
	\end{align*}
	thus~\eqref{eqn : FdTmon/UpperBound/Initialisation/InclusionFav}.

	Equations~\eqref{eqn : FdTmon/UpperBound/Initialisation/PbFav} and~\eqref{eqn : FdTmon/UpperBound/Initialisation/InclusionFav} conclude the proof of the lemma.
\end{proof}

\begin{proof}[Proof of Lemma~\ref{lem : FdTmon/UpperBound/Heredity}]
	For all $K\in \N$. To simplify the notations we define $\sigma \dpe \sigma_{K+1}$. Let $\eps>0$ and $P\in \N^*$. For all $0\le p \le P$, define $t_p\dpe \frac{pL}{P}$. Consider the event
	\stepcounter{fav}
	\begin{equation}
		\label{eqn : FdTmon/UpperBound/Definition_Fav}
		\cFav= \cFav(n,\eps, P) \dpe \LD_n^-(D_K, \eps) \cap \p{ \bigcap_{p=0}^{P-1} \acc{ \frac1n\BoxPT\p{\floor{n\sigma(t_p}), \floor{n\sigma(t_{p+1} )}}\le D\p{\sigma(t_p), \sigma(t_{p+1} )} + \eps } }.
	\end{equation}
	We claim that
	\begin{equation}
		\label{eqn : FdTmon/UpperBound/Pb_Fav}
		\limsup_{n\to\infty}-\frac1n \log\Pb{\cFav} \le \FdTmonsup(D_K) + \int_0^L \FdTpp\p{\sigma'(t), \MD{\sigma}(t) }\d t.
	\end{equation}
	Indeed $\cFav$ is an intersection of decreasing events, thus the FKG inequality gives
	\begin{equation*}
		\Pb{\cFav} \ge \Pb{\LD_n^-(D_K, \eps)} \cdot \p{ \prod_{p=0}^{P-1} \Pb{ \frac1n\BoxPT\p{\floor{n\sigma(t_p}), \floor{n\sigma(t_{p+1} )}}\le D\p{\sigma(t_p), \sigma(t_{p+1} )} + \eps } },
	\end{equation*}
	therefore by Theorem~\ref{thm : PointPoint},
	\begin{align}
		\limsup_{n\to\infty}-\frac1n \log\Pb{\cFav} &\le \FdTmonsup(D_K) + \sum_{p=0}^{P-1}\FdTpp\p{\sigma(t_{p+1} ) - \sigma(t_{p} ) , D\p{\sigma(t_p), \sigma(t_{p+1} )} }.\notag
		\intertext{Since $\sigma$ is a $D$-geodesic,  $D\p{\sigma(t_p), \sigma(t_{p+1} )} = D\p{\restriction{\sigma}{\intervalleff{t_p}{t_{p+1} } } }$ hence}
		\limsup_{n\to\infty}-\frac1n \log\Pb{\cFav} &\le \FdTmonsup(D_K) + \sum_{p=0}^{P-1}\FdTpp\p{\sigma(t_{p+1} ) - \sigma(t_{p} ) , D\p{\restriction{\sigma}{\intervalleff{t_p}{t_{p+1} } } } }.\notag
		\intertext{Lemma~\ref{lem : Topology/GMT/Rademacher} gives}
		\limsup_{n\to\infty}-\frac1n \log\Pb{\cFav} &\le \FdTmonsup(D_K) + \sum_{p=0}^{P-1}\FdTpp\p{\int_{t_p}^{t_{p+1}} \sigma'(t) \d t , \int_{t_p}^{t_{p+1}} \MD{\sigma}(t) \d t }.
	\end{align}
	The function $\FdTpp$ is convex and homogeneous therefore Jensen's inequality in its multivariate version (see~\cite{Per74}) yields~\eqref{eqn : FdTmon/UpperBound/Pb_Fav}.

	We now show that for all $n\in \N^*$,
	\begin{equation}
		\label{eqn : FdTmon/UpperBound/Inclusion_Fav}
		\cFav \subseteq \LD_n^-\p{D, (P+2)\eps + \frac{4 L \BoundTC}{P} }.
	\end{equation}
	Let $n\in\N^*$. Assume that $\cFav$ occurs. Let $x,y\in \X$. By Definition~\ref{def : Topology/RegPD/HW}, one of the following holds:
	\begin{enumerate}[(i)]
		\item \label{item : FdTmon/UpperBound/CasFacile} $D_{K+1}(x,y) = D_K(x,y)$.
		\item \label{item : FdTmon/UpperBound/CasDifficile} There exists $0\le s,t \le L$ such that%
		\[ D_{K+1}(x,y) = D_K\p{x, \sigma(s)} + D\p{\sigma(s), \sigma(t)} + D_K\p{\sigma(t), y}. \]
	\end{enumerate}
	In the case~\eqref{item : FdTmon/UpperBound/CasFacile},
	\begin{equation}
		\label{eqn : FdTmon/UpperBound/CasFacile/WinCon}
		\SPT\p{x,y} \le D_{K+1}(x,y) + \eps 
	\end{equation}
	is straightforward. In the case~\eqref{item : FdTmon/UpperBound/CasDifficile}, there exist $p,q\in \intint0P$ such that $\module{s-t_p}, \module{t-t_q}\le \frac LP$. Since $\sigma$ is $1$-Lipschitz and $D \le D_K \le \TC$, we have by triangle inequality,
	\begin{equation}
		\label{eqn : FdTmon/UpperBound/CasDifficile/Target}
		D_{K+1}(x,y) \ge D_K\p{x, \sigma(t_p)} + D\p{\sigma(t_p), \sigma(t_q)} + D_K\p{\sigma(t_q), y} - \frac{4 L \BoundTC}{P}.
	\end{equation}
	Up to exchanging $x$ and $y$ we may assume that $p\le q$. By triangle inequality
	\begin{align}
		\SPT\p{x,y}%
			&\le \SPT\p{x, \sigma(t_p)} + \sum_{i=p}^{q-1} \SPT\p{ \sigma(t_i), \sigma(t_{i+1})} + \SPT\p{\sigma(t_q), y} \eol
			&\le D_K\p{x, \sigma(t_p)} + \sum_{i=p}^{q-1} D\p{ \sigma(t_i), \sigma(t_{i+1})} + D_K\p{\sigma(t_q), y} + (P+2)\eps.\notag
		\intertext{Besides, $\sigma$ is a $D$-geodesic therefore}
		\label{eqn : FdTmon/UpperBound/CasDifficile/UB1}
		\SPT\p{x,y}%
			&\le D_K\p{x, \sigma(t_p)} + D\p{ \sigma(t_p), \sigma(t_q)} + D_K\p{\sigma(t_q), y} + (P+2)\eps.
	\end{align}
	Combining~\eqref{eqn : FdTmon/UpperBound/CasDifficile/Target} and~\eqref{eqn : FdTmon/UpperBound/CasDifficile/UB1}, we get
	\begin{equation}
		\label{eqn : FdTmon/UpperBound/CasDifficile/WinCon}
		\SPT\p{x,y} \le D_{K+1}(x,y) + (P+2)\eps + \frac{4 L \BoundTC}{P}.
	\end{equation}
	The inclusion~\eqref{eqn : FdTmon/UpperBound/Inclusion_Fav} is a consequence of~\eqref{eqn : FdTmon/UpperBound/CasFacile/WinCon} and~\eqref{eqn : FdTmon/UpperBound/CasDifficile/WinCon}.

	By~\eqref{eqn : FdTmon/UpperBound/Pb_Fav} and~\eqref{eqn : FdTmon/UpperBound/Inclusion_Fav},
	\begin{equation}
		\limsup_{n\to\infty} -\frac1n \log \Pb{ \LD_n^-\p{D_{K+1}, (P+2)\eps + \frac{4 L \BoundTC}{P} } } \le \FdTmonsup(D_K) + \int_0^L \FdTpp\p{\sigma'(t), \MD{\sigma}(t) }\d t.
	\end{equation}
	Choosing $\eps \dpe P^{-2}$ and letting $P\to \infty$ yields~\eqref{eqn : FdTmon/UpperBound/Heredity}.
\end{proof}

\subsection{Lower bounding the monotonous rate function}
\label{subsec : FdTmon/LowerBound}
In this section we fix a pseudometric $D\in \RegPD$ and prove Proposition~\ref{prop : FdTmon/LowerBound}. We first show that it is a consequence of Lemma~\ref{lem : FdTmon/LowerBoundWeak} then prove the lemma.
\begin{Lemma}
	\label{lem : FdTmon/LowerBoundWeak}
	For all finite families of injective, pairwise disjoint, $1$-Lipschitz paths $\p{\gamma_k : \intervalleff{0}{T_k} \rightarrow \X }_{1\le k \le K}$,
	\begin{equation}
		\label{eqn : FdTmon/LowerBoundWeak}
		\FdTmoninf(D) \ge \sum_{k=1}^K \FdTpp\p{\gamma_k(T_k) - \gamma_k(0), D(\gamma_k)}.
	\end{equation}
\end{Lemma}
\begin{proof}[Proof of Proposition~\ref{prop : FdTmon/LowerBound}]
	Let $\p{\gamma_k}_{1\le k \le K}$ be a family of paths as in Proposition~\ref{prop : FdTmon/LowerBound}. Let $P \in \N^*$ and $0 < \delta < \frac1P \min_{1\le k \le K} T_k$. For all $1\le p \le P$ and $1\le k \le K$, we define the paths
	\begin{equation*}
		\gamma_{k,p} \dpe \restriction{\gamma_k}{\intervalleff{\frac{(p-1)T_k}{P}}{ \frac{pT_k}{P} } } \text{ and } \gamma_{k,p}^{(\delta)} \dpe \restriction{\gamma_k}{\intervalleff{\frac{(p-1)T_k}{P}}{ \frac{pT_k}{P} -\delta } }.
	\end{equation*}
	The paths $\p{\gamma_{k,p}^{(\delta)} }_{\substack{1\le k \le K \\ 1\le p \le P} }$ are $1$-Lipschitz, injective and pairwise disjoint therefore by Lemma~\ref{lem : FdTmon/LowerBoundWeak},
	\begin{align}
		\FdTmoninf(D)%
			&\ge \sum_{k=1}^K \sum_{p=1}^P \FdTpp\cro{ \gamma_k\p{\frac{pT_k}{P} - \delta} - \gamma_k\p{\frac{(p-1)T_k}{P}}  , D\p{ \gamma_{k,p}^{(\delta)} } }.\nonumber%
		\intertext{Letting $\delta\to0$ and using the lower semicontinuity of $\FdTpp$, we get}
		\FdTmoninf(D)%
			&\ge \sum_{k=1}^K \sum_{p=1}^P \FdTpp\cro{ \gamma_k\p{\frac{pT_k}{P}} - \gamma_k\p{\frac{(p-1)T_k}{P}}  , D\p{ \gamma_{k,p} } }\nonumber\\
			&= \sum_{k=1}^K \int_0^{T_k} \frac{P}{T_k} \FdTpp\cro{ \gamma_k\p{\frac{T_k}{P} \ceil{\frac{Pt}{T_k}} } - \gamma_k\p{\frac{T_k}{P} \p{ \ceil{\frac{Pt}{T_k}} -1 } }   , D\p{ \gamma_{k,\ceil{\frac{Pt}{T_k}}} } } \d t.\nonumber
		\intertext{By homogeneity of $\FdTpp$,}
		\label{eqn : FdTmon/LowerBound/Morceaux}
		\FdTmoninf(D)%
			&\ge \sum_{k=1}^K \int_0^{T_k}  \FdTpp\cro{ \frac{P}{T_k}\cro{\gamma_k\p{\frac{T_k}{P} \ceil{\frac{Pt}{T_k}}} - \gamma_k\p{\frac{T_k}{P} \p{ \ceil{\frac{Pt}{T_k}} -1 } } }   , \frac{P}{T_k}D\p{ \gamma_{k,\ceil{\frac{Pt}{T_k}}} } } \d t.
	\end{align}
	By~\eqref{eqn : Topology/GMT/Rademacher/Classic},~\eqref{eqn : Topology/GMT/Rademacher/VariationTotale} and Lebesgue's differentiation theorem (see e.g.\ Theorem~7.10 in \cite{Rudin}), for all $1\le k \le K$ and almost every $t\in \intervalleff{0}{T_k}$,
	\begin{equation}
		\label{eqn : FdTmon/LowerBound/Der}
		\lim_{P\to\infty} \frac{P}{T_k}\cro{\gamma_k\p{\frac{T_k}{P} \ceil{\frac{Pt}{T_k}}} - \gamma_k\p{\frac{T_k}{P} \p{ \ceil{\frac{Pt}{T_k}} -1 } } } = \gamma_k'(t)
	\end{equation}
	and
	\begin{equation}
		\label{eqn : FdTmon/LowerBound/MetDer}
		\lim_{P\to\infty} \frac{P}{T_k}D\p{ \gamma_{k,\ceil{\frac{Pt}{T_k}}} } = \MD{\gamma_k}(t).
	\end{equation}
	For all $t$ satisfying~\eqref{eqn : FdTmon/LowerBound/Der} and~\eqref{eqn : FdTmon/LowerBound/MetDer}, by lower semicontinuity of $\FdTpp$,
	\begin{equation*}
		\liminf_{P\to \infty} \FdTpp\cro{ \frac{P}{T_k}\cro{\gamma_k\p{\frac{T_k}{P} \ceil{\frac{Pt}{T_k}}} - \gamma_k\p{\frac{T_k}{P} \p{ \ceil{\frac{Pt}{T_k}} -1 } } }   , \frac{P}{T_k}D\p{ \gamma_{k,\ceil{\frac{Pt}{T_k}}} } } %
			\ge \FdTpp\p{ \gamma_k'(t), \MD{\gamma_k}(t) }.
	\end{equation*}
	Consequently, taking the inferior limit as $P\to \infty$ in~\eqref{eqn : FdTmon/LowerBound/Morceaux} and applying Fatou's lemma gives~\eqref{eqn : FdTmon/Existence/ExpressionSupFini}.
\end{proof}

\begin{proof}[Proof of Lemma~\ref{lem : FdTmon/LowerBoundWeak}]
	If $\FdTmoninf(D) = \infty$, there is nothing to prove. Assume that $\FdTmoninf(D)<\infty$. Let $\p{\gamma_k}_{1\le k \le K}$ be a family of paths as in the lemma.	Fix $\delta>0$ such that for all distinct $1\le k_1, k_2 \le K$,
	\begin{equation}
		\label{eqn : FdTmon/LowerBoundWeak/Couenne}
		\p{\gamma_{k_1} + \clball{0,2\delta} } \cap \p{\gamma_{k_2} + \clball{0,2\delta}} = \emptyset .
	\end{equation}
	For all $1\le k \le K$, we will denote by $\overline\gamma_k$ the set $\gamma_k + \clball{0,2\delta}$. Define
	\begin{equation}
		\eta \dpe \frac{\delta \min_{u\in \S_1}\TC(u)}{2\BoundTC}  \le \delta.
	\end{equation}
	Note that
	\begin{equation}
		\min_{u\in \S_1} \FdTpp\p{\delta u, \BoundTC \eta}>0.
	\end{equation}
	Define
	\begin{equation}
		R\dpe \ceil{\frac{\FdTmoninf(D) + 1}{\min_{u\in \S_1} \FdTpp\p{\delta u, \BoundTC \eta}} }.
	\end{equation}
	Fix $P\in \N^*$ large enough so that $\frac{\max_{1\le k \le K} T_k}{P} \le \frac\eta2$, and $0<\eps \le \frac{\eta \BoundTC}2$. Let $n\ge1$. For all $0\le p \le P$ and $1\le k \le K$, define $x_{k,p} \dpe \floor{n \gamma_k\p{\frac{p T_k}{P}} }$. Recall the definition of hub below Lemma~\ref{lem : Intro/main_thm/LemmeCool}. We consider the events

	\begin{align}
		\label{eqn : FdTmon/LowerBoundWeak/cE_1}
		\cE_1(n) &\dpe \acc{\text{For all $1\le k \le K$ and $0\le p \le P$, $x_{k,p}$ is a hub}},
	\intertext{and}
		\cE_2(n) &\dpe \acc{\begin{array}{c}\text{There exist pairwise disjoint discrete paths $z_1 \Path{\pi_1} z_1', \dots, z_R \Path{\pi_R} z_R'$ in $\intint0n^d$} \\ \text{such that for all $r \in \intint1R$, }  \norme[1]{z_r - z_r'}\ge n\delta \text{ and } \tau(\pi_r) \le n \BoundTC \eta  \end{array}}^\mathrm c.
	\end{align}
	We define
	\stepcounter{fav}
	\begin{equation}
		\label{eqn : FdTmon/LowerBoundWeak/DefFav}
		\cFav = \cFav(n,\eps) \dpe \LD_n^-(D,\eps) \cap \cE_1(n) \cap \cE_2(n).
	\end{equation}
	We claim that
	\begin{equation}
		\label{eqn : FdTmon/LowerBoundWeak/Pb_LD_perturbe}
		\FdTmoninf(D) = \inclim{\eps \to 0} \liminf_{n\to\infty} -\frac1n \log\Pb{\cFav(n,\eps) }.
	\end{equation}
	Indeed $\cE_1(n)$ is an intersection of a finite number of decreasing events whose probability is lower bounded by~\eqref{eqn : Intro/main_thm/LemmeCool}, thus the FKG inequality gives \[ \lim_{n\to\infty} -\frac1n \log \Pb{\cE_1(n)} = 0. \] Since $\LD_n^-(D,\eps)$ is also decreasing, using the FKG inequality again yields
	\begin{equation}
		\label{eqn : FdTmon/LowerBoundWeak/Pb_LD_cap_cE1}
		\FdTmoninf(D) = \inclim{\eps \to 0} \liminf_{n\to\infty} -\frac1n \log\Pb{\LD_n^-(D,\eps) \cap \cE_1(n) }.
	\end{equation}
	Besides, by the union bound and the BK inequality, for all $n\in\N^*$,
	\begin{align*}
		\Pb{\cE_2(n)^\mathrm c} &\le \sum_{ \substack{(z_r)_{1\le r \le R} \\ (z_r')_{1\le r \le R}} } \prod_{r=1}^R \Pb{\PT(z_r, z_r')\le n\BoundTC \eta},
		\intertext{where the sum spans over families of points of $\intint0n^d$ such that for all $1\le r\le R$, $\norme[1]{z_r - z_r'}\ge n\delta$. In particular there exists a polynomial $\Pol$ such that for all $n\in \N^*$, }
		\Pb{\cE_2\p{n}^\mathrm c} &\le \Pol(n)\sup_{ \substack{ z \in \Z^d\\ \norme[1]{z}\ge n\delta} } \Pb{\PT(0,z) \le n\BoundTC \eta}^R.
		\intertext{Taking the $\log$, multiplying by $-\frac1n$ and using~\eqref{eqn : PointPoint/Entier/Limite}, we get}
		-\frac1n \log \Pb{\cE_2\p{n}^\mathrm c} %
			&\ge -\frac1n\log \Pol(n) + R \inf_{ \substack{ z \in \Z^d \\ \norme[1]{z}\ge n\delta} } -\frac1n \log\Pb{\PT(0,z) \le n\BoundTC \eta}\\
			&\ge -\frac1n\log \Pol(n) + R \min_{\norme[1]u \ge 1} \FdTpp\p{\delta u, \BoundTC \eta}%
		\intertext{Applying Theorem~\ref{thm : PointPoint}\eqref{item : PointPoint/General/Croissance} gives}
		-\frac1n \log \Pb{\cE_2\p{n}^\mathrm c}
			&\ge -\frac1n\log \Pol(n) + R \min_{u\in \S_1} \FdTpp\p{\delta u, \BoundTC \eta}\\
			&\ge -\frac1n\log \Pol(n) + \FdTmoninf(D) +1,
	\end{align*}
		hence
	\begin{equation}
		\label{eqn : FdTmon/LowerBoundWeak/Pb_cE2}
		\liminf_{n\to\infty}-\frac1n \log \Pb{\cE_2\p{n}^\mathrm c} \ge \FdTmoninf(D) +1.
	\end{equation}
	Combining~\eqref{eqn : FdTmon/LowerBoundWeak/Pb_LD_cap_cE1} and~\eqref{eqn : FdTmon/LowerBoundWeak/Pb_cE2} gives~\eqref{eqn : FdTmon/LowerBoundWeak/Pb_LD_perturbe}.

	Fix $1\le k \le K$. We now show that for large enough $n$,
	\begin{equation}
		\label{eqn : FdTmon/LowerBoundWeak/Inclusion}
		\cFav \subseteq \acc{ \frac1n\RestPT{n\overline\gamma_k}\p{ \floor{n \gamma_k(0)}, \floor{n \gamma_k(T_k)} } \le D\p{\gamma_k} + \eps P  +  \frac{2 \kappa T_k R}{P}  }.
	\end{equation}
Assume that $\cFav$ occurs, and $n \ge \frac{P (3d+4)}{T_k}$. For all $0 \le p \le P-1$, we denote by $\pi_p$ a discrete $\BoxPT$-geodesic from $x_{k,p}$ to $x_{k,p+1}$. We say that $p$ is \emph{good} if $\pi_p \subseteq n\overline\gamma_k$, and \emph{bad} otherwise. If $p$ is bad, we denote by $x_p \Path{\hat \pi_p} y_p$ the largest subpath of $\pi_p$ containing $x_p$ and included in $n\overline\gamma_k$. For all bad integer $p$, we denote by $\cro{p}$ the set of bad integers $q$ such that $x_p$ and $x_q$ are connected by a discrete path included in%
	\[ \cR \dpe \bigcup_{p\text{ bad}} \hat \pi_p. \]%
	Moreover, a bad integer $p$ is called \emph{maximal bad} if $p=\max \cro p$. Let $\p{p(0), \dots, p(N)}$ be the only subsequence of $\intint0P$ such that $p(0)=0$, $p(N)=P$ and for all $i\in \intint0{N-1}$:
	\begin{enumerate}[(i)]
		\item \label{item : FdTmon/LowerBoundWeak/Good} If $p(i)$ is good then $p(i+1)=p(i)+1$.
		\item \label{item : FdTmon/LowerBoundWeak/MaxBad} If $p(i)$ is maximal bad then $p(i+1)=p(i)+1$.
		\item \label{item : FdTmon/LowerBoundWeak/Bad} If $p(i)$ is bad but not maximal bad then $p(i+1) = \max \cro{p(i)}$.
	\end{enumerate}
	Let $i\in\intint{0}{r-1}$. In the case~\eqref{item : FdTmon/LowerBoundWeak/Good}, since $\pi_p \subseteq n\overline \gamma_k$,
	\begin{align}
		\frac1n \RestPT{n\overline\gamma_k} \p{ x_{k,p(i)}, x_{k,p(i)+1}  }%
			&= \frac1n \BoxPT\p{ x_{k,p(i)}, x_{k,p(i)+1}  }.
		\intertext{Since $\LD_n^-(D,\eps)$ occurs,}
		\label{eqn : FdTmon/LowerBoundWeak/Good}
		\frac1n \RestPT{n\overline\gamma_k} \p{ x_{k,p(i)}, x_{k,p(i+1)}  }
			&\le  D\p{ \gamma_k\p{ \frac{p(i)T_k}{P} }, \gamma_k\p{ \frac{p(i+1)T_k}{P} }  } + \eps.
	\end{align}
	In the case~\eqref{item : FdTmon/LowerBoundWeak/MaxBad}, since $\cE_1(n)$ occurs and $\norme[1]{x_{k,p(i)} - x_{k,p(i)+1}} \le \frac{nT_k}P + d$, there exists a discrete path $x_{p(i)} \Path{\pi_{p(i)}' } x_{p(i)+1}$ included in $\intint0n^d$, such that
	\begin{equation*}
		\tau(\pi_{p(i)}') \le \kappa\p{\frac{nT_k}{P} +d } \le  \frac{2nT_k \kappa}{P} \text{ and } \norme{\pi_{p(i)}'} \le 2\p{\frac{nT_k}{P} +d }+4 \le \frac{3nT_k}{P}.
	\end{equation*}
	Besides,
	\begin{align*}
		\d\p{ \pi_{p(i)}', (n\overline\gamma_k)^\mathrm{c} }%
			&\ge \d\p{x_p, (n\overline\gamma_k)^\mathrm{c}} - \norme{\pi_{p(i)}'}\\
			&\ge \d\p{ n\gamma_k\p{\frac{p T_k}{P}} ,  (n\overline\gamma_k)^\mathrm{c}} - \norme[1]{x_{k,p} - n\gamma\p{\frac{p T_k}{P}}} - \norme{\pi_{p(i)}'} \\
			&\ge 2n\delta - d - \frac{3n T_k}{P} \\
			&\ge n\delta>0.
	\end{align*}
	Consequently, $\pi_{p(i)}' \subseteq n\overline \gamma_k$, thus
	\begin{equation}
		\label{eqn : FdTmon/LowerBoundWeak/MaxBad}
		\frac1n \RestPT{n\overline\gamma_k} \p{ x_{k,p(i)}, x_{k,p(i+1)}  }%
			\le  \frac{2nT_k \kappa}{P}.
	\end{equation}
	In the case~\eqref{item : FdTmon/LowerBoundWeak/Bad}, there exists a self-avoiding path $x_{k,p(i)} \Path{\pi_{p(i)}'} x_{k,p(i+1)} $  included in $\cR$. It satisfies
	\begin{align}
		\tau\p{ \pi_{p(i)}' } %
			&\le \sum_{ \substack{q\in \cro{p(i)} } } \tau\p{ \hat \pi_q}\eol
			&\le \sum_{ \substack{q\in \cro{p(i)} } } \tau\p{ \pi_q}.\eol
			&= \sum_{ \substack{q\in \cro{p(i)} } } \BoxPT\p{ x_{k,q}, x_{k,q+1}}.\notag
	\end{align}
	Besides, $\LD_n^-\p{D,\eps}$ occurs therefore
	\begin{equation}
		\label{eqn : FdTmon/LowerBoundWeak/Bad}
		\frac1n\RestPT{n\overline\gamma}\p{x_{k,p(i)}, x_{k,p(i+1)}}%
		\le \p{ \sum_{ \substack{q\in \cro{p(i)} } } D\p{ \gamma\p{ \frac{qT_k}{P} }, \gamma\p{ \frac{(q+1)T_k}{P} }} } + \eps \cdot \#\cro{p(i)}.
	\end{equation}
	By triangle inequality,~\eqref{eqn : FdTmon/LowerBoundWeak/Good},~\eqref{eqn : FdTmon/LowerBoundWeak/MaxBad} and~\eqref{eqn : FdTmon/LowerBoundWeak/Bad} yield
	\begin{align}
		\frac1n\RestPT{n\overline\gamma_k}\p{x_{k,0}, x_{k,P} }%
			&\le \sum_{\substack{0\le p \le P-1 \\ p\text{ good}} } \p{ D\p{ \gamma_k\p{ \frac{pT_k}{P} }, \gamma_k\p{ \frac{(p+1)T_k}{P} }  }+ \eps } + \frac{2T_k \kappa}{P} \# \set{0\le p \le P-1}{p \text{ is maximal bad}} \eol
			&\quad + \sum_{\substack{0\le p \le P-1 \\ p\text{ bad}} } \p{ D\p{ \gamma_k\p{ \frac{pT_k}{P} }, \gamma\_kp{ \frac{(p+1)T_k}{P} }  }+ \eps }, \notag
		\intertext{thus}
		\label{eqn : FdTmon/LowerBoundWeak/AlmostInclusion}
		\frac1n\RestPT{n\overline\gamma_k}\p{x_{k,0}, x_{k,P} }
			&= \sum_{p=0}^{P-1}D\p{ \gamma_k\p{ \frac{pT_k}{P} }, \gamma_k\p{ \frac{(p+1)T_k}{P} } } + \eps P + \frac{2T_k \kappa}{P} \# \set{0\le p \le P-1}{p \text{ is maximal bad}}.
	\end{align}
	Moreover, for all bad $p$, $\norme{x_p-y_p} \ge n\delta$, and
	\begin{align*}
	 	\frac1n \tau\p{\hat \pi_p} %
	 		&\le D\p{ \gamma_k\p{ \frac{pT_k}{P} }, \gamma\p{ \frac{(p+1)T_k}{P} }  } +\eps\\
	 		&\le \frac{\BoundTC T_k}{P} + \eps \le \BoundTC\p{\frac\eta2 + \frac\eta2} = \BoundTC \eta.
	\end{align*} 
	Since the paths $\hat \pi_p$ for maximal bad integers $p$ are pairwise disjoint and $\cE_2(n)$ occurs, there are at most $R$ maximal bad integers. Plugging this bound into~\eqref{eqn : FdTmon/LowerBoundWeak/AlmostInclusion} gives, for large enough $n$,
	\begin{equation}
		\frac1n\RestPT{n\overline\gamma_k}\p{x_{k,0}, x_{k,P} }%
			\le D\p{\gamma_k} + \eps P  +  \frac{2 \kappa T_k R}{P},
	\end{equation}
	thus the inclusion~\eqref{eqn : FdTmon/LowerBoundWeak/Inclusion}.

	By~\eqref{eqn : FdTmon/LowerBoundWeak/Inclusion}, for large enough $n$,
	\begin{equation*}
		\cFav(n,\eps) \subseteq \bigcap_{k=1}^K \acc{ \frac1n\RestPT{n\overline\gamma_k}\p{ \floor{n \gamma_k(0)}, \floor{n \gamma_k(T_k)} } \le D\p{\gamma_k} + \eps P  +  \frac{2 \kappa T_k R}{P}  }.
	\end{equation*}
	Note that the events the intersection are independent, thus by~\eqref{eqn : FdTmon/LowerBoundWeak/Pb_LD_perturbe},
	\begin{equation}
		\FdTmoninf(D) %
			\ge \sum_{k=1}^K \liminf_{n\to\infty}-\frac1n \log \Pb{ \frac1n\RestPT{n\overline\gamma_k}\p{ \floor{n \gamma_k(0)}, \floor{n \gamma_k(T_k)} } \le D\p{\gamma_k} + \eps P  +  \frac{2 \kappa T_k R}{P}  }.
	\end{equation}
	Theorem~\ref{thm : PointPoint} gives
	\begin{equation*}
		\FdTmoninf(D)
			\ge \sum_{k=1}^K \FdTpp\p{\gamma_k(T_k) - \gamma_k(0) , D\p{\gamma_k} + \eps P  +  \frac{2 \kappa T_k R}{P}  }.
	\end{equation*}
	Since $\FdTpp$ is lower semicontinuous, letting $\eps\to0$ then $P\to \infty$ leads to~\eqref{eqn : FdTmon/LowerBoundWeak}.
\end{proof}


\subsection{Intrisic expression}
\label{subsec : FdTmon/Intrinsic}
In this section we prove Proposition~\ref{prop : FdTmon/Intrinsic}. We fix $L>0$, $D \in \RegPD^L$ and a highway network $\p{\sigma_k : \intervalleff0L \rightarrow \X}_{k\ge 1}$ (see~Definition~\ref{def : Topology/RegPD/HW_Network}). Our main argument is Lemma~\ref{lem : FdTmon/Intrinsic/Pgrad}, which states that if $\FdTmon(D) < \infty$, then $(\Pgrad D)_z(u) = \TC(u)$ except if $z$ belongs to a highway, and $u$ is tangent to the highway at $z$. We postpone its proof after the proof of Proposition~\ref{prop : FdTmon/Intrinsic}.
\begin{Lemma}
	\label{lem : FdTmon/Intrinsic/Pgrad}
	Assume that $\FdTmon(D)<\infty$. Then for $\hd$-almost all $z\in\X \setminus\bigcup_{k\ge 1}\sigma_k$,
	\begin{equation}
		(\Pgrad D)_z = \TC,
	\end{equation}
	and for all $k\ge 1$, for $\hd$-almost all $z  = \sigma_k(t)\in \sigma_k$, for all $u\in \R^d \setminus \p{\sigma_k'(t)\R }$,
	\begin{equation}
		(\Pgrad D)_z(u) = \TC(u).
	\end{equation}
\end{Lemma}
\begin{proof}[Proof of Proposition~\ref{prop : FdTmon/Intrinsic}]
	For all $k \ge 1$, using~\eqref{eqn : Topology/GMT/Gradient/Equality} and Lemma~\ref{lem : Topology/GMT/AreaFormula}, we get
	\begin{align}
	 	\int_0^L \FdTpp\p{\sigma_k'(t), \MD{\sigma_k}(t)}\d t%
	 		&= \int_0^L \FdTpp\p{\sigma_k'(t), (\Pgrad D)_{\sigma_k(t)}\p{\sigma_k'(t) }  }\d t \eol
	 		&= \int_{\sigma_k} \FdTpp\p{ \frac{\sigma_k'\p{ \sigma_k^{-1}(z) }}{ \norme[2]{\sigma_k'\p{ \sigma_k^{-1}(z) }} }, (\Pgrad D)_{z}\p{ \frac{\sigma_k' \p{ \sigma_k^{-1}(z) }}{ \norme[2]{\sigma_k'\p{ \sigma_k^{-1}(z) }} } }}  \hd(\d z).\notag
	\end{align}
	By~\eqref{eqn : FdTmon/Intrinsic},
	\begin{equation}
	 	\label{eqn : FdTmon/Intrinsic/OutputAreaFormula}
	 	\FdTmon(D) = \sum_{k\ge 1} \int_{\sigma_k} \FdTpp\p{ \frac{\sigma_k'\p{ \sigma_k^{-1}(z) }}{ \norme[2]{\sigma_k'\p{ \sigma_k^{-1}(z) }} }, (\Pgrad D)_{z}\p{ \frac{\sigma_k' \p{ \sigma_k^{-1}(z) }}{ \norme[2]{\sigma_k'\p{ \sigma_k^{-1}(z) }} } }}  \hd(\d z).
	\end{equation}
	If $\FdTmon(D) = \infty$, the result now follows from
	\begin{equation}
		\label{eqn : FdTmon/Intrinsic/TrivialBound}
		\FdTpp\p{ \frac{\sigma_k'\p{ \sigma_k^{-1}(z) }}{ \norme[2]{\sigma_k'\p{ \sigma_k^{-1}(z) }} }, (\Pgrad D)_{z}\p{ \frac{\sigma_k' \p{ \sigma_k^{-1}(z) }}{ \norme[2]{\sigma_k'\p{ \sigma_k^{-1}(z) }} } }}%
		\le \max_{u\in \S_2} \FdTpp\p{u, (\Pgrad D)_z(u)},
	\end{equation}
	for all $z\in \bigcup \sigma_k$. Otherwise, by Theorem~\ref{thm : PointPoint}\eqref{item : PointPoint/General/OrdreGrandeur}, for all $u\in\S_2$,
	\begin{equation*}
		\FdTpp\p{ u, \TC(u)} =0.
	\end{equation*}
	In particular, by Lemma~\ref{lem : FdTmon/Intrinsic/Pgrad},~\eqref{eqn : FdTmon/Intrinsic/TrivialBound} is an equality $\hd$-almost everywhere on the highway network, and the right-hand side is zero $\hd$-almost everywhere outside the highway network, giving the desired result.
\end{proof}
\begin{proof}[Proof of Lemma~\ref{lem : FdTmon/Intrinsic/Pgrad} ]
	Let $0<\eps < 1$. Note that since $\FdTpp$ is continuous on $\cX$, Theorem~\ref{thm : PointPoint}\eqref{item : PointPoint/General/OrdreGrandeur} implies
	\begin{equation}
		C(\eps) \dpe \min_{u\in \S_2} \FdTpp\p{u, (1-\eps)\TC(u)} >0.
	\end{equation}
	For all $k\ge1$, we define $X_k$ as the set of regular points of $\sigma_k$ (see Definition~\ref{def : Topology/GMT/RegularPoint}), and
	\begin{equation*}
		X_0 \dpe \X \setminus \bigcup_{k\ge 1} \sigma_k.
	\end{equation*}
	By Lemma~\ref{lem : Topology/GMT/AreaFormula},
	\begin{equation}
		\hd\p{\X \setminus\p{ \bigcup_{k\ge 0} X_k } } =0.
	\end{equation}
	To prove Lemma~\ref{lem : FdTmon/Intrinsic/Pgrad}, it is sufficient to show that the subset
	\begin{equation}
		\label{eqn : FdTmon/Intrinsic/Pgrad/DefAeps}
		\begin{split}
		A_\eps &\dpe \set{z\in X_0}{ \exists u \in \S_2 \text{ s.t.\ } (\Pgrad D)_z(u)\le (1-3\eps)\TC(u) }\\
			&\quad \cup \bigcup_{k\ge 1}\set{\sigma_k(t) \in X_k}{ \exists u \in \S_2 \text{ not colinear to $\sigma_k'(t)$, s.t.\ } (\Pgrad D)_z(u)\le (1-3\eps)\TC(u) }
		\end{split}
	\end{equation}
	is $\hd$-negligible. Let $P\in \N^*$, $R>0$ and $z\in A_\eps$. By definition of $(\Pgrad D)_z$ and $A_\eps$, there exists $u_z\in \S_2$ and a Lispchitz path $\gamma_z : \intervalleff0T\rightarrow \X$ such that $\gamma_z(0)=z$, $\gamma_z'(0)=u_z$,
	\begin{equation*}
		\liminf_{t \to 0}\frac{D\p{ \restriction{\gamma_z}{\intervalleff0t} } }{t} \le (1-2\eps)\TC(u_z),
	\end{equation*}
	and for small enough $t$,
	\begin{equation}
		\label{eqn : FdTmon/Intrinsic/Pgrad/BadZ/Distinct}
		\gamma_z\p{\intervalleff0t}\cap \p{\bigcup_{p=1}^P\sigma_p} \subseteq \acc{z}:
	\end{equation}
	indeed, for all $1\le p \le P$, either $z\notin \sigma_p$ and this is clear, or $z$ is a regular point of $\sigma_p$, and $u_z$ doesn't belong to the tangent line of $\sigma_p$ at $z$. In particular, there exists $0< r_z \le \frac R{20}$ such that
	\begin{equation}
		\label{eqn : FdTmon/Intrinsic/Pgrad/BadZ/Fast}
		D\p{\restriction{\gamma}{\intervalleff0{r_z}}} \le (1-\eps)\TC\p{\gamma_z(r_z) - \gamma_z(0)} 
	\end{equation}
	and for all $t\in \intervalleof0{r_z}$,
	\begin{equation}
		\label{eqn : FdTmon/Intrinsic/Pgrad/BadZ/NonDegenerate}
		\frac t2 \le \norme[2]{\gamma_z(t) - \gamma_z(0)} \le 2t.
	\end{equation}
	Up to a reparametrization, $\gamma_z$ may be assumed to be injective. Since $A_\eps$ is a bounded subset of $\R^d$, by Vitali's covering theorem (see e.g.\ Section~15.A in \cite{Jon01}), there exists a countable subset $\acc{z_i}_{i\ge 1}$ of $A_\eps$ such that
	\begin{equation}
		\label{eqn : FdTmon/Intrinsic/Pgrad/Inclusion}
		A_\eps\subseteq \bigcup_{i\ge 1} \clball[2]{z_i, 6r_{z_i}},
	\end{equation}
	and the balls $\clball[2]{z_i, 2r_{z_i}}$ are pairwise disjoint. To simplify the notations we will write $r_i$ and $\gamma_i$ instead of $r_{z_i}$ and $\gamma_{z_i}$. By~\eqref{eqn : FdTmon/Intrinsic/Pgrad/Inclusion} and the definition of $\preHD{R}(A_\eps)$ (see~\eqref{eqn : Intro/Def_preHD}),
	\begin{equation}
		\label{eqn : FdTmon/Intrinsic/Pgrad/UpperBoundHD}
		\preHD{R}(A_\eps) \le 12 \sum_{i\ge 1} r_i.
	\end{equation}

	We now upper bound the right-hand side of~\eqref{eqn : FdTmon/Intrinsic/Pgrad/UpperBoundHD}. First note that for all $i\ge1$, by~\eqref{eqn : FdTmon/Intrinsic/Pgrad/BadZ/NonDegenerate},
	\begin{equation*}
		\gamma_i\p{\intervalleff0{r_i}} \subseteq \clball{z_i, 2r_i}.
	\end{equation*}
	In particular, for all $0<\delta < 1$ the $\gamma_i\p{\intervalleff{\delta r_i}{r_i}}$ are pairwise disjoint. Since they are also disjoint from the $\p{\sigma_p}_{1\le p \le P}$ by~\eqref{eqn : FdTmon/Intrinsic/Pgrad/BadZ/Distinct}, Proposition~\ref{prop : FdTmon/LowerBound} gives
	\begin{equation*}
		\sum_{i\ge 1} \int_{\delta r_i}^{r_i} \FdTpp\p{ \gamma_i'(t), \MD{\gamma_i}(t) }\d t + \sum_{p=1}^P \int_0^L \FdTpp\p{ \sigma_p'(t), \MD{\sigma_p}(t) }\d t \le \FdTmon(D).
	\end{equation*}
	By monotone convergence, letting $\delta \to 0$ gives
	\begin{equation*}
		\sum_{i\ge 1} \int_{0}^{r_i} \FdTpp\p{ \gamma_i'(t), \MD{\gamma_i}(t) }\d t + \sum_{p=1}^P \int_0^L \FdTpp\p{ \sigma_p'(t), \MD{\sigma_p}(t) }\d t \le \FdTmon(D).
	\end{equation*}
	Hence, by Proposition~\ref{prop : FdTmon/UpperBound},
	\begin{equation}
		\label{eqn : FdTmon/Intrinsic/Pgrad/UpperBoundSumRi/ResteFdT}
		\sum_{i\ge 1} \int_0^{r_i} \FdTpp\p{ \gamma_i'(t), \MD{\gamma_i}(t) }\d t \le \sum_{p= P+1}^\infty \int_0^L \FdTpp\p{ \sigma_p'(t), \MD{\sigma_p}(t) }\d t.
	\end{equation}
	Besides, for all $i\ge1$, Jensen's inequality in its multivariate version (see~\cite{Per74}) and Lemma~\ref{lem : Topology/GMT/Rademacher} gives
	\begin{align}
		\int_0^{r_i} \FdTpp\p{ \gamma_i'(t), \MD{\gamma_i}(t) }\d t%
			&\ge \FdTpp\p{\gamma_i(r_i) - \gamma_i(0), D\p{\restriction{\gamma}{\intervalleff0{r_i}}}}.\notag
		\intertext{By homogeneity of $\FdTpp$,~\eqref{eqn : FdTmon/Intrinsic/Pgrad/BadZ/Fast} and~\eqref{eqn : FdTmon/Intrinsic/Pgrad/BadZ/NonDegenerate}, }
		\int_0^{r_i} \FdTpp\p{ \gamma_i'(t), \MD{\gamma_i}(t) }\d t%
			&\ge  \norme[2]{\gamma_i(r_i) - \gamma_i(0)} \FdTpp\p{ \frac{ \gamma_i(r_i) - \gamma_i(0) }{ \norme[2]{\gamma_i(r_i) - \gamma_i(0)} }, \frac{ D\p{\restriction{\gamma}{\intervalleff0{r_i} }}}{\norme[2]{\gamma_i(r_i) - \gamma_i(0)}} } \eol
			&\ge \frac{r_i}2 \min_{u\in \S_2} \FdTpp\p{u, (1-\eps)\TC(u) } = \frac{r_i C(\eps)}{2}.\notag
	\end{align}
	Combining this inequality with~\eqref{eqn : FdTmon/Intrinsic/Pgrad/UpperBoundSumRi/ResteFdT}, we get
	\begin{equation}
		\label{eqn : FdTmon/Intrinsic/Pgrad/UpperBoundSumRi}
		\sum_{i\ge 1} r_i \le \frac{2}{C(\eps)} \sum_{p= P+1}^\infty \int_0^L \FdTpp\p{ \sigma_p'(t), \MD{\sigma_p}(t) }\d t.
	\end{equation}

	By~\eqref{eqn : FdTmon/Intrinsic/Pgrad/UpperBoundHD} and~\eqref{eqn : FdTmon/Intrinsic/Pgrad/UpperBoundSumRi},
	\begin{align}
		\preHD{R}(A_\eps) &\le \frac{24}{C(\eps)} \sum_{p= P+1}^\infty \int_0^L \FdTpp\p{ \sigma_p'(t), \MD{\sigma_p}(t) }\d t.\notag
		\intertext{Letting $R\to 0$ yields}
		\hd(A_\eps) &\le \frac{24}{C(\eps)} \sum_{p= P+1}^\infty \int_0^L \FdTpp\p{ \sigma_p'(t), \MD{\sigma_p}(t) }\d t.\notag
	\end{align}
	Since $\FdTmon(D)<\infty$, the series on the right-hand side converges by~\eqref{eqn : FdTmon/Existence/ExpressionGeodesics}, thus letting $P\to \infty$ gives $\hd(A_\eps) =0$, which concludes the proof.
\end{proof}
\subsection{Strict monotonicity : proof of Proposition~\ref{prop : FdTmon/StrMon} }
\label{subsec : FdTmon/StrMon}
Let $D_1, D_2 \in \RegPD$ be distinct pseudometrics such that $D_1 \le D_2$: in particular there exist $x,y\in\X$ such that $D_1(x,y) < D_2(x,y)$. Also assume that $\FdTmon(D_2)<\infty$. There exists a Lipschitz path $\gamma : \intervalleff0T \rightarrow \X$ such that $D_1(\gamma) < D_2(\gamma)$. In particular, by~\eqref{eqn : Topology/GMT/Gradient/Integral} and~\eqref{eqn : Topology/GMT/AreaFormula} there exists a Borel set $A\subset \X$ such that $\hd(A)>0$, and for all $z\in A$, there exists $u\in\S_2$ such that
\begin{equation*}
	(\Pgrad D_1)_z (u) < (\Pgrad D_2)_z (u).
\end{equation*}
By Lemma~\ref{lem : FdTmon/Intrinsic/Pgrad} and Remark~\ref{rk : PointPoint/StrictDecroissance}, for all $z\in A$,
\begin{equation*}
	\max_{u \in \S_2} \FdTpp\p{ u, (\Pgrad D_1)_z\p{u} } < \max_{u \in \S_2} \FdTpp\p{ u, (\Pgrad D_2)_z\p{u} }.
\end{equation*}
Besides, the analogous weak inequality is true for all $z\in\X$, thus Proposition~\ref{prop : FdTmon/Intrinsic} concludes.\qed


\section{Large deviation principle}
\label{sec : LDP}
The goal of this section is to prove Theorem~\ref{thm : MAIN}. For all $b \in \intervalleff1\infty$ and $e\in \bbE^d$, we define
	\begin{equation}
		\tau_e^{(b)} \dpe \tau_e \wedge  b.
	\end{equation}
We denote by $\SPT^{(b)}$, $\TC^{(b)}$, $\FdT^{(b)}$ and $\FdT^{-, (b)}$ the analogues of $\SPT$, $\TC$, $\FdT$ and $\FdTmon$ for the edge passage times $\p{\tau_e^{(b)}}_{e\in \bbE^d}$. For all $b<\infty$, we will introduce an auxiliary process $(\TightSPT)_{n\ge 1}$ with values in $\RegPD[b \normeUn\cdot]$ (see~\eqref{eqn : LDP/ContinuousMetric/DefTightSPT}), such that:
\begin{enumerate}[(i)]
	\item 
	\label{item : LDP/TightSPT/Tight}
	For all $\alpha>0$, there exists $L>0$ such that
	\begin{equation}
		\label{eqn : LDP/TightSPT/Tight}
		\liminf_{n\to \infty} -\frac1n \log \Pb{ \TightSPT \notin \RegPD[ b\normeUn\cdot]^L} \ge \alpha.	
	\end{equation} 
	\item 
	\label{item : LDP/TightSPT/ExpEquivalent}
	For all $n\ge 1$,
	\begin{equation}
		\label{eqn : LDP/TightSPT/ExpEquivalent}
		\UnifDistance\p{\SPT^{(b)}, \TightSPT} \le \frac{2bd}{n}.
	\end{equation}
\end{enumerate}
Theorem~\ref{thm : Intro/FdTmon} and~\eqref{eqn : LDP/TightSPT/Tight} will imply that $(\TightSPT)_{n\ge1}$ satisfies the LDP, with the rate function $\FdT$. We will deduce the same for $\p{\SPT^{(b)}}_{n\ge 1}$ by~\eqref{eqn : LDP/TightSPT/ExpEquivalent}. We treat the general case by letting $b\to \infty$, under some good moment assumptions. To control the probability that $\SPT$ takes abnormally large values, we rely on Lemma~\ref{lem : OdG}, proven in Section~\ref{appsec : OdG}.
\begin{Lemma}
	\label{lem : OdG}
	Assume~\eqref{ass : Intro/main_thm/exp_moment}. For all $\eps>0$,
	\begin{equation}
		\label{eqn : OdG}
		\lim_{n \to \infty} -\frac1n \sup_{x,y \in \X}\log \Pb{\SPT(x,y) \ge \TC(x-y) + \eps } = +\infty.
	\end{equation}
\end{Lemma}

\subsection{Geodesic lengths}
\label{subsec : LDP/GeoLen}
Lemma~\ref{lem : LDP/GeoLen} essentially states that it is very atypical for the random metric $\BoxPT$ to have geodesics of length $Ln$, with large $L$. It will be the key argument in the proof of~\eqref{eqn : LDP/TightSPT/Tight}. Note that the bound~\eqref{eqn : LDP/GeoLen} is uniform over all truncations of the passage times by $b\in \intervalleff1\infty$. For all $b\in \intervalleff1\infty$, $L>0$ and $n\ge 1$, we define the event
\begin{equation}
	\label{eqn : LDP/GeoLen/Event}
	\LongGeo_n^{(b)}(L) \dpe \acc{ \begin{array}{c} \text{There exists a discrete geodesic $\sigma\subseteq \intint0n^d$ }  \\  \text{for $\BoxPT^{(b)}$, such that }\norme{\sigma}\ge L n  \end{array} }.
\end{equation}
\begin{Lemma}
	\label{lem : LDP/GeoLen}
	For all $\alpha>0$, there exists $L>0$ such that for all $b\in \intervalleff1\infty$,
	\begin{equation}
		\label{eqn : LDP/GeoLen}
		\liminf_{n\to\infty} -\frac1n \log \Pb{ \LongGeo_n^{(b)}(L) } \ge \alpha.
	\end{equation}
\end{Lemma}
\begin{proof}
	For all $L, C>0$ and $n\ge 1$ we define the event
	\begin{equation}
		\cE_n(L, C) \dpe \acc{\begin{array}{c} \text{There exists a self-avoiding discrete path $\pi\subseteq \intint0n^d$ }  \\  \text{such that }\norme{\pi}\ge nL  \text{ and } \tau^{(1)}(\pi) \le nC  \end{array}}.
	\end{equation}
	We claim that there exists $\beta>0$ such that for large enough $L$,
	\begin{equation}
		\label{eqn : LDP/GeoLen/OutputKesten}
		\liminf_{n\to\infty } -\frac1n \log \Pb{ \cE_n(L, d\BoundTC +1 )} \ge \floor L\beta,
	\end{equation}
	By Proposition~5.8 in \cite{KestenStFlour} (and a union bound on the possible starting point of a discrete path in $\intint0n^d$) there exists a constant $\Cl{GEO}>0$ such that
	\begin{equation}
		\beta \dpe \liminf_{n\to\infty}-\frac1n \log\Pb{  \cE_n(1, \Cr{GEO})  } > 0.
	\end{equation}
	Let $L \ge \frac{d\BoundTC +1}{\Cr{GEO}}$ be an integer. We have
	\begin{align*}
		\Pb{ \cE_n(L, d\BoundTC +1)} %
			&\le \Pb{ \cE_n(L, L\Cr{GEO} )}\\
			&\le \Pb{ \cE_{nL}(1, \Cr{GEO} )}.
	\end{align*}
	Consequently,
	\begin{equation*}
		- \frac1n \log \Pb{ \cE_n(L, d\BoundTC +1 )} \ge L \cdot \p{ -\frac1{nL} \log\Pb{\cE_{nL}(1, \Cr{GEO} ) } }.
	\end{equation*}
	Letting $n\to\infty$ gives~\eqref{eqn : LDP/GeoLen/OutputKesten}.
	
	By Lemma~\ref{lem : OdG},
	\begin{equation}
		\label{eqn : LDP/GeoLen/OdG}
		\lim_{n\to\infty}-\frac1n \log \sup_{x,y\in \intint0n^d} \Pb{ \BoxPT(x,y) \ge n(d\BoundTC + 1 ) } = \infty.
	\end{equation}
	By union bound and~\eqref{eqn : LDP/GeoLen/OdG},
	\begin{equation}
		\label{eqn : LDP/GeoLen/OdG_Global}
		\lim_{n\to\infty} - \frac1n \log \Pb{ \exists x,y \in \intint0n^d,\quad  \BoxPT(x,y) \ge n(d\BoundTC + 1 ) } = \infty.
	\end{equation}

	Let $L>0$ be such that~\eqref{eqn : LDP/GeoLen/OutputKesten} holds and $b\in \intervalleff1\infty$. We have the inclusion
	\begin{equation*}
		\LongGeo_n^{(b)}(L) \subseteq \acc{ \exists x,y \in \intint0n^d,\quad  \BoxPT(x,y) \ge n(d\BoundTC + 1 ) } \cup \cE_n(L, d\BoundTC+1),
	\end{equation*}
	therefore by~\eqref{eqn : LDP/GeoLen/OutputKesten} and~\eqref{eqn : LDP/GeoLen/OdG_Global},
	\begin{equation*}
		\liminf_{n\to\infty}-\frac1n \log \Pb{ \LongGeo_n^{(b)}(L) } \ge \floor L \beta,
	\end{equation*}
	which concludes the proof.
\end{proof}

\subsection{The continuous metric}
\label{subsec : LDP/ContinuousMetric}
In this section we fix $b\in \intervallefo1\infty$ and define $\TightSPT$. We prove that it is a good approximation of $\SPT^{(b)}$ and it follows the LDP with the rate function $\FdT^{(b)}$.

\begin{Definition}
	\label{def : LDP/ContinuousMetric}
	Let $n \ge 1$. We extend $\BoxPT^{(b)}$ to $\intervalleff0n^d$ as follows. For all edges $(x^-, x^+), (y^-,y^+) \in \edges{\intint0n^d}$ and all $x\in \intervalleff{x^-}{x^+}$, $y\in \intervalleff{y^-}{y^+}$, we define
	\begin{equation}
		\TightBoxPT(x,y) \dpe \min_{\substack{x' \in \acc{x^-,x^+} \\ y' \in \acc{y^-,y^+}  } } \p{ \norme[1]{x-x'}\tau_{(x^-, x^+)} + \TightBoxPT(x',y') + \norme[1]{y-y'}\tau_{(y^-, y^+)} }.
	\end{equation}
	We then define, for all $x,y\in \intervalleff0n^d$,
	\begin{equation}
		\label{eqn : LDP/ContinuousMetric/DefTightPT}
		\TightBoxPT(x,y) \dpe \p{ b\norme[1]{x-y} } \wedge  \min_{x', y'} \p{ b\norme[1]{x-x'} + \TightBoxPT(x',y') + b\norme[1]{y' - y}  },
	\end{equation}
	where the minimum is taken on all pairs of points, belonging to a pair of (maybe equal) edges in $\edges{\intint0n^d}$. For all $x,y\in \X$, we define
	\begin{equation}
		\label{eqn : LDP/ContinuousMetric/DefTightSPT}
		\TightSPT(x,y) \dpe \frac1n \TightBoxPT(nx,ny).
	\end{equation}
\end{Definition}

\begin{Lemma}
	\label{lem : LDP/ContinuousMetric}
	The process $(\TightSPT)_{n\ge 1}$ takes values in $\RegPD[b \normeUn\cdot]$ and for all $x,y \in \frac1n\intint0n^d$,
	\begin{equation}
		\label{eqn : LDP/ContinuousMetric/Extension}
		\TightSPT(x,y) = \SPT^{(b)}(x,y).
	\end{equation}
	Moreover, for all $\alpha>0$, there exists $L>0$ that does not depend on $b$, such that
	\begin{equation}
		\label{eqn : LDP/ContinuousMetric/Tightness}
		\liminf_{n\to \infty} -\frac1n \log \Pb{ \TightSPT \notin \RegPD[ b\normeUn\cdot]^L} \ge \alpha.	
	\end{equation}
\end{Lemma}
\begin{proof}
	Let $n\ge 1$. It is clear that $\TightSPT$ is a pseudometric on $\X$, is upper bounded by $b\norme[1]\cdot$ and satisfies~\eqref{eqn : LDP/ContinuousMetric/Extension}. Let $x,y\in \intervalleff0n^d$. It follows from the definition of $\TightBoxPT$ that there exists a $\TightBoxPT$-geodesic $\sigma$ from $nx$ to $ny$ that admits the decomposition
	\begin{equation}
		nx\Path{\sigma_1} x' \Path{\sigma_2} x'' \Path{\sigma_3} y'' \Path{\sigma_4} y' \Path{\sigma_5} ny,
	\end{equation}
	where:
	\begin{itemize}
		\item the points $x'$ and $y'$ belong to edges in $\edges{\intint0n^d}$, and the points $x'', y''$ belong to $\intint0n^d$,
		\item the paths $\sigma_1$, $\sigma_2$, $\sigma_4$ and $\sigma_5$ are straight lines,
		\item the path $\sigma_3$ is a discrete geodesic for $\TightBoxPT$ between $x''$ and $y''$
	\end{itemize}
	(some $\sigma_i$ may be trivial).
	Moreover, $\sigma_2$ and $\sigma_4$ are subsets of edges in $\edges{\intint0n^d}$, thus
	\begin{equation}
		\norme[1]{\sigma} \le \norme{\sigma_3} + 2dn + 2. 	
	\end{equation}
	Fix $\alpha>0$ and let $L>0$ be the number provided by Lemma~\ref{eqn : LDP/GeoLen}. Since $t\mapsto \frac1t\sigma(t)$ is a $\TightSPT$-geodesic, we have
	\begin{equation*}
		\liminf_{n\to \infty} -\frac1n \log \Pb{ \TightSPT \notin \RegPD[ b\normeUn\cdot]^{L+2d + 2/n} } \ge \alpha,	
	\end{equation*}
	thus
	\begin{equation*}
		\liminf_{n\to \infty} -\frac1n \log \Pb{ \TightSPT \notin \RegPD[ b\normeUn\cdot]^{L+2d +1} } \ge \alpha,	
	\end{equation*}
	which concludes the proof.
\end{proof}
\begin{Lemma}
	\label{lem : LDP/ContinuousMetric/ExpEquiv}
	For all $n\ge 1$, almost surely
	\begin{equation}
		\label{eqn : LDP/ContinuousMetric/ExpEquiv}
		\UnifDistance\p{\SPT^{(b)}, \TightSPT} \le \frac{2b d}{n}.
	\end{equation}
	In particular, replacing $\SPT^{(b)}$ by $\TightSPT$ in the definition of $\FdTsup^{(b)}$ (see~\eqref{eqn : Intro/def_FdTsup}), $\FdTinf^{(b)}$ (see~\eqref{eqn : Intro/def_FdTinf}) and $\FdT^{-, (b)}$ (see~\eqref{eqn : Intro/main_thm/def_FdTmonsup} and~\eqref{eqn : Intro/main_thm/def_FdTmoninf} defines the same functions. 
\end{Lemma}
\begin{proof}
	The pseudometrics $\SPT^{(b)}$ and $\TightSPT$ coincide on $\p{ \frac 1n\intint0n^d }^2$, thus the almost sure bound follows by triangle inequality.
\end{proof}
\begin{Lemma}
	\label{lem : LDP/ContinuousMetric/LDP}
	The process $(\TightSPT)_{n\ge 1}$ follows the LDP with the good rate function
	\begin{align}
		\FdT^{(b)} : \BoundedF &\longrightarrow \intervalleff0\infty \eol
			D &\longmapsto \begin{cases} \infty \quad &\text{if } D\notin \RegPD[\TC^{(b)}], \\ \FdT^{-, (b)}(D) \quad &\text{if } D\in \RegPD[\TC^{(b)}]. \end{cases}
	\end{align}
\end{Lemma}
\begin{proof}
	Let $D\in \BoundedF$. We claim that $(\TightSPT)_{n\ge 1}$ follows the weak LDP with the rate function $\FdT^{(b)}$, i.e.\
	\begin{equation}
		\label{eqn : LDP/ContinuousMetric/LDP/weakLDP}
		\FdTinf^{(b)}(D) = \FdTsup^{(b)}(D) = \FdT^{(b)}(D).
	\end{equation}
	For all $\eps>0$, we define the events
	\begin{align}
		\widetilde{\LD}_n^{(b)}(D,\eps) &\dpe \acc{\UnifDistance\p{\TightSPT , D}\le \eps }
		\intertext{and}
		\widetilde{\LD}_n^{-, (b)}(D,\eps) &\dpe \acc{\forall x,y\in X, \quad \TightSPT\p{x,y} \le D(x,y) + \eps },
	\end{align}
	i.e.\ the analogues of $\LD_n(D,\eps)$ and $\LD_n^-(D,\eps)$ with $\SPT$ replaced by $\TightSPT$ (see~\eqref{eqn : Intro/def_LDn-} and~\eqref{eqn : Intro/def_LDn}). We treat three cases differently.

	\emph{Case 1:} Assume that $D \notin \RegPD[b \normeUn\cdot]$. Let $L>0$. Since $\RegPD[b \normeUn\cdot]^L$ is compact, there exists $\eps >0$ such that
	\begin{equation*}
		\widetilde\LD_n^{(b)}(D,\eps) \subseteq \acc{ \TightSPT \notin \RegPD[ b\normeUn\cdot]^{L} }.
	\end{equation*}
	In particular,
	\begin{equation*}
		\FdTinf^{(b)}(D) \ge \sup_{L>0} \liminf_{n\to \infty} -\frac1n \log \Pb{ \TightSPT \notin \RegPD[ b\normeUn\cdot]^L}.
	\end{equation*}
	By~\eqref{eqn : LDP/ContinuousMetric/Tightness}, $\FdTinf^{(b)}(D) = \infty$, thus~\eqref{eqn : LDP/ContinuousMetric/LDP/weakLDP}.

	\emph{Case 2:} Assume that $D\in \RegPD[b \normeUn\cdot] \setminus \RegPD[\TC^{(b)}]$. Then there exists $x,y\in \X$ such that $D(x,y) > \TC^{(b)}(x-y)$. There exists $\eps>0$ be such that
	\begin{equation*}
		\widetilde \LD_n(D,\eps) \subseteq \acc{ \TightSPT(x,y) \ge \TC^{(b)}(x-y) + \eps }.
	\end{equation*}
	In particular, by Lemma~\ref{lem : OdG}, $\FdTinf^{(b)}(D) =\infty$, thus~\eqref{eqn : LDP/ContinuousMetric/LDP/weakLDP}.

	\emph{Case 3:} Assume that $D\in\RegPD[\TC^{(b)}]$. The inequality $\FdTinf^{(b)}(D) \ge \FdT^{-,(b)}(D) = \FdT^{(b)}(D)$ is clear. If $\FdT^{-,(b)}(D) =\infty$, then~\eqref{eqn : LDP/ContinuousMetric/LDP/weakLDP} is proven. Otherwise by Lemma~\ref{lem : LDP/ContinuousMetric} there exists $L>0$ such that
	\begin{equation}
		\label{eqn : LDP/ContinuousMetric/LDP/CasDur1}
		\liminf_{n\to \infty} -\frac1n \log \Pb{ \TightSPT \notin \RegPD[ b\normeUn\cdot]^L} > \FdT^{-, (b)}(D) .	
	\end{equation}
	Fix $\eps>0$. Let $\delta \in \intervalleoo0\eps$. Consider the compact sets
	\begin{align*}
		K_\eps &\dpe \set{ D' \in \RegPD[b\normeUn\cdot]^L }{D' \le D,\text{ and there exists $x,y\in \X$ such that } D'(x,y) \le D(x,y)-\eps  }
		\intertext{and}
		K_{\eps, \delta} &\dpe \set{ D' \in \RegPD[b\normeUn\cdot]^L }{D' \le D+\delta,\text{ and there exists $x,y\in \X$ such that } D'(x,y) \le D(x,y)-\eps  }
	\end{align*}
	By Proposition~\ref{prop : FdTmon/StrMon}, compactness and lower semicontinuity of $\FdT^{-,(b)}$,
	\begin{equation}
		\label{eqn : LDP/ContinuousMetric/LDP/CasDur2}
		\min_{D'\in K_\eps} \FdT^{-, (b)}(D') > \FdT^{-,(b)}(D).
	\end{equation}
	Note that by~\eqref{eqn : Intro/sketch/UB_LB/LB},
	\begin{align}
		\liminf_{n\to\infty} -\frac1n\log \Pb{\TightSPT \in K_{\eps, \delta} } %
			&\ge \min_{D'\in K_{\eps, \delta}} \FdTinf^{(b)}(D')\eol
			&\ge \min_{D'\in K_{\eps, \delta}} \FdT^{-,(b)}(D').\nonumber
		\intertext{Using compactness and lower semicontinuity again yields}
		\inclim{\delta \to 0} \liminf_{n\to\infty} -\frac1n\log \Pb{\TightSPT \in K_{\eps, \delta} }
			&\ge \min_{D'\in K_{\eps}} \FdT^{-,(b)}(D').\nonumber
		\intertext{Consequently, by~\eqref{eqn : LDP/ContinuousMetric/LDP/CasDur2}, for small enough $\delta \in \intervalleoo0\eps$,}
		\liminf_{n\to\infty} -\frac1n\log \Pb{\TightSPT \in K_{\eps, \delta} }
			&> \FdT^{-,(b)}(D).\nonumber
	\end{align}
	This implies that for small enough $\delta \in \intervalleoo0\eps$,
	\begin{align*}
		\limsup_{n \to \infty}-\frac1n \log \Pb{\widetilde \LD_n^-(D,\delta)}%
			&= \limsup_{n \to \infty}-\frac1n \log \Pb{\widetilde \LD_n^-(D,\delta) \cap \acc{\TightSPT \notin K_{\eps, \delta} }  }\\
			&\ge \limsup_{n \to \infty}-\frac1n \log \Pb{ \widetilde \LD_n(D,\eps)  }.
	\end{align*}
	Letting $\delta \to 0$ then $\eps \to 0$ leads to $\FdT^{-,(b)}(D) \ge \FdT^{(b)}(D) $, which concludes the proof of~\eqref{eqn : LDP/ContinuousMetric/LDP/weakLDP}.

	Equation~\eqref{eqn : LDP/ContinuousMetric/LDP/weakLDP} and Lemma~\ref{lem : Intro/sketch/UB_LB} imply that $(\TightSPT)_{n\ge 1}$ follows the weak LDP, with the rate function $\FdT^{(b)}$. Furthermore, Proposition~\ref{prop : Topology/RegPD/Compact} and~\eqref{eqn : LDP/ContinuousMetric/Tightness} implies that $(\TightSPT)_{n\ge 1}$ is exponentially tight, meaning that for all $\alpha>0$, there exists a compact $K\subseteq \BoundedF$ such that
	\begin{equation}
		\liminf_{n\to \infty} -\frac1n \log \Pb{\TightSPT \notin K } > \infty.
	\end{equation}
	Consequently, by Lemma~1.2.18 in \cite{LDTA}, the process $(\TightSPT)_{n\ge1}$ follows the LDP with the good rate function $\FdT^{(b)}$.
\end{proof}
\subsection{LDP for the classic passage time}
\label{subsec : LDP/Classic}
We now prove that under the Assumptions~\eqref{ass : Intro/main_thm/SubcriticalAtom} and~\eqref{ass : Intro/main_thm/exp_moment}, $(\SPT)_{n \ge 1}$ follows the LDP with the rate function $\FdT$. For the truncated rescaled metric $\SPT^{(b)}$, it is a consequence of Lemmas~\ref{lem : LDP/ContinuousMetric/ExpEquiv} and~\ref{lem : LDP/ContinuousMetric/LDP}. For the untruncated one, we use a general large deviation theory result providing a LDP for a process whenever it is in some sense the limit of processes which all follow LDPs. Definition~\ref{def : LDP/ExpEquiv} and Theorem~\ref{thm : LDP/ExpEquiv_LDP} are a reformulation of Definitions~4.2.10, 4.2.14 and Theorems~4.2.13, 4.2.16 in \cite{LDTA} for our framework.
\begin{Definition}
	\label{def : LDP/ExpEquiv}
	Let $(\bbX, \d_\bbX)$ be a metric space. All the random variables mentioned in this definition are assumed to be defined on the probability space.
	\begin{enumerate}[(i)]
		\item We say that two processes $(X_n)_{n\ge 1}$ and $(Y_n)_{n\ge 1}$ on $\bbX$ are \emph{exponentially equivalent} if for all $\eps>0$ and $n\ge 1$, the event $\acc{ \d_\bbX\p{X_n, Y_n } \ge \eps  }$ is measurable, and with fixed $\eps>0$, \begin{equation}
			\lim_{n\to \infty} -\frac1n \log\Pb{ \d_\bbX\p{X_n, Y_n } \ge \eps  } = \infty.
		\end{equation}
		\item We say that a family of processes $\p{ (X_n^{(b)})_{n \ge 1} , b\in\intervallefo0\infty }$ on $\bbX$ is an \emph{exponentially good approximation} of $\p{X_n}_{n\ge 1}$ as $b\to \infty$ if for all $\eps>0$, $b\ge 1$ and $n\ge 1$, the event $\acc{ \d_\bbX\p{X_n^{(b)}, X_n } \ge \eps  }$ is measurable, and with fixed $\eps>0$,
		\begin{equation}
			\lim_{b\to \infty} \liminf_{n\to\infty} -\frac1n \log\Pb{ \d_\bbX\p{X_n^{(b)}, X_n } \ge \eps  } = \infty.
		\end{equation}
	\end{enumerate}
\end{Definition}
\begin{Theorem}
	\label{thm : LDP/ExpEquiv_LDP}
	Let $(\bbX, \d_\bbX)$ be a metric space.
	\begin{enumerate}[(i)]
		\item If the processes $(X_n)_{n\ge 1}$ and $(Y_n)_{n\ge 1}$ on $\bbX$ are exponentially equivalent and $(X_n)_{n\ge 1}$ follows a LDP with a good rate function, then $(Y_n)_{n\ge 1}$ also follows a LDP, with the same rate function.
		\item \label{item : LDP/ExpEquiv_LDP/EGA}Assume that the processes $(X_n^{(b)})_{n \ge 1}$ on $\bbX$ are exponentially good approximations of $\p{X_n}_{n\ge 1}$ as $b\to \infty$, and for all $b \ge 1$, $(X_n^{(b)})_{n \ge 1}$ follows the LDP with a rate function $I^{(b)}$. Then $(X_n)_{n\ge 1}$ follows the weak LDP with the rate function 
		\begin{equation}
			I : x \longmapsto \inclim{\eps \to 0} \liminf_{b\to\infty} \inf_{\hat x \in \ball[\bbX]{x, \eps} } I^{(b)}(\hat x). 	
		\end{equation}
		If furthermore $I$ is a good rate function and satisfies, for all closed sets $F$,
		\begin{equation}
			\inf_{x\in F}I(x) \le \limsup_{b\to\infty} \inf_{x\in F} I^{(b)}(x),
		\end{equation}
		then $(X_n)_{n\ge 1}$ follows the LDP with the rate function $I$.
	\end{enumerate}
\end{Theorem}
\begin{Lemma}
	\label{lem : LDP/Classic/EGA}
	Assume~\eqref{ass : Intro/main_thm/SubcriticalAtom} and~\eqref{ass : Intro/main_thm/exp_moment}. The processes $(\SPT^{(b)})_{n\ge 1}$ are exponentially good approximations of $(\SPT)_{n\ge 1}$ as $b\to\infty$.
\end{Lemma}
\begin{proof}
	Fix $\eps>0$, $\alpha >0$. Let $L>0$ be the number given by Lemma~\ref{lem : LDP/GeoLen}. For all $n\ge 1$ and $b\ge 1$, define
	\begin{equation}
		\LDGreedy_n(b) \dpe \acc{ \begin{array}{c} \text{There exists a self-avoiding discrete path $\pi \subseteq \intint0n^d$}, \\ \text{such that $\norme{\pi}\le Ln$ and $\tau(\pi) - \tau^{(b)}(\pi) \ge 2d\eps n$}
		\end{array}}.
	\end{equation}
	Note that for all $n\ge 1$ and $b\ge 1$, we have
	\begin{equation}
		\label{eqn : LDP/Classic/EGA/Reunion}
		\acc{\UnifDistance(\SPT^{(b)}, \SPT ) \ge 2d\eps} \subseteq \LongGeo_n^{(b)}(L) \cup \LDGreedy_n(b).
	\end{equation}
	By union bound and the estimate~\eqref{eqn : LDP/GeoLen}, it is thus sufficient to prove that for large enough $b$,
	\begin{equation}
		\label{eqn : LDP/Classic/EGA/Pb_LDGreedy}
		\liminf_{n\to\infty}-\frac1n\log\Pb{ \LDGreedy_n(b) } \ge \alpha,
	\end{equation}
	Lemma~4.2 in Dembo-Gandolfi-Kesten (2001) \cite{Dem01} provides an answer for a similar problem, in the framework of the so-called \emph{greedy lattice animals}: they give an estimate for the upper-tail of the mass gathered by an animal of size $n$ containing the origin, when the masses are scattered on the vertices of $\Z^d$. An animal there is defined as a finite connected subset of $\Z^d$. In particular, self-avoiding paths are animals. In order to adapt their result to our framework, let us introduce for all $v\in \Z^d$ and $z\in \acc{\pm \base i}_{1\le i \le d}$ the random variable
	\begin{equation}
		M_v(z) \dpe \p{ \tau_{(v, v+z)} - b }^+,
	\end{equation}
	as well as
	\begin{equation}
		M_v \dpe \sum_{z\in \acc{\pm \base i}_{1\le i \le d}} M_v(z).
	\end{equation}
	For all $z$, the variables $\p{M_v(z)}_{v\in \Z^d}$ are i.i.d.\ and for all discrete path $\pi$,
	\begin{equation}
		\label{eqn : LDP/Classic/EGA/LienVerticesEdges}
		\tau(\pi) - \tau^{(b)}(\pi) \le \sum_{v \in \pi} M_v.
	\end{equation}
	By Lemma~4.2 in \cite{Dem01} and~\eqref{ass : Intro/main_thm/exp_moment}, for large enough $b$, for all $z\in \acc{\pm \base i}_{1\le i \le d}$,
	\begin{equation*}
		\liminf_{n\to \infty}-\frac1n \log \Pb{ \begin{array}{c} \text{There exists a self-avoiding discrete path $\pi \subseteq \Z^d$}, \\ \text{such that $0\in \pi$, $\norme{\pi}\le Ln$ and $\sum_{v \in \pi} M_v(z) \ge \eps n$} 			
		\end{array}} \ge \alpha.
	\end{equation*}
	Consequently, the union bound and~\eqref{eqn : LDP/Classic/EGA/LienVerticesEdges} give~\eqref{eqn : LDP/Classic/EGA/Pb_LDGreedy}.
\end{proof}

\begin{proof}[Proof of Theorem~\ref{thm : MAIN}]
	Let $b<\infty$. By~\eqref{eqn : LDP/ContinuousMetric/ExpEquiv} the processes $(\TightSPT)_{n\ge 1}$ and $(\SPT^{(b)})_{n\ge 1}$ are exponentially equivalent. Consequently, by Theorem~\ref{thm : LDP/ExpEquiv_LDP}, $(\SPT^{(b)})_{n\ge 1}$ follows the LDP with the rate function $\FdT^{(b)}$.

	By Theorem~\ref{thm : LDP/ExpEquiv_LDP}\eqref{item : LDP/ExpEquiv_LDP/EGA} and Lemma~\ref{lem : LDP/Classic/EGA}, $\p{\SPT}_{n\ge 1}$ follows the weak LDP with the rate function
	\begin{align}
		\FdT^{(\infty)} : \BoundedF &\longrightarrow \intervalleff0\infty \eol
			D &\longmapsto \sup_{\delta > 0} \liminf_{b\to\infty} \inf_{\hat D \in \ball[\infty]{D,\delta} } \FdT^{(b)}(\hat D),
	\end{align}
	where $\ball[\infty]{D,\delta}$ denotes the open ball of center $D$ and radius $\delta$, for $\UnifDistance$. To prove the LDP with this rate function it is sufficient by the second part of Theorem~\ref{thm : LDP/ExpEquiv_LDP}\eqref{item : LDP/ExpEquiv_LDP/EGA} to show to that $\FdT^{(\infty)}$ is a good rate function, and for all closed sets $F\subseteq \BoundedF$,
	\begin{equation}
		\label{eqn : LDP/Classic/CritereLDP}
		\inf_{D\in F} \FdT^{(\infty)}(D) \le \limsup_{b\to \infty} \inf_{D \in F} \FdT^{(b)}(D).
	\end{equation}
	By~\eqref{eqn : LDP/ContinuousMetric/Tightness}, using the same argument as in the proof of Lemma~\ref{lem : LDP/ContinuousMetric/LDP}, Case 1, one show that for all $\alpha>0$ there exists $L>0$ such that for all $b\in \intervalleff1\infty$,
	\begin{equation}
		\label{eqn : LDP/Classic/Tightness}
		\inf_{D \in \BoundedF \setminus \RegPD^L} \FdT^{(b)}(D) \ge \alpha.
	\end{equation}
	Consequently, $\FdT^{(\infty)}$ is a good rate function and~\eqref{eqn : LDP/Classic/CritereLDP} only needs to be proven for compact $F$. Let $F\subseteq \BoundedF$ be a compact set. For all $b\ge 1$, since $\FdT^{(b)}$ is lower semicontinuous, it admits a minimizer $D_b$ on $F$. By compactness there exists a sequence $(b_k)_{k\ge 1}$ that diverges to $\infty$ and $D\in F$ such that
	\begin{equation*}
		\lim_{k\to \infty} D_{b_k} = D.
	\end{equation*}
	In particular, for all $\delta>0$, for large enough $k$,
	\begin{equation*}
		\UnifDistance\p{D, D_{b_k}} < \delta,
	\end{equation*}
	thus
	\begin{align*}
		\FdT^{(\infty)}(D) %
			&\le \liminf_{k\to \infty} \FdT^{(b_k)}\p{ D_{b_k} } \\
			&=\liminf_{k\to \infty} \inf_{\hat D \in F} \FdT^{(b_k)}\p{ \hat D }\\
			&\le \limsup_{b\to\infty} \inf_{\hat D \in F} \FdT^{(b)}\p{ \hat D },
	\end{align*}
	thus~\eqref{eqn : LDP/Classic/CritereLDP}.

	We now prove that $\FdT^{(\infty)} = \FdT$. If $D\in \BoundedF \setminus \RegPD$, then~\eqref{eqn : LDP/Classic/Tightness} implies $\FdT^{(\infty)}(D) = \infty = \FdT(D)$. Let $D \in \RegPD^L$. To show the inequality $\FdT^{(\infty)}(D) \ge \FdTmon(D)$, notice that by an elementary inclusion, for all $\eps>0$,
	\begin{equation*}
		\limsup_{n\to\infty} -\frac1n  \log \Pb{\LD_n(D,\eps)} \ge \limsup_{n\to\infty} -\frac1n  \log \Pb{\LD_n^-(D,\eps)}.
	\end{equation*}
	Moreover, since $D$ is in the interior of $\ball[\infty]{D,\eps}$ and $\SPT$ follows the LDP with the rate function $\FdT^{(\infty)}$, the left-hand side may be upper bounded by $\FdT^{(\infty)}(D)$. Letting $\eps\to 0$ gives $\FdT^{(\infty)}(D) \ge \FdTmon(D)$.

	Let us prove the converse inequality. If $\FdTmon(D)=\infty$, there is nothing to do. Assume the contrary. In particular there exists $L>0$ such that
	\begin{equation}
		\label{eqn : LDP/Classic/EncoreTightness}
		\inf_{\hat D \in \BoundedF \setminus \RegPD^L} \FdT^{(\infty)}(\hat D) > \FdTmon(D).
	\end{equation}
	For all $\eps \ge 0$ we consider the closed set
	\begin{align}
		F_\eps &\dpe \set{\hat D \in \BoundedF}{\forall x,y\in \X,\; \hat D(x,y) \le D(x,y) + \eps },
	\end{align}
	and denote by $K_\eps$ the compact set $F_\eps \cap \RegPD^L$. Since $\p{\SPT}_{n\ge 1}$ follows the LDP with the rate function $\FdT^{(\infty)}$, for all $\eps>0$,
	\begin{align*}
		\liminf_{n\to\infty}-\frac 1n\log\Pb{ \LD_n^-(D, \eps)}%
			&= \liminf_{n\to\infty}-\frac 1n\log\Pb{ \SPT\in F_\eps }\\
			&\ge \inf_{\hat D \in F_\eps} \FdT^{(\infty)}(\hat D).
		\intertext{By~\eqref{eqn : LDP/Classic/EncoreTightness},}
		\liminf_{n\to\infty}-\frac 1n\log\Pb{ \LD_n^-(D, \eps)}%
			&\ge \min_{\hat D \in K_\eps} \FdT^{(\infty)}(\hat D).
		\intertext{By compactness there exists a sequence $(\eps_k)_{k\ge 1}$ decreasing to $0$ and a converging sequence of metrics $(D_k)_{k\ge 1}$ such that for all $k\ge 1$ , $D_k \in K_{\eps_k}$ and realizes the minimum above. The limit of $(D_k)_{k\ge 1}$ belongs to $\bigcap_{\eps>0} K_\eps = K_0$. Letting $k\to\infty$ gives, by lower semicontinuity of $\FdT^{(\infty)}$ ,}
		\FdTmon(D)%
			&\ge \min_{\hat D \in K_0}\FdT^{(\infty)}(\hat D).
	\end{align*}
	However, for all $\hat D \in K_0\setminus\acc{D}$, by Proposition~\ref{prop : FdTmon/StrMon},
	\begin{equation*}
		\FdT^{(\infty)}(\hat D) \ge \FdTmon(\hat D) >  \FdTmon(D). 
	\end{equation*}
	Consequently, $\FdTmon(D) \ge \FdT^{(\infty)}(D)$, which concludes the proof of $\FdT = \FdT^{(\infty)}$.
\end{proof}


\appendix
\section{Probability that a vertex is a hub}
\label{appsec : Cool}

In this section we prove Lemma~\ref{lem : Intro/main_thm/LemmeCool}. Our arguments are adapted from Cox and Durrett's proof of the shape theorem (see Theorem~3.3 in \cite{ShapeThm}). We will use Lemma~\ref{lem : Cool/DisjointPaths} twice in order to dominate the passage time between two points by $d$ independent sums of variables with distribution $\nu$. The first step is to prove a similar result under a stronger moment condition, when the edge passage times may have short range dependence. Under the assumption $\E{\tau_e^{d+\xi}}<\infty$,~\eqref{eqn : Intro/main_thm/LemmeCool} holds by standard estimates on i.i.d.\ sums (see Lemma~\ref{lem : Cool/GoodCase}). To conclude under~\eqref{ass : Intro/main_thm/ShapeThmStrong}, we need an extra step consisting in bounding the passage time between two neighbors in $3\Z^d$ by the minimum of the passage times along the $d$ paths provided by Lemma~\ref{lem : Cool/DisjointPaths}. At the scale of those "long edges", Lemma~\ref{lem : Cool/GoodCase} can be applied.

\begin{Lemma}
	\label{lem : Cool/DisjointPaths}
	Fix $n \in \N^*$ and distinct vertices $x,y\in \intint0n^d$. There exist at least $d$ pairwise disjoint except at their endpoints discrete paths $\pi_1, \dots, \pi_d$ from $x$ to $y$, with length $\norme[1]{x-y}$ or $\norme[1]{x-y}+2$.
\end{Lemma}
\begin{proof}
	Write $x=(x_1, \dots, x_d)$ and $y=(y_1, \dots, y_d)$. Without loss of generality, we can assume the existence $i_0 \in \intint0{d}$, such that
	\begin{itemize}
		\item For all $i\in \intint1{i_0}$, $x_i = y_i \neq n$.
		\item For all $i\in \intint{i_0+1}{d}$, $x_i < y_i$.
	\end{itemize}
	Let $\p{x= z(0), \dots, z(r)= y}$ be any path from $x$ to $y$ with length $\norme[1]{x-y}$. For all $i\in \intint1{i_0}$, we define $\pi_i$ the path $(x, z(0)+\base i, z(1)+ \base i, \dots , z(r)+ \base i , y)$. For all $i\in \intint{i_0+1}{d}$, we define $\pi_i$ as the only path from $x$ to $y$ which is the concatenation of $d-i_0$ straight lines of respective directions $\p{\base i, \base{i+1},\dots, \base{d},\base{i_0 +1}, \dots, \base{i-1} }$.
\end{proof}

\begin{Lemma}
    \label{lem : Cool/GoodCase}
    Let $(\tau'_e)_{e \in \bbE^d}$ be a family of identically distributed random variables. Assume that there exists $\xi>0$ such that $\E{(\tau_e')^{d+\xi}} < \infty$, and for all $e\in \bbE^d$, $\tau_e'$ is independent of the family passage times along edges having no common endpoint with $e$. Then there exists a constant $\kappa<\infty$ such that 
    \begin{equation}
        \label{eqn : Cool/GoodCase}
        \liminf_{n\to\infty} \min_{x\in \intint0n^d} \Pb{\begin{array}{c} \text{For all $y \in \intint0n^d$, there exists a discrete path $x\Path\pi y$ included} \\ \text{in $\intint0n^d$, such that } \tau'(\pi) \le \kappa \norme[1]{x-y} \text{ and } \norme{\pi} \le \norme[1]{x-y} +2\end{array}} >0.
    \end{equation}
\end{Lemma}
\begin{proof}
	Let $x,y \in \intint0n^d$ be distinct vertices and $\gamma_1, \dots, \gamma_d$ be the paths given by Lemma~\ref{lem : Cool/DisjointPaths}, with removed endpoints. The $\gamma_i$ have no common vertex, hence their passage times are independent. Fix $i\in\intint1d$. We denote by $\gamma_i^{(1)}$ and $\gamma_i^{(2)}$ the set of edges along $\gamma_i$, with odd indices and even indices respectively. By union bound and Markov's inequality,
	\begin{align}
		&\Pb{\tau'(\gamma_i) - \E{\tau'(\gamma_i)} \ge \norme[1]{x-y} }\eol
			&\quad \le \Pb{\tau'(\gamma_i^{(1)}) - \E{\tau'(\gamma_i^{(1)})} \ge \frac12\norme[1]{x-y} } + \Pb{\tau'(\gamma_i^{(2)}) - \E{\tau'(\gamma_i^{(2)})} \ge \frac12\norme[1]{x-y} }\eol
			&\quad \le \frac{2^{d+\xi}}{\norme[1]{x-y}^{d+\xi} }\E{ \module{\tau'(\gamma_i^{(1)}) - \E{\tau'(\gamma_i^{(1)})}}^{d+\xi}   + \module{\tau'(\gamma_i^{(2)}) - \E{\tau'(\gamma_i^{(2)})}}^{d+\xi}  }.\label{eqn : AvantRosenthal}
	\end{align}
	Since $\tau'(\gamma_i^{(1)})$ is sum of independent variables distributed as $\tau'$, by Rosenthal's inequality (see Theorem~3 in~\cite{Ros70}) there exists a constant $\Cl{preROSENTHAL}$, depending only on $d+\xi$, such that
	\begin{equation}
		\begin{split}
		\E{ \module{\tau'(\gamma_i^{(1)}) - \tau'(\gamma_i^{(1)}) }^{d+ \xi} }^{1/(d+\xi)} &\le \Cr{preROSENTHAL} \max \biggl( \p{ \#\gamma_i^{(1)} \E{\module{\tau' - \E{\tau'} }^{d+\xi} } }^{1/(d+\xi)},\\
		&\quad \p{ \#\gamma_i^{(1)}\E{\module{\tau' - \E{\tau'} }^{2} }   }^{1/2}
		\biggr).
		\end{split}
	\end{equation}
	Consequently, there exists a constant $\Cl{ROSENTHAL}$, depending only on $d+\xi$ and the distribution of $\tau'$, such that
	\begin{equation*}
		\E{ \module{\tau'(\gamma_i^{(1)}) - \tau'(\gamma_i^{(1)}) }^{d+ \xi} } \le \Cr{ROSENTHAL}\norme[1]{x-y}^{(d+\xi)/2}.
	\end{equation*}
	The same goes for the paths $\gamma_i^{(2)}$. Applying this bound to~\eqref{eqn : AvantRosenthal} yields
	\begin{align}
		\Pb{\tau'(\gamma_i) - \E{\tau'(\gamma_i)} \ge \norme[1]{x-y} }%
			&\le \frac{ \Cr{ROSENTHAL} 2^{d+\xi+ 1} }{\norme[1]{x-y}^{(d+\xi)/2} }.\notag
		\intertext{Consequently, plugging in the inequality $\norme{\gamma_i}\le \norme[1]{x-y} $ yields }
		\label{eqn : Cool/Goodcase/BoundPath}
		\Pb{\tau'(\gamma_i) \ge \p{\E{\tau'_e} + 1}\norme[1]{x-y} }%
			&\le \frac{ \Cr{ROSENTHAL}2^{d+\xi+ 1} }{\norme[1]{x-y}^{(d+\xi)/2} }.
	\end{align}
	Since the variables $\tau'(\gamma_i)$ for $i\in \intint1d$ are independent,
	\begin{equation}
		\label{eqn : Cool/Goodcase/Bound1}
		\Pb{\bigcap_{i=1}^d \acc{\tau'(\gamma_i) \ge \p{\E{\tau'_e} + 1}\norme[1]{x-y} }}%
		 \le \frac{ \Cr{ROSENTHAL}^d 2^{d(d+\xi+ 1)} }{\norme[1]{x-y}^{(d^2+d\xi)/2} }.
	\end{equation}
	Besides, by union bound and Markov's inequality,
	\begin{equation}
		\label{eqn : Cool/Goodcase/Bound2}
		\Pb{\max_{z \in \acc{\pm \base i}_{1\le i \le d} }\p{ \tau_{(x, x+z)}\vee \tau_{(y,y+z)} }\ge \norme[1]{x-y} } \le \frac{4d\E{(\tau_e')^{d+\xi}} }{\norme[1]{x-y}^{d+\xi} }.
	\end{equation}

	The right-hand sides of~\eqref{eqn : Cool/Goodcase/Bound1} and~\eqref{eqn : Cool/Goodcase/Bound2} are summable over $y\in \Z^d \setminus\acc x$, and only depends on $x-y$. Consequently, there exists an integer $r\ge 1$ and a constant $\Cl{BOREL-CANTELLI}>0$ such that for all $n\ge 1$ and $x\in\intint0n^d$,
	\begin{align*}
		&\sum_{\substack{y\in \intint0n^d \\ \norme[1]{x-y}\ge r } } \biggl[ \Pb{\bigcap_{i=1}^d \acc{\tau'(\gamma_i) \ge \p{\E{\tau'_e} + 1}\norme[1]{x-y} }} \\
		&\quad + \Pb{\max_{z \in \acc{\pm \base i}_{1\le i \le d} }\p{ \tau_{(x, x+z)}\vee \tau_{(y,y+z)} }\ge \norme[1]{x-y} } \biggl] \le 1- \Cr{BOREL-CANTELLI}.
	\end{align*}
	Using the union bound again, we get that for all $n\ge 1$ and $x\in\intint0n^d$,
	\begin{equation*}
		\Pb{\begin{array}{c} \text{For all $y \in \intint0n^d$ such that $\norme[1]{x-y}\ge r$, there exists a discrete path $x\Path\pi y$ included} \\ \text{in $\intint0n^d$, such that } \tau'(\pi) \le (\E{\tau'_e}+3) \norme[1]{x-y} \text{ and } \norme{\pi} \le \norme[1]{x-y} +2\end{array}} \ge \Cr{BOREL-CANTELLI}.
	\end{equation*}
	In particular,~\eqref{eqn : Cool/GoodCase} holds for $\kappa = \E{\tau'_e}+3$. 
\end{proof}
\begin{proof}[Proof of Lemma~\ref{lem : Intro/main_thm/LemmeCool}]
	Assume~\eqref{ass : Intro/main_thm/ShapeThmStrong}. For all $x,y \in \intint0n^d$, we write $x\sim y$ if $x-y \in 3\Z^d$. We will denote by $\cro x$ the equivalence class of $x$ for $\sim$ in $\intint0n^d$. If $x-y \in \acc{\pm 3 \base i}_{1\le i \le d}$, we say that $x$ and $y$ are $3$-neighbours and define $\tau'(x,y)$ as the minimum of the passage times along the $d$ paths between $x$ and $y$ provided by~Lemma~\ref{lem : Cool/DisjointPaths}. Following the proof of Lemma~3.1 in~\cite{ShapeThm}, one shows that for all such $x,y$, $\E{\tau'(x,y)^{d+\xi}}<\infty$. Up to translation by $3v - \floor{3v}$ and rescaling by a factor $3$, for all equivalence classes $\cro v$ of $\sim$ the family $\p{\tau'(x,y)}$, spanning over pairs of $3$-neighbours $(x,y)\in \cro v^2$ satisfies the hypothesis of Lemma~\ref{lem : Cool/GoodCase}. We say that a vertex $x \in \intint0n^d$ is a \emph{pre-hub} if it satisfies the event in~\eqref{eqn : Cool/GoodCase} for the lattice structure on $\cro x$ induced by $3$-neighbours, and the passage times $\tau'$. If $x$ is a pre-hub, then for all $y \in \cro x$, there exists a discrete path $x\Path\pi y$ included in $\intint0n^d$, such that \[ \tau(\pi) \le \kappa \norme[1]{x-y} \text{ and } \norme{\pi} \le \frac53\norme[1]{x-y} + \frac{10}{3} \le 2\norme[1]{x-y} + 4. \] Given $x\in \intint0n^d$, consider the event
	\begin{equation}
		\cE(n,x) \dpe \p{ \bigcap_{x' \in (x + \intervalleff{-2}{2}^d)\cap\intint0n^d}\acc{x' \text{ is a pre-hub}} }%
			\cap\p{  \bigcap_{e \in \edges{(x + \intervalleff{-2}{2}^d)\cap\intint0n^d } } \acc{ \tau_e \le \kappa} }.
	\end{equation}
	It is an intersection of decreasing events, therefore by the FKG inequality and Lemma~\ref{lem : Cool/GoodCase},
	\begin{equation}
		\liminf_{n\to \infty} \min_{x\in \intint0n^d} \Pb{ \cE(n,x) } > 0.
	\end{equation}
	Let $x\in \intint0n^d$. We are left to show that on the event $\cE(n,x)$, $x$ is a hub. Assume that $\cE(n,x)$ occurs and let $y\in \intint0n^d$. There exists $x'\in (x + \intervalleff{-2}{2}^d)\cap\intint0n^d$ such that $x' \sim y$ and $\norme[1]{x-y} = \norme[1]{x-x'} + \norme[1]{x'-y}$. There exists a path $x \Path{\pi_1} x'$ with length $\norme[1]{x-x'}$, such that $\tau(\pi_1) \le \kappa\norme[1]{x-x'}$. Moreover, since $x'$ is a pre-hub, there exists a discrete path $x'\Path{\pi_2} y$ included in $\intint0n^d$, such that $\tau(\pi_2) \le \kappa \norme[1]{x'-y} \text{ and } \norme{\pi} \le 2\norme[1]{x'-y} + 4$. The concatenation of $\pi_1$ and $\pi_2$ has the desired properties. 
\end{proof}


\section{Bound for upper-tail large deviations}
\label{appsec : OdG}
In this section we prove Lemma~\ref{lem : OdG}, which states that under Assumption~\eqref{ass : Intro/main_thm/exp_moment}, the order of the upper-tail large deviation probability for the point-point time is lower than $\exp\p{-\alpha n}$, for all $\alpha>0$. We essentially follow Kesten's proof (\cite{KestenStFlour}, Theorem~5.9) of the special case $x=0$, $y=\base 1$: for large $n$, we consider a large number of pairwise disjoint corridors, with fixed but large width $m$ which are close to the segment $\intervalleff{nx}{ny}$. On the large deviation event $\acc{\SPT(x,y) \ge \TC(x-y) + \eps }$, either the passage times across all corridors or the passage times from $nx$ and $ny$ to the entrances of each are abnormally large. The probability of each scenario is bounded by $\exp\p{-\alpha n}$, with arbitrary $\alpha>0$.
\begin{proof}
	We first fix distinct $x,y\in \mathring \X$ and $\zeta > \TC(x-y)$ then prove the following weaker version of the result:
	\begin{equation}
		\label{eqn : OdG/Weak}
		\lim_{n \to \infty} -\frac1n\log \Pb{\SPT(x,y) \ge \zeta } = +\infty.
	\end{equation}
	For all integers $n,m \ge 1$ and $v \in \Z^d$, we introduce the subset
	\begin{equation}
		\Cylinder(v,m,n) \dpe \set{ z + s(y-x) }{z\in \clball[2]{v,m},\; s\in \intervalleff0{n}  }.
	\end{equation}
	Let $\eps >0$ be such that $\zeta - \TC(x-y) \ge 4\eps$. By adapting the argument at the bottom of p.\ 198 in \cite{KestenStFlour}, one shows the existence of $m\ge1$ such that for all $n\ge m$,
	\begin{equation*}
		\frac{ \E{\RestPT{\Cylinder(0,m,n)}\p{ 0 , \floor{n(x-y)} } } }{n} \le \TC(x-y) + \eps.
	\end{equation*}
	In particular, by Talagrand's inequality (see e.g.\ \cite{50yFPP}, Theorem 3.13),
	\begin{equation}
		\label{eqn : OdG/Talagrand}
		\alpha_m \dpe \liminf_{n\to \infty} -\frac1{n} \log \Pb{ \RestPT{\Cylinder(0,m,n)}\p{ 0 , \floor{n(x-y)} } \ge  n\p{ \TC(x-y) + 2\eps } } > 0.
	\end{equation}
	Let us fix an intger $r$. For large enough $n$, there exist $v_1, \dots, v_r \in \intint0n^d$, such that:
	\begin{itemize}
		\item The sets $\Cylinder\p{ v_i,m,n }$, for $i\in \intint1r$ are pairwise disjoint and included in $\intint0n^d$.
		\item For all $i\in \intint1r$,
		\begin{equation}
			\label{eqn : OdG/Taille_raccords}
			\norme[1]{\vphantom{\big\vert} \floor{n x} - v_i} \le \sqrt n \text{ and } \norme[1]{\vphantom{\big\vert} \floor{ny} - \p{ v_i + \floor{n(x-y)} } } \le \sqrt n.
		\end{equation}
	\end{itemize}
	By triangle inequality, we have the inclusion
	\begin{equation}
		\label{eqn : OdG/BigInclusion}
		\begin{split}
			\acc{ \SPT\p{ x,y } \ge \zeta }%
				&\subseteq \p{\bigcup_{i=1}^r \acc{ \BoxPT\p{\floor{nx} , v_i} > \eps n } } \\
				&\quad \cup \p{\bigcup_{i=1}^r \acc{ \BoxPT\p{\floor{ny} , \p{ v_i + \floor{n(x-y)} } }  > \eps n } } \\
				&\quad \cup \p{\bigcap_{i=1}^r \acc{ \BoxPT[\Cylinder(v_i,m,n)]\p{ v_i , v_i + \floor{n(x-y)} } \ge  \p{ \TC(x-y) + 2\eps }n } }.
		\end{split}
	\end{equation}
	Let $i\in \intint1r$. We claim that
	\begin{align}
		\label{eqn : OdG/Pb_Mauvais_Raccord1}
		\lim_{n\to\infty}-\frac1{n} \log \Pb{ \BoxPT[n]\p{\floor{nx} , v_i} > \eps n } &= \infty.
		\intertext{and}
		\label{eqn : OdG/Pb_Mauvais_Raccord2}
		\lim_{n\to\infty}-\frac1{n} \log \Pb{ \BoxPT[n]\p{\floor{ny} , \p{ v_i + \floor{n(x-y)} } }  > \eps n } &= \infty.
	\end{align}
	Indeed for all $\lambda>0$, using Chernoff's inequality to bound the passage time along an oriented path from $\floor{nx}$ and $v_i$ and applying~\eqref{eqn : OdG/Taille_raccords} provides
	\begin{equation*}
		\Pb{ \BoxPT\p{\floor{nx} , v_i} > \eps n } \le \exp\p{-\lambda \eps n} \E{\exp(\lambda \tau_e)}^{\sqrt n},
	\end{equation*}
	therefore
	\begin{equation*}
		\liminf_{n\to\infty}-\frac1{n} \log \Pb{ \BoxPT\p{\floor{nx} , v_i} > \eps n }%
			\ge \lambda \eps.
	\end{equation*}
	By~\eqref{ass : Intro/main_thm/exp_moment} this bound holds for all $\lambda>0$, therefore we get~\eqref{eqn : OdG/Pb_Mauvais_Raccord1}. Equation~\eqref{eqn : OdG/Pb_Mauvais_Raccord2} is proven similarly. Besides, by independence and stationarity,
	\begin{equation}
		\label{eqn : OdG/Pb_Corridor}
		\liminf_{n\to \infty}-\frac1n \log \Pb{\bigcap_{i=1}^r \acc{ \RestPT{ \Cylinder(v_i,m,n)}\p{ v_i , v_i + \floor{n(x-y)} } \ge  \p{ \TC(x-y) + 2\eps }n } } \ge \alpha_m r.
	\end{equation}
	Using the union bound on~\eqref{eqn : OdG/BigInclusion} and applying Equations~\eqref{eqn : OdG/Pb_Mauvais_Raccord1},~\eqref{eqn : OdG/Pb_Mauvais_Raccord2},~\eqref{eqn : OdG/Pb_Corridor} give
	\begin{equation*}
		\liminf_{n\to\infty}-\frac1n \log \Pb{ \SPT\p{x,y} \ge \zeta} \ge \alpha_m r.
	\end{equation*}
	Letting $r\to\infty$ yields~\eqref{eqn : OdG/Weak}.

	We now turn to the proof of~\eqref{eqn : OdG}. Fix $\eps>0$ and $\alpha >0$. We define
	\begin{equation}
		\label{eqn : OdG/Uniform/ChoixLambdaDelta}
		\lambda \dpe \frac{\alpha +1 }{\eps} \text{ and } \delta \dpe \frac{\lambda \eps - \alpha}{2 \log \E{\exp\p{\lambda \tau} } }\wedge \frac12. 	
	\end{equation} Let $(x_i)_{1\le i \le r}$ be a finite family of points in $\mathring \X$ such that $\X \subseteq \bigcup_{i=1}^r \clball[1]{x_i, \delta}$. Let $x,y\in\X$. There exists $i,j\in\intint1r$ such that for all $n \ge 1/\delta$,
	\begin{equation}
		\label{eqn : OdG/Uniform/Net}
		\max \p{ \norme[1]{\vphantom{\big \vert} \floor{nx} - \floor{n x_i}}, \norme[1]{\vphantom{\big \vert} \floor{ny} - \floor{n x_j}} } \le 2\delta n .
	\end{equation}
	By triangle inequality,
	\begin{equation}
		\label{eqn : OdG/Uniform/BigReunion}
		\begin{split}
		\acc{\SPT(x,y)\ge \TC(x-y)  + \p{3+2\BoundTC }\eps }%
			&\subseteq \acc{\SPT(x,x_i)> \eps} \cup \acc{\SPT(y,x_j)>  \eps} \\
			&\quad \cup \acc{\SPT(x_i,x_j) \ge \TC(x_i-x_j) + \eps }.
		\end{split}
	\end{equation}
	Using Chernoff's bound and applying~\eqref{eqn : OdG/Uniform/Net} gives
	\begin{equation*}
		\Pb{ \BoxPT[n]\p{nx , n x_i} > \eps n} \le \exp\p{ -\eps \lambda n}\E{\exp(\lambda \tau_e)}^{2\delta n},
	\end{equation*}
	thus, by definition of $\lambda$ and $\delta$,
	\begin{align*}
		-\frac1{n} \log \Pb{ \BoxPT[n]\p{nx , n x_i} > \eps n} &\ge \alpha. \intertext{Similarly, }%
		-\frac1{n} \log \Pb{ \BoxPT[n]\p{ny , n x_j }   > \eps n } &\ge \alpha.
	\end{align*}
	Consequently, applying the union bound to~\eqref{eqn : OdG/Uniform/BigReunion} and plugging in these two inequalities, we get
	\begin{equation}
	\begin{split}
		&-\frac1n \log \Pb{\SPT(x,y)\ge \TC(x-y)  + \p{3+2\BoundTC }\eps } \\
		&\quad\ge \alpha \wedge \min_{1\le i,j, \le r} \cro{ \frac1n\log\Pb{ \SPT(x_i, x_j) \ge \TC(x_i - x_j) + \eps } } - \frac{\log 2}{n} .
	\end{split}
	\end{equation}
	Letting $n\to \infty$ and using~\eqref{eqn : OdG/Weak} gives
	\begin{equation}
		\liminf_{n\to\infty}-\frac1n \log \Pb{\SPT(x,y)\ge \TC(x-y)  + \p{3+2\BoundTC }\eps } \ge \alpha.
	\end{equation}
	That bound holds for all $\alpha>0$, hence~\eqref{eqn : OdG}.
\end{proof}

\newpage
\bibliographystyle{plainurl}
\bibliography{biblio}
\end{document}